\documentclass[11pt,a4paper]{article}
\pdfoutput=1 

\usepackage{epsf,epsfig,amsfonts,amsgen,amsmath,amstext,amsbsy,amsopn,amsthm,cases,listings,color
}
\usepackage{ebezier,eepic}
\usepackage{color}
\usepackage{multirow}
\usepackage{epstopdf}
\usepackage{graphicx}
\usepackage{pgf,tikz}
\usepackage{mathrsfs}
\usepackage[marginal]{footmisc}
\usepackage{enumitem}
\usepackage[titletoc]{appendix}
\usepackage{booktabs}
\usepackage{url}
\usepackage{mathtools}

\usepackage{pgfplots}
\usepackage{authblk}
\usepackage{amssymb}
\usepackage{enumitem}
\usepackage{wasysym}

\usepackage{empheq}

\usepackage{dsfont}

\usepackage{caption}
\usepackage{subcaption}
\pgfplotsset{compat=1.18}
\usepackage{mathrsfs}
\usepackage{wasysym} 
\usetikzlibrary{arrows}
\usepackage{aligned-overset}
\usepackage{bm}
\usepackage[backref=page]{hyperref}
\usepackage{makecell}

\usepackage{float}

\allowdisplaybreaks[1]

\definecolor{uuuuuu}{rgb}{0.27,0.27,0.27}
\definecolor{sqsqsq}{rgb}{0.1255,0.1255,0.1255}

\newtheorem{definition}{Definition} [section]
\newtheorem{theorem}[definition]{Theorem}
\newtheorem{lemma}[definition]{Lemma}
\newtheorem{proposition}[definition]{Proposition}

\newtheorem{corollary}[definition]{Corollary}
\newtheorem{conjecture}[definition]{Conjecture}
\newtheorem{claim}[definition]{Claim}
\newtheorem{problem}[definition]{Problem}

\newtheorem{fact}[definition]{Fact}

\newenvironment{pf}{\noindent{\bf Proof}}

\begin{document}
\title{\bf\Large Tiling $H$ in dense graphs}
\date{\today}
\author[1]{Nannan Chen\thanks{Email: \texttt{chennannan@mail.sdu.edu.cn}}}
\author[2]{Xizhi Liu\thanks{Research supported by ERC Advanced Grant 101020255. Email: \texttt{xizhi.liu.ac@gmail.com}}}
\author[3]{Lin Sun\thanks{
Email: \texttt{sun.lin@qdu.edu.cn}}}
\author[1]{Guanghui Wang\thanks{Research supported by  the Natural Science Foundation of China (12231018). Email: \texttt{ghwang@sdu.edu.cn}}} 
\affil[1]{School of Mathematics, Shandong University, Jinan, China}
\affil[2]{Mathematics Institute and DIMAP,
            University of Warwick,
            Coventry, CV4 7AL, UK}
\affil[3]{School of Mathematics and Statistics, Qingdao University, Qingdao, China}
\maketitle
\begin{abstract}
    We determine asymptotically the two extremal constructions for the tiling problem of the $H$-shaped tree.  
    In particular, the first extremal construction is close to the complement of two cliques, in contrast to previously studied bipartite graphs, where the first extremal construction is close to the complement of a single clique.
    This result refutes one of Lang's conjectures, which seeks to generalize the Erd\H{o}s Matching Conjecture.
    
 \medskip

\end{abstract}
\section{Introduction}\label{SEC:Intorduction}
Fix an integer $r\ge 2$, an $r$-graph $\mathcal{G}$ is a collection of $r$-subsets of some finite set $V$. We identify a hypergraph $\mathcal{G}$ with its edge set and use $V(\mathcal{G})$ to denote its vertex set. The size of $V(\mathcal{G})$ is denoted by $v(\mathcal{G})$. 
Given a vertex $v\in V(\mathcal{G})$, 
the \textbf{degree} $d_{\mathcal{G}}(v)$ of $v$ in $\mathcal{G}$ is the number of edges in $\mathcal{G}$ containing $v$.
We use $\delta(\mathcal{G})$, $\Delta(\mathcal{G})$, and $d(\mathcal{G})$ to denote the \textbf{minimum degree}, the \textbf{maximum degree}, and the \textbf{average degree} of $\mathcal{G}$, respectively.
We will omit the subscript $\mathcal{G}$ if it is clear from the context.

Given two $r$-graphs $F$ and $\mathcal{G}$, an \textbf{$F$-tiling} in $\mathcal{G}$ is a collection of vertex-disjoint copies of $F$ in $\mathcal{G}$. The size of an $F$-tiling is the number of copies of $F$ in it. 
The \textbf{$F$-matching number} $\nu(F, \mathcal{G})$ of $\mathcal{G}$ is defined as the maximum size of an \textbf{$F$-tiling} in $\mathcal{G}$. 
This concept generalizes the well-studied matching number $\nu(\mathcal{G})$ of $\mathcal{G}$, since $\nu(\mathcal{G}) = \nu(K_r^r, \mathcal{G})$. Here,  $K_{\ell}^r$ denotes the complete $r$-graph with $\ell$ vertices. 
We say $\mathcal{G}$ is \textbf{$F$-free} if $\nu(\mathcal{G}) = 0$. 

Numerous works have been dedicated to studying the minimum degree threshold that guarantees an $n$-vertex $r$-graph $\mathcal{G}$ contains an $F$-tiling of a given size, starting from the foundational Corr\'{a}di--Hajnal Theorem~\cite{CH63} and Hajnal--Szemer\'{e}di Theorem~\cite{HS70} (see surveys~\cite{KS96,KSSS02,KO09Survery}). These questions are now quite well understood for graphs (see e.g.~\cite{KO09}) thanks to decades of effort of numerous researchers. 
In contrast, the natural problem of determining the average degree threshold that guarantees an $n$-vertex $r$-graph $\mathcal{G}$ contains an $F$-tiling of a given size remains largely unexplored. 

For convenience, given an $r$-graph $F$, an integer $n \in \mathbb{N}$, and a real number $k \in \left[0, n/v(F)\right]$, let $\mathrm{ex}(n,(k+1)F)$ denote the maximum number of edges in an $n$-vertex $r$-graph $\mathcal{G}$ with $\nu(F,\mathcal{G}) < k+1$. 
The case $k = 0$ is known as the Tur\'{a}n problem of $F$, which is a central topic in Extremal Combinatorics (see surveys~\cite{Keevash11,FS13}), originating from the foundational work  Tur\'{a}n~\cite{TU41} (and even earlier, by Mantel~\cite{Mantel07}). 

Determining (even asymptotically) the behavior of $\mathrm{ex}(n,(k+1)F)$ as $k$ ranges from $0$ to $n/v(F)$ is a very challenging problem in general, with only a few results known in this direction. 
The simplest case, $F = K_2$, was completely solved by the celebrated Erd\H{o}s--Gallai Theorem~\cite{EG59}. 
However, the extension of the Erd\H{o}s--Gallai Theorem becomes quite nontrivial even for very simple structured graphs, such as paths (see e.g.~\cite{Gor11,BK11,YZ17,CL18,YZ21}) and stars (see e.g.~\cite{LLP13}). For general trees, the asymptotic behavior of $\mathrm{ex}(n,(k+1)F)$ remains largely unknown. 

For complete graphs, Erd\H{o}s~\cite{Erdos62} determined $\mathrm{ex}\left(n, (k+1)K_3\right)$ for $k \le \frac{\sqrt{n}}{20}$. 
Moon~\cite{Moon68} improved and extended the theorem of Erd\H{o}s to general $r \ge 4$ and determined $\mathrm{ex}\left(n, (k+1)K_{r}\right)$ for $k \le \frac{2 n}{r^3-r^2+1}+o(n)$ (for $r=3$, his bound is $k \le \frac{2n}{9} + o(n)$).
Akiyama--Frankl~\cite{AF85} determined $\mathrm{ex}\left(n, (k+1)K_{r}\right)$ for the case when $k = \lfloor n/r \rfloor - 1$ (see also~\cite{BE78}).  
A complete determination of $\mathrm{ex}\left(n, (k+1)K_3\right)$ was achieved only a decade ago by Allen--B\"{o}ttcher--Hladk\'{y}--Piguet~\cite{ABHP15} for large $n$. 
An asymptotic complete determination of $\mathrm{ex}\left(n, (k+1)K_4\right)$ was obtained very recently by Hou et al.~\cite{HHLY25}. 
Extending the work of Erd\H{o}s and Moon to general $r$-graphs $F$, tight bounds for $\mathrm{ex}\left(n, (k+1)F\right)$ were recently established by Hou et al. in a series of works~\cite{HLLYZ23,HHLLYZ23a,HHLLYZ23b,HHLLYZ23c} for the case where $k$ is much smaller than $\frac{\mathrm{ex}(n,F)}{n^{r-1}}$. 

We focus on the case where $F$ is an $r$-partite $r$-graph. Given integers $s_1, \ldots, s_r \ge 1$, let $K_{s_1, \ldots, s_r}^r$ denote the complete $r$-partite $r$-graph with parts of sizes $s_1, \ldots, s_r$. 
A natural lower bound for $\mathrm{ex}(n, \beta n \cdot K_{s_1, \ldots, s_r}^r)$ is provided by the following construction. 

\textbf{Construction $G_{n,i,\beta}(s_1, \ldots, s_r)$}:
Given integers $n \ge r\ge i \ge 1$, $1 \le s_1 \le \cdots \le s_r$, and a real number $\beta \in \left(0, \frac{1}{m}\right)$, where $m \coloneqq s_1 + \cdots + s_r$, let $G_{n,i,\beta}(s_1, \ldots, s_r)$ denote the $n$-vertex $r$-graph whose vertex set $V$ has a partition $V = V_1 \cup V_2$ with $|V_1| = \lfloor \beta (s_1 + \cdots + s_i) n \rfloor - 1$, such that the edge set of $G_{n,i,\beta}(s_1, \ldots, s_r)$ satisfies 
\begin{align*}
    G_{n,i,\beta}(s_1, \ldots, s_r)
    = \left\{E \in \binom{V}{r} \colon |E\cap V_1| \ge i\right\}. 
\end{align*}

It is easy to observe that $\nu(K_{s_1, \ldots, s_r}^{r}, G_{n,i,\beta}(s_1, \ldots, s_r)) < \beta n$ for every $i \in [r]$, and thus 
\begin{align}\label{equ:trivial-lower-bound}
    \mathrm{ex}(n,\beta n \cdot K_{s_1, \ldots, s_r}^{r}) 
    \ge \max\left\{|G_{n,i,\beta}(s_1, \ldots, s_r)| \colon i \in [r]\right\}.
\end{align}
For the case $(s_1, \ldots, s_r) = (1,\ldots, 1)$, straightforward calculations show that the maximum value of $|G_{n,i,\beta}(1,\ldots, 1)|$ is achieved when either $i = 1$ or $i=r$. 
The Erd\H{o}s--Gallai Theorem~\cite{EG59} mentioned above states that the inequality in~\eqref{equ:trivial-lower-bound} holds with equality when $r=2$ and $(s_1, s_2) = (1,1)$. 
Extending the Erd\H{o}s--Gallai Theorem to $r \ge 3$ is a major open problem in Extremal Combinatorics. 
The famous Matching Conjecture by Erd\H{o}s~\cite{Erdos65} states that, for $(s_1, \ldots, s_r) = (1,\ldots, 1)$ and $\beta \in [0, 1/r]$, the inequality in~\eqref{equ:trivial-lower-bound} also holds with equality, i.e., for $n \ge r(k+1) - 1$, 
\begin{align*}
    \mathrm{ex}(n, (k+1)\cdot K_{1, \ldots, 1}^{r})
    = \max\left\{\binom{n}{r}-\binom{n-k}{r},~\binom{r(k+1)-1}{r}\right\}. 
\end{align*}
The Erd\H{o}s  Matching Conjecture for the case $r=3$ was solved by {\L}uczak--Mieczkowska~\cite{LM14} for large $n$, and by Frankl~\cite{Frankl17A} for all $n$.
However, the case $r \ge 4$ remains open in general (see e.g.~\cite{HLS12,Frankl13,LM14,Frankl17A,Frankl17B} for some recent progress on this topic).

Going one step beyond the Erd\H{o}s Matching Conjecture, Lang~\cite{Lang23} proposed the following conjecture for general $r$-partite $r$-graphs. 
\begin{conjecture}[Lang~{\cite[Conjecture~10.1]{Lang23}}]\label{CONJ:Lang}
    Let $r\ge 2$ and $1 \le s_1 \le \cdots \le s_r$ be integers. Let $m \coloneqq s_1 + \cdots + s_r$. Suppose that $F$ is a spanning subgraph of $K_{s_1, \ldots, s_r}^{r}$. Then for every real number $\beta \in \left(0, \frac{1}{m}\right)$,
    \begin{align*}
        \mathrm{ex}(n,\beta n \cdot F) 
        = \max\left\{|G_{n,i,\beta}(s_1, \ldots, s_r)| \colon i \in [r]\right\} + o(n^r).
    \end{align*}
\end{conjecture}

The conjecture for the case where $F$ is a complete bipartite graphs follows from the following theorem by Grosu--Hladk{\'y}~\cite{GH12}. 

\begin{theorem}[Grosu--Hladk{\'y}~\cite{GH12}]\label{THM:GH12}
    Let $t \ge s \ge 1$ be integers. Suppose that $F$ is a spanning subgraph of $K_{s,t}$. Then, for every real number $\beta \in \left(0,\frac{1}{s+t}\right)$, 
    \begin{align*}
        \mathrm{ex}(n,\beta n \cdot F)
        \le \left(\max\left\{\frac{s\beta\left(2-s\beta\right)}{2},~\frac{(s+t)^2\beta^2}{2}\right\} + o(1)\right)n^2. 
    \end{align*}
\end{theorem}

Results by Hou et al.~\cite{HHLLYZ23a,HHLLYZ23b,HHLLYZ23c} show that Conjecture~\ref{CONJ:Lang} holds for certain class (see Section~\ref{SEC:Remark}) of $r$-partite $r$-graphs $F$ when $k$ lies in intervals 
\begin{align*}
    \left[0,~\frac{\varepsilon \cdot \mathrm{ex}(n,F)}{n^{r-1}}\right],\quad 
    \left[\frac{\varepsilon \cdot \mathrm{ex}(n,F)}{n^{r-1}},~\varepsilon n\right], \quad\text{and}\quad 
    \left[\frac{n}{v(F)} - \varepsilon n,~\frac{n}{v(F)}\right], 
\end{align*}
where $\varepsilon>0$ is a constant depending only on $F$. 
A conjecture similar to Conjecture~\ref{CONJ:Lang} was proposed earlier by Gan et al.~\cite{GHSW21} for the $r$-partite $r$-graph  $Y_{r,b}$, where $Y_{r,b}$ is the $r$-graph consisting of two edges that intersect in exactly $b < r$ vertices.
The case $(r,b) = (3,2)$ was recently confirmed by Han--Sun--Wang~\cite{HSW23}.  

Our main result (Theorem~\ref{THM:Main-H-tiling}) is an asymptotic determination of  $\mathrm{ex}(n,\beta n \cdot H)$, where $H$ denotes the $H$-shaped tree defined below. In particular, our result shows that the extremal construction for $\beta$ in the interval $\left[0, \frac{1}{9}\right]$ differs significantly from the constructions provided in Conjecture~\ref{CONJ:Lang}, thereby disproving this conjecture. 

\begin{figure}[H]
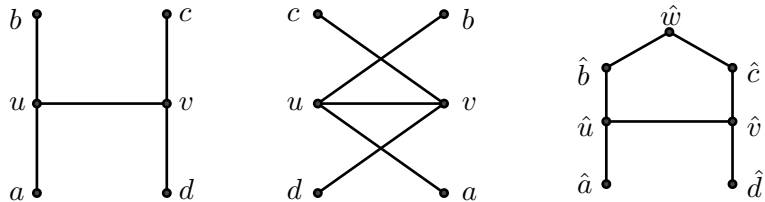

\centering
\tikzset{every picture/.style={line width=1pt}} 

\caption{$H$ and $\hat{H}$.} 
\label{Fig:H-def}
\end{figure}

Let $H$ denote the graph with vertex set $\{u,v,a,b,c,d\}$ and edge set (see Figure~\ref{Fig:H-def})
\begin{align*}
    \{uv, ua, ub, vc, vd\}. 
\end{align*}
Let $\hat{H}$ denote the graph with vertex set $\{\hat{u},\hat{v},\hat{w},\hat{a},\hat{b},\hat{c},\hat{d}\}$ and edge set (see Figure~\ref{Fig:H-def})
\begin{align*}
    \{\hat{u}\hat{v}, \hat{u}\hat{a}, \hat{u}\hat{b},\hat{v}\hat{c},\hat{v}\hat{d},\hat{w}\hat{b}, \hat{w}\hat{c}\}. 
\end{align*}

\begin{figure}[H]
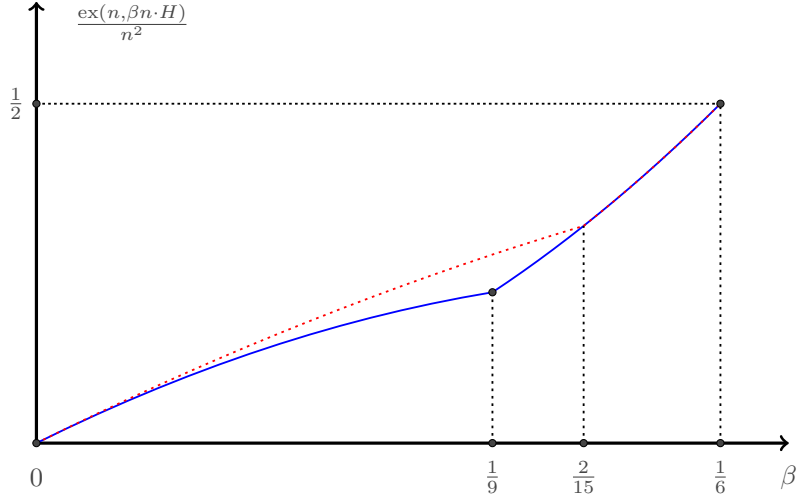

\centering

\caption{The asymptotic behavior of $\frac{\mathrm{ex}(n,\beta n \cdot H)}{n^2}$ as a function of $\beta$. The red dotted curve represents the conjectured value.}
\label{Fig:functions}
\end{figure}

\begin{theorem}\label{THM:Main-H-tiling}
    Let $n \ge 0$ be an integer and $\beta \in \left[0, \frac{1}{6}\right]$ be a real number. 
    Then 
    \begin{align*}
        \mathrm{ex}(n,\beta n \cdot H) 
        =  
        \left(\Xi(\beta) + o(1) \right) n^2,
    \end{align*}
    where 
    \begin{align*}
        \Xi(\beta)
        \coloneqq 
        \begin{cases}
            3\beta(1-3\beta), &\quad\text{if}\quad \beta \in \left[0, \frac{1}{9}\right], \\[0.5em]
            18\beta^2, &\quad\text{if}\quad \beta \in \left[\frac{1}{9}, \frac{1}{6}\right]. 
        \end{cases}
    \end{align*}
\end{theorem}
\textbf{Remarks.} 
\begin{itemize}
    \item Since $H \subseteq K_{3,3}$, it follows from the theorem of Grosu--Hladk\'y (Theorem~\ref{THM:GH12}) that $\mathrm{ex}(n,\beta n \cdot H) \le (18\beta^2+o(1))n^2$ for $\beta \in \left[\frac{2}{15}, \frac{1}{6}\right]$ (see Figure~\ref{Fig:functions}).
    \item The lower bound for $\mathrm{ex}(n,\beta n\cdot H)$ when $\beta \in \left[\frac{1}{9}, \frac{1}{6}\right]$ follows from~\eqref{equ:trivial-lower-bound}. The lower bound for the case $\beta \in \left[0, \frac{1}{9}\right]$ is provided by the complete bipartite graph with parts of sizes $\left\lfloor3\beta n\right\rfloor - 1$ and $n - \left\lfloor3\beta n\right\rfloor +1$. 
\end{itemize}

In Section~\ref{SEC:prelim}, we present some definitions and preliminary results. 
In Section~\ref{SEC:Proof-H-tiling}, we present the proof of Theorem~\ref{THM:Main-H-tiling}.
The proofs of technical results are deferred to Sections~\ref{SEC:Proof-K2-H-hatH-tiling} and~\ref{SEC:proof-LEMMA-Hi-Hj-U-bar-edges}. 
In Section~\ref{SEC:Remark}, we include some remarks and open problems.

\section{Preliminaries}\label{SEC:prelim}
For a graph $F$, the \textbf{$F$-covering ratio} $\tau(F, G)$ of a graph $G$ is defined as 
\begin{align*}
    \tau(F, G)
    \coloneqq \frac{v(F) \cdot \nu(F,G)}{v(G)}. 
\end{align*}
In other words, $\tau(F, G)$ is the proportion of vertices that are covered by a maximum $F$-tiling in $G$.

Given a family $\mathcal{F}$ of graphs, an $\mathcal{F}$-tiling in $G$ is a collection of vertex-disjoint copies of graphs from $\mathcal{F}$ within $G$.

Let $G$ be a graph. 
For every vertex set $S \subseteq V(G)$, let $G[S]$ denote the \textbf{induced subgraph} of $G$ on $S$. The number of edges in $G[S]$ is denoted by $e_{G}(S)$ for simplicity. 
We use $G - S$ to denote the induced subgraph of $G$ on $V(G) \setminus S$. 
For every vertex $v \in V(G)$, we set $N_{G}(v, S) \coloneqq N_{G}(v) \cap S$ and $d_{G}(v, S) \coloneqq |N_{G}(v,S)|$. 
Given another vertex $T \subseteq V(G)$ that is disjoint from $S$, let $G[S,T]$ denote the \textbf{induced bipartite subgraph} of $G$ on $S$ and $T$, i.e. $G[S,T]$ consists of all edges in $G$ that have nonempty intersection with both $S$ and $T$. The number of edges in $G[S,T]$ is denoted by $e_{G}(S,T)$ for simplicity. 
We will omit the subscript $G$ if it is clear from the context. 

Recall that $H$ is a tree with five edges. 
Thus, we have the following upper bound for $\mathrm{ex}(n, H)$, which can be easily derived using a greedy argument. 
\begin{fact}\label{FACT:H-free-Turan}
    For every integer $n \ge 0$, we have $\mathrm{ex}(n,H) < 5n$.     
\end{fact}

We will use the following results on rainbow matchings, which follow from a simple inductive argument on $t$.
\begin{proposition}\label{PROP:matching-mixed-hypergraphs}
    Let $n, r, t \ge 1$ be integers and $V$ be a set of size $n$. 
    Suppose that for each $i \in [t]$, $\mathcal{G}_i$ is an $r_i$-uniform hypergraph on $V$ with $|\mathcal{G}_i| \ge r t n^{r_i-1}$, where $r_i$ is a positive integer satisfying $r_i \le r$. Then there exists $e_i \in \mathcal{G}_i$ for every $i \in [t]$ such that $e_1, \ldots, e_{t}$ are pairwise disjoint. 
\end{proposition}

Let $\alpha \ge 0$ be a real number. Define  
\begin{align*}
    \Delta_{\alpha}
    \coloneqq \left\{(y_0, y_1, y_2, y_3) \in \mathbb{R}^{4} \colon y_i \ge 0~\text{for}~i\in [0,3]~\text{and}~y_0 + y_1 + y_2 + y_3 \le \alpha\right\}.
\end{align*}
For every $(y_0, y_1, y_2, y_3) \in \mathbb{R}^{4}$, define
\begin{align*}
    \Psi_{\alpha}(y_0, y_1, y_2, y_3)
        &\coloneqq 18 y_0^2 + 12 y_1^2 + 12y_2^2 + 9y_3^2 + 30 y_0 y_1 + 24y_0y_2 + 24y_0y_3 \\ 
        & \quad + 24y_1y_2 + 24y_1y_3 + 21 y_2 y_3 + (y_1 + 2y_2 + 3y_3)(1-6\alpha). 
\end{align*}
Define
\begin{align*}
    \Psi_{\alpha}^{\ast}
    \coloneqq \max\left\{\Psi_{\alpha}(y_0, y_1, y_2, y_3) \colon (y_0, y_1, y_2, y_3) \in \Delta_{\alpha}\right\}.
\end{align*}
The following result, though it can be proved using methods from~{\cite[Appendix~A]{ABHP15}} and~{\cite[Section~7]{HHLY25}}, is proved here using Mathematica for simplicity.
The corresponding calculation file is available at  \url{https://xliu2022.github.io/Calculations-H-tiling.nb}.
\begin{proposition}\label{PROP:optimization-Psi}
    Suppose that $\alpha \in \left[0, \frac{1}{6}\right]$. Then 
    \begin{align*}
        \Psi_{\alpha}^{\ast} 
        = 
        \Xi(\alpha)
        = 
        \begin{cases}
            3\alpha(1-3\alpha), &\quad\text{if}\quad \alpha \in \left[0, \frac{1}{9}\right], \\[0.5em]
            18\alpha^2, &\quad\text{if}\quad \alpha \in \left[\frac{1}{9}, \frac{1}{6}\right].
        \end{cases}
    \end{align*}
\end{proposition}

We now present some standard definitions and results related to the well-known Regularity Lemma and the Blowup Lemma (see~\cite{KSS97,KSSS02} for further details). 

Let $G$ be a bipartite graph with parts $V_1$ and $V_2$. 
The \textbf{density} of $G$, denoted by $\rho(G)$, is defined as $\rho(G) \coloneqq \frac{|G|}{|V_1| |V_2|}$. 
For a real number $\varepsilon \in (0,1]$, we say $G$ is $\varepsilon$-regular if 
\begin{align*}
    |\rho(G[U_1, U_2]) - \rho(G)|
    \le \varepsilon
\end{align*}
for every pair of subsets $U_1 \subseteq V_1$ and $U_2 \subseteq V_2$ satisfying $|U_1| \ge \varepsilon |V_1|$ and $|U_2| \ge \varepsilon |V_2|$. 

Let $\varepsilon \in [0,1]$ be a real number. 
An \textbf{$\varepsilon$-regular partition} of an $n$-vertex graph $G$ is a partition $V_1 \cup \cdots \cup V_{k} = V(G)$ satisfying 
\begin{itemize}
    \item $|V_i - V_j| \le 1$ for all pairs $\{i,j\} \in \binom{[k]}{2}$, and 
    \item $G[V_i, V_j]$ is $\varepsilon$-regular for all but at most $\varepsilon k^2$ pairs $\{i,j\} \in \binom{[k]}{2}$.  
\end{itemize}
For a real number $\rho \in [0,1]$, the \textbf{$\rho$-reduced graph} $R$ of $G$ with respect to an $\varepsilon$-regular partition $V_1 \cup \cdots \cup V_{k} = V(G)$ is defined as the graph with vertex set $\{V_1, \ldots, V_k\}$, where two sets $V_i$ and $V_j$ are adjacent in $R$ iff $G[V_i, V_j]$ is $\varepsilon$-regular and $d(G[V_i, V_j]) \ge \rho$. 
It follows from the definition that 
\begin{align*}
    |R| \cdot \left\lceil \frac{n}{k} \right\rceil^2 + \left(\binom{k}{2}-|R|\right) \cdot \rho  \left\lceil \frac{n}{k} \right\rceil^2 + k \cdot \binom{\left\lceil n/k \right\rceil}{2}
    \ge |G|
\end{align*}
which implies that 
\begin{align}\label{equ:reduced-graph-lower-bound}
    |R|
    & \ge (1-o(1))\frac{k^2}{n^2}\left(|G| - k  \binom{\left\lceil n/k \right\rceil}{2} -\binom{k}{2}  \rho  \left\lceil \frac{n}{k} \right\rceil^2 \right) 
    \ge (1-o(1))\left( \frac{|G|}{n^2} - \frac{1}{2k} - \frac{\rho}{2} \right) k^2.
\end{align}

\begin{lemma}[{\cite[Theorem~2]{KSSS02}}]\label{LEMMA:regularity}
    For every $\varepsilon \in (0,1)$ and $m \in \mathbb{N}$, there exists $M_{\ref{LEMMA:regularity}} = M_{\ref{LEMMA:regularity}}(\varepsilon, m)$ such that every graph $G$ on $n \ge M_{\ref{LEMMA:regularity}}$ vertices has an $\varepsilon$-regular partition $V_1 \cup \cdots \cup V_{k} = V(G)$ for some $k \in [m,M_{\ref{LEMMA:regularity}}]$. 
\end{lemma}

The following lemma, known as the Slicing Lemma, follows directly from the definition of regular partitions.
\begin{lemma}[{\cite[Lemma~3]{GH12}}]\label{LEMMA:blowup-refine-partition}
    Let $\varepsilon, \rho \in (0,1)$ be real numbers and $k \ge 2$ be an integer. 
    Suppose that $R$ is the $\rho$-reduced graph of $G$ with respect to some $\varepsilon$-regular partition. 
    Then the blowup $R[k]$ is a subgraph of the $(\rho-\varepsilon)$-reduced graph of $G$ with respect to some $k \varepsilon$-regular partition. 
\end{lemma}

Let $R$ be a graph, $k \ge 1$ be an integer, and $\varepsilon, \rho \ge 0$ be real numbers. 
\begin{itemize}
    \item The \textbf{$k$-blowup} of $R$, denoted by $R[k]$, is the graph obtained by replacing each vertex of $R$ with a set of size $k$, and each edge of $R$ with a complete bipartite graph between the corresponding sets. 
    \item A \textbf{$(k, \rho, \varepsilon)$-blowup} of $R$, denoted by $R[k,\rho, \varepsilon]$, is the graph obtained by replacing each vertex of $R$ with a set of size $k$, and each edge of $R$ with an $\varepsilon$-regular bipartite graph with density at least $\rho$ between the corresponding sets. 
\end{itemize}


The following result is a variation of the well-known Blowup Lemma.
\begin{lemma}[{\cite[Lemma~6]{GH12}}]\label{LEMMA:blowup}
    Let $R$ and $F$ be two graphs. 
    For every $\beta, \rho \in (0,1)$, there exists $\varepsilon_{\ref{LEMMA:blowup}} = \varepsilon_{\ref{LEMMA:blowup}}(F, \beta, \rho) > 0$ such that the following holds for every integer $k \ge 1$.
    Suppose that $R[k]$ contains an $F$-tiling that covers at least $\beta \cdot k \cdot v(R)$ vertices of $R[k]$. Then $R[k,\rho, \varepsilon]$ also contains an $F$-tiling that covers at least $\beta \cdot k \cdot v(R)$ vertices of $R[k,\rho, \varepsilon]$. 
\end{lemma}

\section{Proof of Theorem~\ref{THM:Main-H-tiling}}\label{SEC:Proof-H-tiling}
In this section, we present the proof of Theorem~\ref{THM:Main-H-tiling}. 
Our proof builds on the framework used by Grosu--Hladk{\'y}~\cite{GH12} (which is a variation of the technique developed by Koml{\'o}s~\cite{Kom00}). However, two crucial aspects require new ideas and additional work: 
First, by the Blowup Lemma (Lemma~\ref{LEMMA:blowup}), to show that $G$ contains an $H$-tiling covering $6\beta$ proportions of the vertex set, it suffices to show that a reduced graph of $G$ (with respect to a certain regular partition of $G$) contains an $H$-tiling of the same proportion. Since, by the Slicing Lemma (Lemma~\ref{LEMMA:blowup-refine-partition}), the blowup of a reduced graph is also a reduced graph (corresponding to a refined regular partition of $G$), it suffices to show that the blowup of a dense graph $R$ contains a large $H$-tiling. This is achieved by finding a large $\{K_2, H, \hat{H} \}$-tiling in $R$ (Proposition~\ref{PROP:H-K2-H-hat-tiling}). The choice of $K_2$ and $\hat{H}$ is not arbitrary; these are carefully selected based on properties established in  Proposition~\ref{PROP:K2-H-hat-decomposition}.
Second, it is nontrivial and quite technical to show that a dense graph contains a large $\{K_2, H, \hat{H}\}$-tiling. To achieve this, we adopt the framework developed by Allen--B\"{o}ttcher--Hladk\'{y}--Piguet in~\cite{ABHP15} for finding large $K_3$-tilings. 

\begin{figure}[H]
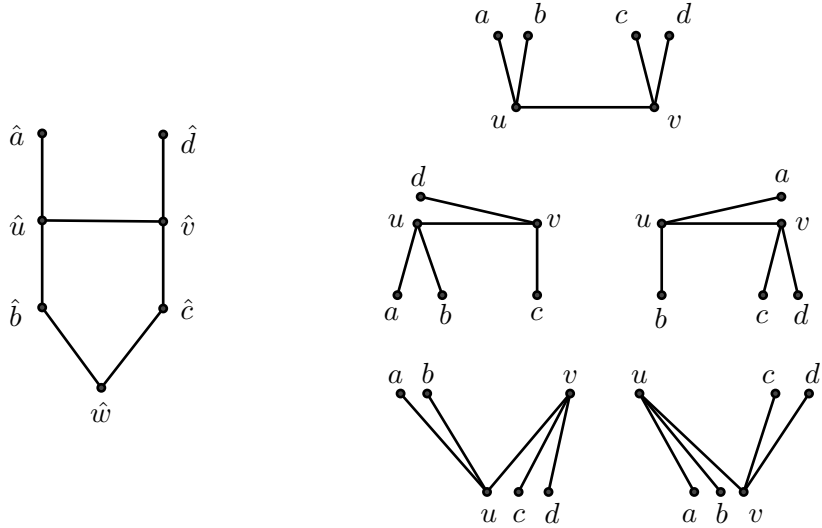

\centering
\tikzset{every picture/.style={line width=1pt}} 

\caption{Five ways to embed $H$ into a blowup of $\hat{H}$.} 
\label{Fig:H-hat-decomposition}
\end{figure}

\begin{proposition}\label{PROP:K2-H-hat-decomposition}
    The following statements hold for every integer $t \ge 1$. 
    \begin{enumerate}[label=(\roman*)]
        \item\label{PROP:K2-H-hat-decomposition-1} The blowup $K_{2}[t]$ contains an $H$-tiling of size $\left\lfloor t/3 \right\rfloor$. In particular, $K_{2}[6]$ contains a perfect $H$-tiling. 
        \item\label{PROP:K2-H-hat-decomposition-2} The blowup $\hat{H}[t]$ contains an $H$-tiling of size $\left\lfloor t/2 \right\rfloor + 4\cdot \left\lfloor t/6\right\rfloor$. In particular, $\hat{H}[6]$ contains a perfect $H$-tiling. 
    \end{enumerate}
\end{proposition}
\begin{proof}[Proof of Proposition~\ref{PROP:K2-H-hat-decomposition}]
    Fix an integer $t \ge 1$. 
    Proposition~\ref{PROP:K2-H-hat-decomposition}~\ref{PROP:K2-H-hat-decomposition-1} follows trivially from the fact that $H \subseteq K_{3,3}$. So it suffices to consider~\ref{PROP:K2-H-hat-decomposition-2}. 
    
    Consider the following five embeddings $\psi_1, \ldots, \psi_5$ of $H$ into $\hat{H}[t]$ (see Figure~\ref{Fig:H-hat-decomposition}):
    \begin{enumerate}[label=(\roman*)]
        \item $\psi_1(u) \in [\hat{u}]$, $\psi_1(v) \in [\hat{v}]$, $\{\psi_1(a), \psi_1(b)\} \subseteq [\hat{a}]$, and $\{\psi_1(c), \psi_1(d)\} \subseteq [\hat{d}]$. 
        \item $\{\psi_2(u), \psi_2(d)\} \subseteq [\hat{u}]$, $\psi_2(v) \in [\hat{v}]$, $\{\psi_2(a), \psi_2(b)\} \subseteq [\hat{b}]$, and $\psi_2(c) \in [\hat{c}]$.
        \item $\psi_3(u) \in [\hat{u}]$, $\{\psi_3(v), \psi_3(a) \} \subseteq [\hat{v}]$, $\psi_3(b) \in [\hat{b}]$, and $\{\psi_3(c), \psi_3(d)\} \subseteq [\hat{c}]$. 
        \item $\{\psi_4(u), \psi_4(c), \psi_4(d) \} \subseteq [\hat{w}]$, $\psi_4(v) \in [\hat{c}]$, and $\{\psi_4(a), \psi_4(b)\} \subseteq [\hat{b}]$. 
        \item $\psi_5(u) \in [\hat{b}]$, $\{\psi_5(v), \psi_5(a), \psi_5(b)\} \subseteq [\hat{w}]$, and $\{\psi_5(c), \psi_5(d)\} \subseteq [\hat{c}]$.
    \end{enumerate}
    Let $m \coloneqq \left\lfloor t/2 \right\rfloor + 4\cdot \left\lfloor t/6 \right\rfloor$. We can embed $m$ vertex-disjoint copies of $H$ by first embedding $\left\lfloor t/2 \right\rfloor$ copies of $H$ using the map $\psi_1$, and then embedding the remaining  $4\cdot \left\lfloor t/6 \right\rfloor$ copies of $H$ using the maps $\psi_2, \ldots, \psi_4$, with each map embedding $\left\lfloor \frac{t}{6} \right\rfloor$ copies. 
\end{proof}

The following result, which is a key ingredient in the proof of Theorem~\ref{THM:Main-H-tiling}, states that if a dense graph has a small $H$-matching number, then one can find a $\{K_2, H, \hat{H}\}$-tiling that covers substantially more vertices than a maximum $H$-tiling.
Its proof is postponed to Section~\ref{SEC:Proof-K2-H-hatH-tiling}. 
\begin{proposition}\label{PROP:H-K2-H-hat-tiling}
    Let $\beta \in (0, 1/6)$ be a real number. 
    For every $\varepsilon > 0$ there exist $\delta_{\ref{PROP:H-K2-H-hat-tiling}} = \delta_{\ref{PROP:H-K2-H-hat-tiling}}(\beta,\varepsilon)>0$ and $N_{\ref{PROP:H-K2-H-hat-tiling}} = N_{\ref{PROP:H-K2-H-hat-tiling}}(\beta,\varepsilon)$ such that the following holds for all $n \ge N_{\ref{PROP:H-K2-H-hat-tiling}}$. 
    Suppose that $G$ is an $n$-vertex graph with $\nu(H, G) \le \beta n$ and 
    \begin{align*}
        |G|
        \ge \left(\Xi(\beta) + \varepsilon \right) n^2
    \end{align*}
    Then $G$ contains a $\{K_2, H, \hat{H}\}$-tiling that covers at least $6 \cdot \nu(H, G) + \delta_{\ref{PROP:H-K2-H-hat-tiling}} n$ vertices of $G$. 
\end{proposition}

A direct consequence of Propositions~\ref{PROP:K2-H-hat-decomposition} and~\ref{PROP:H-K2-H-hat-tiling} is the following result, which states that every sufficiently large blowup of a dense graph contains an $H$-tiling of the expected size. 
\begin{corollary}\label{CORO:blowup-H-tiling}
     Let $\beta \in (0, 1/6)$ be a real number. 
    For every $\varepsilon > 0$ there exist $N_{\ref{CORO:blowup-H-tiling}} = N_{\ref{CORO:blowup-H-tiling}}(\beta,\varepsilon)$ and $M_{\ref{CORO:blowup-H-tiling}} = M_{\ref{CORO:blowup-H-tiling}}(\beta,\varepsilon)$ such that the following holds for all $n \ge N_{\ref{PROP:H-K2-H-hat-tiling}}$. 
    Suppose that $G$ is a graph on $n$ vertices with $|G| \ge \left(\Xi(\beta) + \varepsilon\right)n^2$. 
    Then $G[M_{\ref{CORO:blowup-H-tiling}}]$ contains an $H$-tiling of size at least $\beta M_{\ref{CORO:blowup-H-tiling}} n$. 
\end{corollary}
\begin{proof}[Proof of Corollary~\ref{CORO:blowup-H-tiling}]
    Fix $\beta \in (0, 1/6)$ and $\varepsilon > 0$. 
    Let $\delta_{\ref{PROP:H-K2-H-hat-tiling}} = \delta_{\ref{PROP:H-K2-H-hat-tiling}}(\beta,\varepsilon)>0$ and $N_{\ref{PROP:H-K2-H-hat-tiling}} = N_{\ref{PROP:H-K2-H-hat-tiling}}(\beta,\varepsilon)$ be the constants returned by Proposition~\ref{PROP:H-K2-H-hat-tiling}. 
    Let $M_{\ref{CORO:blowup-H-tiling}} \coloneqq 6^{k_{\ast}}$, where $k_{\ast} \coloneqq \lceil 6 \beta/\delta_{\ref{PROP:H-K2-H-hat-tiling}} \rceil$. 
    Let $V \coloneqq V(G)$. 
    Set $G_0 \coloneqq G$, and for $i \le k_{\ast}$, define $G_{i+1} \coloneqq G_{i}[6]$. 

    \begin{claim}\label{CLAIM:tau-increase}
        For every $i \in [k_{\ast}]$, we have 
        \begin{align*}
            \tau(H, G_{i}) 
            \ge \min\left\{ \tau(H, G_{i-1}) + \delta_{\ref{PROP:H-K2-H-hat-tiling}},~6 \beta \right\}.  
        \end{align*}
    \end{claim}
    \begin{proof}[Proof of Claim~\ref{CLAIM:tau-increase}]
        Fix $i \in [k_{\ast}]$. 
        If $\tau(H, G_{i-1}) \ge 6 \beta$, then it trivially follows that $\tau(H, G_{i}) \ge \tau(H, G_{i-1}) \ge 6 \beta$, as desired. 
        So we may assume that $\tau(H, G_{i-1}) < 6 \beta$. 
        Since $v(G_{i}) = 6^{i} \cdot v(G) \ge N_{\ref{PROP:H-K2-H-hat-tiling}}$ and 
        \begin{align*}
            |G_{i-1}|
            = 36^{i-1} \cdot |G_0|
            \ge \left(\Xi(\beta) + \varepsilon\right) \left(v(G_{i-1})\right)^2, 
        \end{align*}
        applying Proposition~\ref{PROP:H-K2-H-hat-tiling} to $G_{i-1}$, we find a $\{K_2, H, \hat{H}\}$-tiling $\mathcal{H}$ that covers at least $6 \cdot \nu(H, G_{i-1}) + \delta_{\ref{PROP:H-K2-H-hat-tiling}} \cdot v(G_{i-1})$ vertices of $G_{i-1}$. 
        
        Consider the blowup $G_{i} = G_{i-1}[6]$.
        Note that $\mathcal{H}$ naturally gives a $\{K_2[6], H[6], \hat{H}[6]\}$-tiling, denoted by $\mathcal{H}'$, in $G_{i}$ that covers 
        \begin{align*}
            6 \cdot |V(\mathcal{H})|
            \ge 6 \cdot \left(6 \cdot \nu(H, G_{i-1}) + \delta_{\ref{PROP:H-K2-H-hat-tiling}} \cdot v(G_{i-1})\right)
            = 36 \cdot \nu(H, G_{i-1}) + \delta_{\ref{PROP:H-K2-H-hat-tiling}} \cdot v(G_{i})
        \end{align*}
        vertices in $G_{i}$. 
        By Proposition~\ref{PROP:K2-H-hat-decomposition} and the fact that $H[6]$ contains a perfect $H$-tiling, each member in $\mathcal{H}'$ contains a perfect $H$-tiling. Thus we can transform $\mathcal{H}'$ into an $H$-tiling in $G_i$ that covers the same number of vertices as $\mathcal{H}'$. 
        It follows that 
        \begin{align*}
            \tau(H, G_{i})
            \ge \frac{36 \cdot \nu(H, G_{i-1}) + \delta_{\ref{PROP:H-K2-H-hat-tiling}} \cdot v(G_{i})}{v(G_i)}
            = \frac{36 \cdot \nu(H, G_{i-1})}{6 \cdot v(G_{i-1})} + \delta_{\ref{PROP:H-K2-H-hat-tiling}}
            =  \tau(H, G_{i-1}) + \delta_{\ref{PROP:H-K2-H-hat-tiling}}, 
        \end{align*}
        as desired. 
    \end{proof}

    Since $\tau(H, G_0) + k_{\ast} \cdot \delta_{\ref{PROP:H-K2-H-hat-tiling}} \ge \lceil 6 \beta/\delta_{\ref{PROP:H-K2-H-hat-tiling}} \rceil \cdot \delta_{\ref{PROP:H-K2-H-hat-tiling}} \ge 6 \beta$, it follows from Claim~\ref{CLAIM:tau-increase} that $\tau(H,G_{k_{\ast}}) \ge \max\{\tau(H, G_0) + k_{\ast} \cdot \delta_{\ref{PROP:H-K2-H-hat-tiling}},~6\beta\} = 6 \beta$. 
    This completes the proof of Corollary~\ref{CORO:blowup-H-tiling}. 
\end{proof}

We are now ready to present the proof of Theorem~\ref{THM:Main-H-tiling}. 

\begin{proof}[Proof of Theorem~\ref{THM:Main-H-tiling}]
    Fix $\beta \in (0, 1/6)$ and $\varepsilon > 0$. We may assume that $\varepsilon$ is sufficiently small. 
    Let $\delta_{\ref{PROP:H-K2-H-hat-tiling}} = \delta_{\ref{PROP:H-K2-H-hat-tiling}}(\beta, \varepsilon) > 0$ and $N_{\ref{PROP:H-K2-H-hat-tiling}} = N_{\ref{PROP:H-K2-H-hat-tiling}}(\beta, \varepsilon)$ be the constants returned by Proposition~\ref{PROP:H-K2-H-hat-tiling}. 
    Let $N_{\ref{CORO:blowup-H-tiling}} = N_{\ref{CORO:blowup-H-tiling}}(\beta,\varepsilon)$ and $M_{\ref{CORO:blowup-H-tiling}} = M_{\ref{CORO:blowup-H-tiling}}(\beta,\varepsilon)$ be the constants returned by Corollary~\ref{CORO:blowup-H-tiling}. 
    Let $\rho, \varepsilon_1 > 0$ be sufficiently small constants that satisfy $\varepsilon_1 \ll \rho \ll \varepsilon$.  
    In particular, we can choose $\varepsilon_1$ sufficiently small such that $\varepsilon_1 \ll \varepsilon_{\ref{LEMMA:blowup}}$, where $\varepsilon_{\ref{LEMMA:blowup}} = \varepsilon_{\ref{LEMMA:blowup}}(H, \beta, \rho) > 0$ is the constant returned by Lemma~\ref{LEMMA:blowup}. 
    Let $m \gg \max\{\varepsilon_1^{-1}, N_{\ref{PROP:H-K2-H-hat-tiling}}\}$ be a sufficiently large integer. 
    Let $M_{\ref{LEMMA:regularity}} = M_{\ref{LEMMA:regularity}}(\varepsilon_1, m)$ be the constant returned by Lemma~\ref{LEMMA:regularity}. 
    Let $n$ be a sufficiently large integer and $G$ be an $n$-vertex graph with $|G| \ge \left(\Xi(\beta) + 2 \varepsilon\right)n^2$. 
    Let $V \coloneqq V(G)$. 

    Applying Lemma~\ref{LEMMA:regularity} to $G$, with $\varepsilon, m$ there corresponding to $\varepsilon_1, m$ here, we obtain an $\varepsilon_1$-regular partition $V_1 \cup \cdots \cup V_{k} = V(G)$ of $G$ for some $k \in [m, M_{\ref{LEMMA:regularity}}]$. 
    Let $R$ denote the corresponding $\rho$-reduced graph of $G$. It follows from~\eqref{equ:reduced-graph-lower-bound} that 
    \begin{align}\label{equ:R-edge-density}
        |R|
        \ge (1-o(1))\left( \frac{|G|}{n^2} - \frac{1}{2k} - \frac{\rho}{2} \right) k^2 
        \ge \left(\Xi(\beta) + \varepsilon\right) k^2. 
    \end{align}
    It follows from Corollary~\ref{CORO:blowup-H-tiling} that $R[M_{\ref{CORO:blowup-H-tiling}}]$ contains an $H$-tiling covering at least $6\beta$ of its vertex set. 
    Therefore, by Lemma~\ref{LEMMA:blowup}, $G$ also contains an $H$-tiling covering at least $6\beta$ of its vertex set. 
    This completes the proof of Theorem~\ref{THM:Main-H-tiling}. 
\end{proof}

\section{Proof of Proposition~\ref{PROP:H-K2-H-hat-tiling}}\label{SEC:Proof-K2-H-hatH-tiling}
We prove Proposition~\ref{PROP:H-K2-H-hat-tiling} in this section. 
For simplicity, we will omit the floors and ceilings from our calculations, as our focus is on asymptotic bounds.

Fix $\beta \in \left(0, \frac{1}{6}\right)$ and $\varepsilon > 0$. We may assume that $\varepsilon$ is sufficiently small. 
Let $\delta > 0$ be another sufficiently small constant such that $\delta \ll \varepsilon$.
Let $n$ be a sufficiently large integer. 
Let $G$ be an $n$-vertex graph such that 
\begin{align}\label{equ:assumption-G-lower-bound}
    \nu(H, G) \leq \beta n
    \quad\text{and}\quad 
    |G| 
    \ge \left(\Xi(\beta) + \varepsilon \right) n^2. 
\end{align}
Suppose to the contrary that $G$ does not contain a $\{K_2, H, \hat{H}\}$-tiling that covers at least $6 \cdot \nu(H,G) + \delta n$ vertices in $G$. 
Let $t \coloneqq \nu(H, G)$ and $\mathcal{H} = \{H_1, \ldots, H_t\}$ be an $H$-tiling in $G$ of size $t$, where each $H_i$ is a copy of $H$ in $G$. 
Let us label vertices in each $H_i$ by $\{u_i, v_i, a_i, b_i, c_i, d_i\}$ such that 
\begin{align*}
    H_i = \left\{u_iv_i, u_ia_i, u_ib_i, v_ic_i, v_id_i\right\}.
\end{align*}
Let 
\begin{align*}
    V \coloneqq V(H), \quad 
    U \coloneqq V(H_1) \cup \cdots \cup V(H_t),
    \quad\text{and}\quad 
    \overline{U} \coloneqq V\setminus U.  
\end{align*}
It follows from the definition of $\mathcal{H}$ that $|U| = 6 t$. 
Additionally, it follows from $\nu(H, G) = |\mathcal{H}| = t$ that the induced subgraph $G[\overline{U}]$ is $H$-free. Thus, by Fact~\ref{FACT:H-free-Turan}, we have 
\begin{align}
    |G[\overline{U}]| 
    \le 5 |\overline{U}|
    = 5 (n - 6t)
    \le 5n. 
\end{align}
Let 
\begin{align*}
    L 
    \coloneqq \left\{u \in U \colon |N_{G}(u) \cap \overline{U}| \ge 19 \delta^{1/6} n\right\}.
\end{align*}
For a pair of distinct members $\{H_i, H_j\}$ in  $\mathcal{H}$: 
\begin{itemize}
    \item We say $\{H_i, H_j\}$ is \textbf{extendable with respect to a set $S \subseteq \overline{U}$} if there exists a $\{K_2, H, \hat{H}\}$-tiling in the induced subgraph $G[V(H_i) \cup V(H_j) \cup S]$ that covers at least $13$ vertices in $V(H_i) \cup V(H_j) \cup S$.
    \item For an integer $\ell \ge 1$, we say $\{H_i, H_j\}$ is \textbf{$\ell$-extendable} if it is extendable with respect to at least $19 \delta n^{\ell}$ sets of size $\ell$ in $\overline{U}$. 
    \item We say $\{H_i, H_j\}$ is \textbf{extendable} if it is $\ell$-extendable for some $\ell \ge 1$. 
\end{itemize}
Define an auxiliary graph $\mathcal{A}$ whose vertex set is $\mathcal{H}$, and two members $H_i, H_j \in \mathcal{H}$ are adjacent in $\mathcal{A}$ iff the pair $\{H_i, H_j\}$ is extendable. 

Take a matching $\mathcal{M} \subseteq \mathcal{A}$ of maximum size, and let $\mathcal{H}' \subseteq \mathcal{H}$ denote the subfamily obtained by removing members of $\mathcal{H}$ that are covered by $\mathcal{M}$.
Note from the maximality of $\mathcal{M}$ that no pair in $\mathcal{H}'$ is $\ell$-extendable for any $\ell$. 

Below, we present several lemmas that are essential for proving Proposition~\ref{PROP:H-K2-H-hat-tiling}.  
\begin{lemma}\label{LEMMA:extendable-ell-upper-bound}
    Suppose that $\{H_i, H_j\} \subseteq \mathcal{H}$ is extendable. Then it is $\ell$-extendable for some integer $\ell \le 19$. 
\end{lemma}
\begin{proof}[Proof of Lemma~\ref{LEMMA:extendable-ell-upper-bound}]
    Let $\ell \ge 1$ be the smallest integer such that $\{H_i, H_j\}$ is $\ell$-extendable.
    Suppose to the contrary that $\ell \ge 20$. 
    By definition, there exists an $\ell$-graph $\mathcal{G} \subseteq \binom{\overline{U}}{\ell}$ of size at least $19 \delta n^{\ell}$ such that $\{H_i, H_j\}$ is $S$-extendable for every $S \in \mathcal{G}$. 
    
    Fix a set $S \in \mathcal{G}$. Let $S' \subseteq S$ be a nonempty set of smallest size such that $\{H_i, H_j\}$ is $S'$-extendable. 
    We claim that $|S'| \le 19$. 
    Indeed, suppose to the contrary that $|S'| \ge 20$.
    Let $\mathcal{H}'$ be a $\{K_2, H, \hat{H}\}$-tiling of $G[V(H_i) \cup V(H_j) \cup S']$ that covers at least $13$ vertices. 
    By the minimality of $S'$, we know that $\mathcal{H}'$ covers all vertices in $S'$.
    Removing an arbitrary member that have nonempty intersection with $S'$ from $\mathcal{H}'$, we obtain a new $\{K_2, H, \hat{H}\}$-tiling of $G[V(H_i) \cup V(H_j) \cup S']$ that covers at least $|S'| - 7 \ge 13$ vertices, which contradicts the minimality of $S'$. 
    Therefore, $|S'| \le 19$. 
    This shows that every set $S \in \mathcal{G}$ contains a subset $S'$ of size $19$ (we can enlarge $S'$ if it is less than $19$) such that $\{H_i, H_j\}$ is $S'$-extendable. 
    Let $\mathcal{G}'$ be the collection of such sets. Note that each set in $\mathcal{G}'$ is contained in at most $\binom{|\overline{U}|-19}{\ell-19} \le n^{\ell-19}$ sets in $\mathcal{G}$. Therefore, we have 
    \begin{align*}
        |\mathcal{G}'|
        \ge \frac{|\mathcal{G}|}{n^{\ell-19}}
        \ge \frac{19 \delta n^{\ell}}{n^{\ell-19}}
        \ge 19 \delta n^{19}, 
    \end{align*}
    which means that $\{H_i, H_j\}$ is $19$-extendable, contradicting the minimality of $\ell$. 
\end{proof}

\begin{lemma}\label{LEMMA:aux-graph-matching-number}
    We have $\nu(\mathcal{A}) \le \delta n$. 
\end{lemma}
\begin{proof}[Proof of Lemma~\ref{LEMMA:aux-graph-matching-number}]
    Suppose to the contrary that $\nu(\mathcal{A}) \ge \delta n$. 
    By relabeling members of $\mathcal{H}$, we may assume that $\{H_{2i-1}, H_{2i}\} \in \mathcal{A}$ for $i \in [\delta n]$. 
    By the definition of $\mathcal{A}$ and Lemma~\ref{LEMMA:extendable-ell-upper-bound}, for each $i \in [\delta n]$, there exists an integer $r_i \le 19$ such that $\{H_{2i-1}, H_{2i}\}$ is $r_i$-extendable.
    In other words, for each $i \in [\delta n]$, there exists an $r_i$-graph $\mathcal{G}_i \subseteq \binom{\overline{U}}{r_i}$ of size at least $19 \delta n^{r_i}$ such that, for every $S \in \mathcal{G}_i$, the induced subgraph $G[V(H_{2i-1}) \cup V(H_{2i}) \cup S]$ contains a $\{K_2, H, \hat{H}\}$-tiling that covers at least $13$ vertices in $V(H_{2i-1}) \cup V(H_{2i}) \cup S$. 

    Since $|\mathcal{G}_i| \ge 19 \delta n^{r_i} = 19 \cdot \delta n \cdot n^{r_i-1}$ for every $i \in [\delta n]$, it follows from  Proposition~\ref{PROP:matching-mixed-hypergraphs} that there exists $S_i \in \mathcal{G}_i$ for every $i \in [\delta n]$ such that $S_1, \ldots, S_{\delta n}$ are pairwise disjoint. 
    For each $i \in [\delta n]$, let $\mathcal{H}_i$ be a $\{K_2, H, \hat{H}\}$-tiling in $G[V(H_{2i-1}) \cup V(H_{2i}) \cup S_i]$ that covers at least $13$ vertices of $V(H_{2i-1}) \cup V(H_{2i}) \cup S_i$. 
    Then, it is easy to see that $\mathcal{H}_1 \cup \cdots \cup \mathcal{H}_{\delta n} \cup \{H_{2\delta n+1}, \ldots, H_t\}$ is a $\{K_2, H, \hat{H}\}$-tiling in $G$ that covers at least 
    \begin{align*}
        13 \cdot \delta n + 6 \cdot (t - 2\delta n)
        = 6 t + \delta n
    \end{align*}
    vertices, which contradicts our assumption. 
    This completes the proof of Lemma~\ref{LEMMA:aux-graph-matching-number}. 
\end{proof}

Note from Lemma~\ref{LEMMA:aux-graph-matching-number} that $|\mathcal{H}'| = |\mathcal{H}| - 2|\mathcal{M}| \ge (1-2\delta)t$. 

%
\begin{lemma}\label{LEMMA:H-R-at-most-3-edges}
    The following statements hold for every $H_i \in \mathcal{H}'$. 
    \begin{enumerate}[label=(\roman*)]
        \item\label{LEMMA:H-R-at-most-3-edges-1} Either $\{a_i, b_i\} \cap L = \emptyset$ or $\{c_i, d_i\} \cap L = \emptyset$.
        \item\label{LEMMA:H-R-at-most-3-edges-2} Either $\{a_i, b_i\} \cap L = \emptyset$ or $u_i \not\in L$. 
        \item\label{LEMMA:H-R-at-most-3-edges-3} Either $\{c_i, d_i\} \cap L = \emptyset$ or $v_i \not\in L$. 
    \end{enumerate}
    In particular, only one of the following three cases can occur: 
    \begin{align*}
        |V(H_i) \cap L| \le 2,\quad
        V(H_i) \cap L = \{u_i, c_i, d_i\},\quad\text{or}\quad  
        V(H_i) \cap L = \{v_i, a_i, b_i\}. 
    \end{align*}
\end{lemma}
\begin{proof}[Proof of Lemma~\ref{LEMMA:H-R-at-most-3-edges}]
    Let us first prove~\ref{LEMMA:H-R-at-most-3-edges-1}. 
    Suppose to the contrary that~\ref{LEMMA:H-R-at-most-3-edges-1} fails. 
    By symmetry, we may assume that $b_i \in L$ and $c_i \in L$.  
    Fix an arbitrary pair of distinct vertices $(w_1, w_2) \in N_{G}(b_i, \overline{U}) \times N_{G}(c_i, \overline{U})$. 
    Note from the definition of $L$ that the number of such (unordered) pairs $\{w_1, w_2\}$ is at least 
    \begin{align}\label{equ:w1w2-choices}
        \frac{1}{2} \cdot 19 \delta^{1/6} n \cdot (19 \delta^{1/6} n -1) 
        \ge 19 \delta n^2. 
    \end{align}
    Let $H_j \in \mathcal{H}' \setminus \{H_i\}$ be an arbitrary member. 
    Note that $G[V(H_i) \cup \{w_1, w_2\}]$ contains a $K_2$-tiling of size four $\{w_1, b_i\}$, $\{w_2, c_i\}$, $\{u_i, a_i\}$, $\{v_i, d_i\}$. This together with $H_j$ forms a $\{K_2, H\}$-tiling in $G[V(H_i) \cup V(H_j) \cup \{w_1, w_2\}]$ that covers $2\cdot 4 + 6 = 14 \ge 13$ vertices. Thus, $\{H_i, H_j\}$ is $\{w_1, w_2\}$-extendable. By~\eqref{equ:w1w2-choices}, this means that $\{H_i, H_j\}$ is $2$-extendable, which is a contradiction. 

    Next, we prove~\ref{LEMMA:H-R-at-most-3-edges-2}.
    Suppose to the contrary that~\ref{LEMMA:H-R-at-most-3-edges-2} fails. 
    By symmetry, we may assume that $u_i \in L$ and $a_i \in L$. 
    Fix an arbitrary pair of distinct vertices $(w_1, w_2) \in N_{G}(b_i, \overline{U}) \times N_{G}(c_i, \overline{U})$. 
    Let $H_j \in \mathcal{H}' \setminus \{H_i\}$ be an arbitrary member. 
    Note that $G[V(H_i) \cup \{w_1, w_2\}]$ contains a perfect $\{K_2,H\}$-tiling: $\{w_1, a_i\}$ and $\{u_iv_i, u_iw_2, u_ib_i, v_ic_i, v_id_i\}$. This together with $H_j$ forms a $\{K_2, H\}$-tiling in $G[V(H_i) \cup V(H_j) \cup \{w_1, w_2\}]$ that covers $7 + 6 = 13$ vertices. Thus, $\{H_i, H_j\}$ is $\{w_1, w_2\}$-extendable. 
    Similarly,  this means that $\{H_i, H_j\}$ is $2$-extendable, which is a contradiction. 
    
    By symmetry,~\ref{LEMMA:H-R-at-most-3-edges-3} holds as well.
    This completes the proof of Lemma~\ref{LEMMA:H-R-at-most-3-edges}. 
\end{proof}

By Lemma~\ref{LEMMA:H-R-at-most-3-edges}~\ref{LEMMA:H-R-at-most-3-edges-1}, we can relabel each $H_i \in \mathcal{H}'$ such that the following always holds:  
\begin{align}\label{equ:assum-ci-di-not-in-L}
    \{c_i, d_i\} \cap L = \emptyset. 
\end{align}

For each $\ell \in \mathbb{N}$, let 
\begin{align*}
    \mathcal{H}_{\ell} 
    \coloneqq \left\{H_i \in \mathcal{H}' \colon |V(H_i) \cap L| = \ell \right\}
    \quad\text{and}\quad 
    V_{\ell}
    \coloneqq V(\mathcal{H}_{\ell})
    \coloneqq \bigcup_{H_i \in \mathcal{H}_{\ell}}V(H_i). 
\end{align*}
It follows from Lemma~\ref{LEMMA:H-R-at-most-3-edges} that $\mathcal{H}_{\ell} = \emptyset$ for every $\ell \ge 4$, and hence, $\mathcal{H}_0 \cup \mathcal{H}_1 \cup \mathcal{H}_2 \cup \mathcal{H}_3$ is a partition of $\mathcal{H}'$. 
Let $y_i \coloneqq {|\mathcal{H}_i|}/{n}$ for $i \in [0,3]$. 
Note that 
\begin{align*}
    y_0 + y_1 + y_2 + y_3 
    = \frac{|\mathcal{H}'|}{n}
    \le \frac{|\mathcal{H}|}{n}
    \le \frac{t}{n}. 
\end{align*}

For convenience, for $(i, j) \in [0,3] \times [0,3]$, we will use $e(\mathcal{H}_i, \mathcal{H}_j)$ to denote the number of edges in the induced subgraph $G[V_i, V_j]$, with $G[V_i, V_i] \coloneqq G[V_i]$. 
The following lemma, whose proof is deferred to Section~\ref{SEC:proof-LEMMA-Hi-Hj-U-bar-edges}, provides an upper bound for the number of edges inside $G[V_0 \cup \cdots \cup V_3]$. 
\begin{lemma}\label{LEMMA:Hi-Hj-U-bar-edges}
    We have 
    \begin{align*}
        \frac{1}{n^2} \sum_{0 \le i \le j \le 3} e(\mathcal{H}_i, \mathcal{H}_j)  
        & \le 18 y_0^2 + 12 y_1^2 + 12y_2^2 + 9y_3^2  + 30 y_0 y_1 + 24y_0y_2 + 24y_0y_3 \\ 
        & \quad + 24y_1y_2 + 24y_1y_3 + 21 y_2 y_3 + \frac{3y_1}{n} + \frac{3y_2}{n} + \frac{6y_3}{n}.  
    \end{align*}
\end{lemma}
Lemma~\ref{LEMMA:Hi-Hj-U-bar-edges} follows from the trivial upper bound 
    \begin{align*}
        e(\mathcal{H}_0) \le 36\binom{y_0 n}{2} + \binom{6}{2}y_0 n 
        \le 18 y_0^2 n^2, 
    \end{align*}
as well as Lemmas~\ref{LEMMA:H0-H1-upper-bound},~\ref{LEMMA:H1-H1-upper-bound},~\ref{LEMMA:H2-Hi-upper-bound}, and~\ref{LEMMA:H3-upper-bound}, whose statements and proofs are deferred to  Section~\ref{SEC:proof-LEMMA-Hi-Hj-U-bar-edges} for technical reasons. 

Let us now present the proof of Proposition~\ref{PROP:H-K2-H-hat-tiling}. 
\begin{proof}[Proof of Proposition~\ref{PROP:H-K2-H-hat-tiling}]   
    Let $\alpha \coloneqq t/n$ and recall that $y_i \coloneqq {|\mathcal{H}_i|}/{n}$ for $i \in [0,3]$. 
    Note that, by the definition of $\mathcal{H}_0, \ldots, \mathcal{H}_3$, we have 
    \begin{align*}
        \frac{1}{n^2} \sum_{i=0}^{3} e(\mathcal{H}_i, \overline{U}) 
        & \le (y_1 + 2y_2 + 3y_3)(1-6\alpha) + 6(y_0 + y_1 + y_2 + y_3) \cdot 19 \delta^{1/6} \\
        & \le (y_1 + 2y_2 + 3y_3)(1-6\alpha) +  114 \alpha \delta^{1/6}. 
    \end{align*}
    Combining this with Lemmas~\ref{LEMMA:aux-graph-matching-number} and~\ref{LEMMA:Hi-Hj-U-bar-edges}, we obtain 
    \begin{align*}
        \frac{|G|}{n^2}
        & = \frac{1}{n^2} \left(\sum_{0 \le i \le 3} e(\mathcal{H}_i, \overline{U})  + \sum_{0\le i \le j \le 3} e(\mathcal{H}_i, \mathcal{H}_j) + e(\mathcal{H} \setminus \mathcal{H}', V) + |G[\overline{U}]| \right)  \\ 
        & \le (y_1 + 2y_2 + 3y_3)(1-6\alpha) +  114 \alpha \delta^{1/6} +  18 y_0^2 + 12 y_1^2 + 12y_2^2 + 9y_3^2 + 30 y_0 y_1  \\ 
        & \quad  + 24y_0y_2 + 24y_0y_3 + 24y_1y_2 + 24y_1y_3 + 21 y_2 y_3 + \frac{3y_1}{n} + \frac{3y_2}{n} + \frac{6y_3}{n} + 12 \delta + \frac{5}{n}\\
        & \le \Phi_{\alpha}(y_0, y_1, y_2, y_3) + 200 \delta^{1/6}, 
    \end{align*}
    where recall that 
    \begin{align*}
        \Phi_{\alpha}(y_0, y_1, y_2, y_3)
        &\coloneqq (y_1 + 2y_2 + 3y_3)(1-6\alpha) +   18 y_0^2 + 12 y_1^2 + 12y_2^2 + 9y_3^2  \\ 
        & \quad + 30 y_0 y_1 + 24y_0y_2 + 24y_0y_3 + 24y_1y_2 + 24y_1y_3 + 21 y_2 y_3.
    \end{align*}
    Note that $(y_0, y_1, y_2, y_3) \in \Delta_{\alpha}$ and $\alpha \le \beta$. So it follows from Proposition~\ref{PROP:optimization-Psi} and the monotonicity of $\Xi(\cdot)$ that $\Phi(y_0, y_1, y_2, y_3) \le \Xi(\alpha) \le \Xi(\beta)$. 
    Therefore, 
    \begin{align*}
        \frac{|G|}{n^2} 
        \le \Xi(\beta) + 200 \delta^{1/6}, 
    \end{align*}
    which proves Proposition~\ref{PROP:H-K2-H-hat-tiling}. 
\end{proof}

\section{Local estimation}\label{SEC:proof-LEMMA-Hi-Hj-U-bar-edges}
In this section, we present the proof of Lemma~\ref{LEMMA:Hi-Hj-U-bar-edges}, continuing with the notations introduced in the previous section.
The proofs in this section, while elementary, are quite tedious due to extensive case analyses. To enhance clarity, we include numerous figures to aid in understanding the arguments.

For convenience, for a pair $\{H_i, H_j\}$ in $\mathcal{H}'$, we use $G[H_i, H_j]$ to denote the induced bipartite graph $G[V(H_i), V(H_j)]$, and use $e(H_i, H_j)$ to denote the number of edges in $G[V(H_i), V(H_j)]$. 

\subsection{Upper bound for $e(\mathcal{H}_0, \mathcal{H}_1)$}\label{SUBSEC:H0-H1-upper-bound}
\begin{lemma}\label{LEMMA:H0-H1-upper-bound}
    Suppose that $\{H_i, H_j\} \subseteq \mathcal{H}'$ is a pair such that $|V(H_i) \cap L| \ge 1$. 
    Then $e(H_i, H_j) \le 30$. 
    In particular, $e(\mathcal{H}_0, \mathcal{H}_1) \le 30 y_0 y_1 n^2$. 
\end{lemma}
\begin{proof}[Proof of Lemma~\ref{LEMMA:H0-H1-upper-bound}]
    Fix a pair $\{H_i, H_j\} \subseteq \mathcal{H}'$ such that $|V(H_i) \cap L| \ge 1$.
    Suppose to the contrary that $e(H_i, H_j) \ge 31$. 

\medskip 

    \textbf{Case 1}: $\{u_i, v_i\} \cap L \neq\emptyset$. 

    By symmetry, we may assume that  $u_i \in L$. Fix an arbitrary vertex $w \in N_{G}(u_i, \overline{U})$. 
    Note from the definition of $L$ that the number of choices for such a vertex $w$ is at least $19 \delta^{1/6} n$. 

    \begin{claim}\label{CLAIM:one-edge-a}
        There exists a vertex in $\{a_j, b_j, c_j, d_j\}$ that has neighbors in both $\{a_i, b_i\}$ and $\{c_i, d_i\}$.
    \end{claim}
    \begin{proof}[Proof of Claim~\ref{CLAIM:one-edge-a}]
        Suppose to the contrary that this is not true. Then we would have 
        \begin{align*}
            e(H_i, H_j) 
            \le 36 - 4\cdot 2
            = 28 
            < 31,
        \end{align*}
        a contradiction. 
    \end{proof}

    By symmetry, we may assume that $\{d_jb_i, d_jd_i\} \subseteq G$.

    \begin{claim}\label{CLAIM:one-edge-b}
        The following statements hold. 
        \begin{enumerate}[label=(\roman*)]
            \item\label{CLAIM:one-edge-b-1} $\{b_ia_j, b_ib_j\}  \cap G = \emptyset$.
            \item\label{CLAIM:one-edge-b-2} $a_iv_j \not\in G$. 
            \item\label{CLAIM:one-edge-b-3} $\{ a_i a_j, a_i b_j\}  \cap G = \emptyset$.
        \end{enumerate}
    \end{claim}
    \begin{proof}[Proof of Claim~\ref{CLAIM:one-edge-b}]
        Let us first prove~\ref{CLAIM:one-edge-b-1}. 
        By symmetry, it suffices to show that $b_ia_j \not\in G$. 
        Observe that the set $\{u_i, v_i, a_i, w, c_i, d_i\}$ spans a copy of $H$ in $G$. 
        If $b_ia_j \in G$, then the set $\{u_j, v_j, a_j, b_j, c_j, d_j, b_i\}$ would span a copy of $\hat{H}$ in $G$. Since the number of choices for $w$ is at least $19 \delta^{1/6} n$, by definition,  the pair $\{H_i, H_j\}$ is $1$-extendable, a contradiction. 

        Next, we prove~\ref{CLAIM:one-edge-b-2}.
        Observe that the set $\{u_i, v_i, w, b_i, c_i, d_i, d_j\}$ spans a copy of $\hat{H}$ in $G$. If $a_iv_j \in G$, then the set $\{u_j, v_j, a_j, b_j, c_j, a_i\}$ would span a copy of $H$ in $G$. Since the number of choices for $w$ is at least $19 \delta^{1/6} n$, by definition,  the pair $\{H_i, H_j\}$ is $1$-extendable, a contradiction.  

\begin{figure}[H]
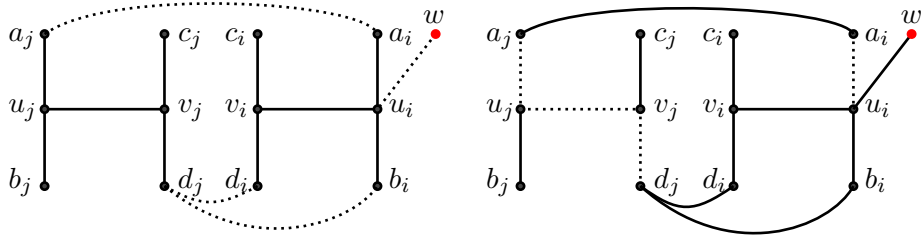

\centering
\tikzset{every picture/.style={line width=1pt}} 

\caption{Auxiliary figure for the proof of Claim~\ref{CLAIM:one-edge-b}~\ref{CLAIM:one-edge-b-3}.} 
\label{Fig:one-edge-rotation-1}
\end{figure}

        Finally, we prove~\ref{CLAIM:one-edge-b-3}. 
        By symmetry, it suffices to show that $a_i a_j \not\in G$. Suppose to the contrary that this fails. 
        Then the set $V(H_i \cup H_j) \cup \{w\}$ can be covered by a copy of $\hat{H}$ and three pairwise disjoint edges, as shown on the right side of  Figure~\ref{Fig:one-edge-rotation-1}. Since the number of choices for $w$ is at least $19 \delta^{1/6} n$, by definition,  the pair $\{H_i, H_j\}$ is $1$-extendable, a contradiction. 
    \end{proof}

\begin{figure}[H]
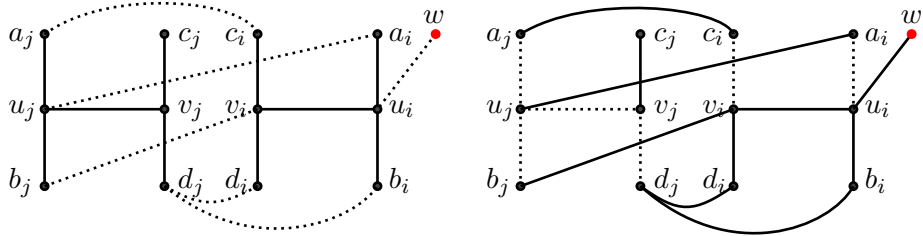

\centering
\tikzset{every picture/.style={line width=1pt}} 

\caption{Decomposition of $V(H_i \cup H_j) \cup \{w\}$ into  $\hat{H}$ and three pairwise disjoint edges.} 
\label{Fig:one-edge-rotation-2}
\end{figure}
    
    Note from Claim~\ref{CLAIM:one-edge-b} that there are already $5 = 36 - 31$ pairs in $V(H_i) \times V(H_j)$ that do not belong to $G[H_i, H_j]$. So all the remaining pairs in $V(H_i) \times \left(V(H_j) \cup \{w\} \right)$ are edges in $G[H_i, H_j]$. In particular, we have $\{a_iu_j, c_ia_j, v_ib_j\} \subseteq G$. 
    However, this implies that the set $V(H_i \cup H_j) \cup \{w\}$ can be covered by a copy of $\hat{H}$ and three pairwise disjoint edges, as shown on the right side of  Figure~\ref{Fig:one-edge-rotation-2}. Since the number of choices for $w$ is at least $19 \delta^{1/6} n$, by definition,  the pair $\{H_i, H_j\}$ is $1$-extendable, a contradiction.

    \medskip 

    \textbf{Case 2}: $\{a_i, b_i, c_i, d_i\} \cap L \neq\emptyset$. 
    
    By symmetry, we may assume that  $a_i \in L$. 
    Fix an arbitrary vertex $w \in N_{G}(a_i, \overline{U})$. 
    Note from the definition of $L$ that the number of choices for such a vertex $w$ is at least $19 \delta^{1/6} n$. 
    
\begin{figure}[H]
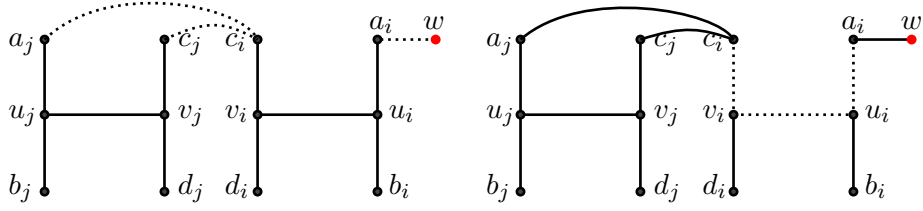

\centering
\tikzset{every picture/.style={line width=1pt}} 

\caption{Auxiliary figure for the proof of Claim~\ref{CLAIM:one-edge-c}.} 
\label{Fig:one-edge-rotation-4}
\end{figure}
    
    \begin{claim}\label{CLAIM:one-edge-c}
        The following statements hold. 
        \begin{enumerate}[label=(\roman*)]
            \item\label{CLAIM:one-edge-c-1} Either $\{c_ia_j, c_ib_j\} \cap G = \emptyset$ or $\{c_ic_j, c_id_j\} \cap G = \emptyset$.
            \item\label{CLAIM:one-edge-c-2} Either $\{d_ia_j, d_ib_j\} \cap G = \emptyset$ or $\{d_ic_j, d_id_j\} \cap G = \emptyset$.
        \end{enumerate}
    \end{claim}
    \begin{proof}[Proof of Claim~\ref{CLAIM:one-edge-c}]
        By symmetry, it suffices to show~\ref{CLAIM:one-edge-c-1}. 
        Suppose to the contrary that~\ref{CLAIM:one-edge-c-1} fails. 
        By symmetry,  we may assume that $\{c_ia_j, c_ic_j\} \subseteq G$. Then the set $V(H_i \cup H_j) \cup \{w\}$ can be covered by a copy of $\hat{H}$ and three pairwise disjoint edges, as shown on the right side of  Figure~\ref{Fig:one-edge-rotation-4}. Since the number of choices for $w$ is at least $19 \delta^{1/6} n$, by definition,  the pair $\{H_i, H_j\}$ is $1$-extendable, a contradiction. 
    \end{proof}

\begin{figure}[H]
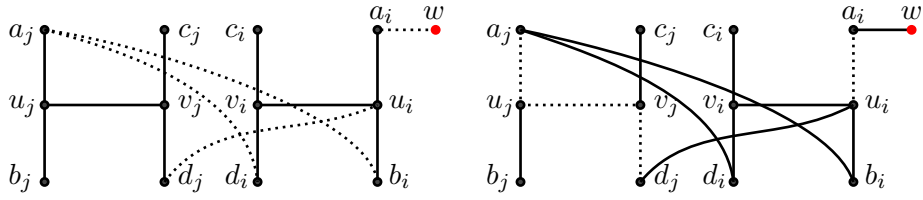

\centering
\tikzset{every picture/.style={line width=1pt}} 

\caption{Decomposition of $V(H_i \cup H_j) \cup \{w\}$ into  $\hat{H}$ and three pairwise disjoint edges.} 
\label{Fig:one-edge-rotation-3}
\end{figure}

    Note from Claim~\ref{CLAIM:one-edge-c} that there already at least four pairs in $V(H_i) \times V(H_j)$ that do not belong to $G[H_i, H_j]$. Therefore, $b_i$ is adjacent to at least three vertices in $\{a_j, b_j, c_j, d_j\}$, and by symmetry, we may assume that they are $\{a_j, b_j, c_j\}$. 
    For the same reason, the vertex $u_i$ has neighbors in both $\{a_j, b_j\}$ and $\{c_j, d_j\}$, and by symmetry, we may assume that $\{u_i b_j, u_id_j\} \subseteq G$. 
    
    Since $e(H_i, H_j) \ge 31 = 36-5$, there exists a vertex in $\{c_i, d_i\}$ that is adjacent to a vertex in $\{a_j, b_j, c_j\}$. By symmetry, we may assume that $d_i a_j \in G$. However, this implies that the set $V(H_i \cup H_j) \cup \{w\}$ can be covered by a copy of $\hat{H}$ and three pairwise disjoint edges, as shown on the right side of  Figure~\ref{Fig:one-edge-rotation-3}. Since the number of choices for $w$ is at least $19 \delta^{1/6} n$, by definition,  the pair $\{H_i, H_j\}$ is $1$-extendable, a contradiction. 
    This completes the proof of Lemma~\ref{LEMMA:H0-H1-upper-bound}. 
\end{proof}

\subsection{Upper bounds for $e(\mathcal{H}_1)$, $e(\mathcal{H}_1, \mathcal{H}_2)$, and $e(\mathcal{H}_1, \mathcal{H}_3)$}\label{SUBSEC:H1-upper-bound}
\begin{lemma}\label{LEMMA:H1-H1-upper-bound}
    Suppose that $\{H_i, H_j\} \subseteq \mathcal{H}'$ is a pair such that $|V(H_i) \cap L| \ge 1$ and $|V(H_j) \cap L| \ge 1$. 
    Then $e(H_i, H_j) \le 24$. 
    In particular, 
    \begin{align*}
        e(\mathcal{H}_1) 
        \le 24 \binom{y_1 n}{2} + \binom{6}{2}y_1 n
        \le 12 y_1^2 n^2 + 3y_1 n,  
    \end{align*}
    and $e(\mathcal{H}_1, \mathcal{H}_i) \le 24 y_1 y_i n^2$ for $i \in \{2,3\}$.
\end{lemma}
\begin{proof}[Proof of Lemma~\ref{LEMMA:H1-H1-upper-bound}]
    Fix a pair $\{H_i, H_j\} \subseteq \mathcal{H}'$ such that $|V(H_i) \cap L| \ge 1$ and $|V(H_j) \cap L| \ge 1$. 
    Suppose to the contrary that $e(H_i, H_j) \ge 25$. 
    Fix a vertex $x \in V(H_i) \cap L$ and a vertex $y \in V(H_j) \cap L$. 
    Fix an arbitrary pair of distinct vertices $(w_1, w_2) \in N_{G}(x, \overline{U}) \times N_{G}(y, \overline{U})$. 
    Note from the definition of $L$ that the number of such (unordered) pairs $\{w_1, w_2\}$ is at least 
    \begin{align}\label{equ:w1-w2-choices}
        \frac{1}{2} \cdot 19 \delta^{1/6} n \cdot (19 \delta^{1/6} n - 1) 
        \ge 19 \delta n^2. 
    \end{align}

    \medskip 
    
    \textbf{Case 1}: $(x, y) \in \{u_i, v_i\} \times \{u_j, v_j\}$.

    By symmetry, we may assume that $\{x,y\} = \{u_i, u_j\}$. 

\begin{figure}[H]
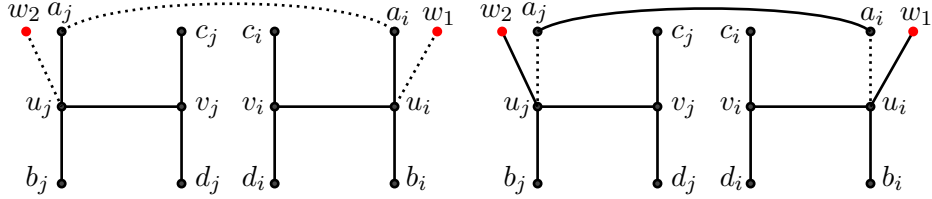

\centering
\tikzset{every picture/.style={line width=1pt}} 

\caption{Auxiliary figure for the proof of Claim~\ref{CLAIM:two-edges-j}.} 
\label{Fig:two-edges-rotation-16}
\end{figure}

    \begin{claim}\label{CLAIM:two-edges-j}
        No pair in $\{a_i, b_i\} \times \{a_j, b_j\}$ is an edge in $G$. 
    \end{claim}
    \begin{proof}[Proof of Claim~\ref{CLAIM:two-edges-j}]  
        Suppose to the contrary that this claim fails. By symmetry, we may assume that $a_i a_j \in G$. Then $G[W]$ contains two disjoint copies of $H$ and one edge, which together cover $14$ vertices (see Figure~\ref{Fig:two-edges-rotation-16}). 
        By~\eqref{equ:w1-w2-choices}, this means that $\{H_i, H_j\}$ is $2$-extendable, which is a contradiction.
    \end{proof} 

\begin{figure}[H]
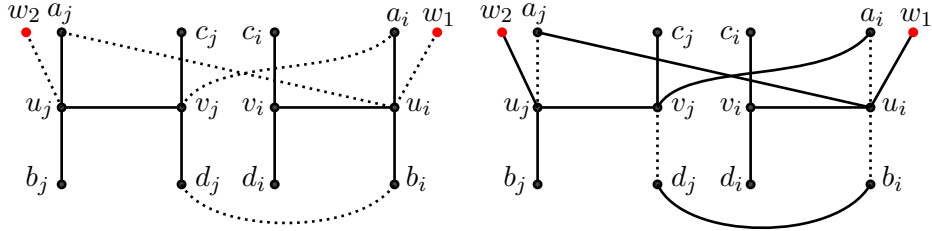

\centering
\tikzset{every picture/.style={line width=1pt}} 

\caption{Auxiliary figure for the proof of Claim~\ref{CLAIM:two-edges-k}.} 
\label{Fig:two-edges-rotation-17}
\end{figure}

    \begin{claim}\label{CLAIM:two-edges-k}
        The following statements hold. 
        \begin{enumerate}[label=(\roman*)]
        \item\label{CLAIM:two-edges-k-a} Either $\{u_i a_j, u_i b_j\} \cap G = \emptyset$, or $a_iv_j \not\in G$, or $\{b_i c_j, b_i d_j\} \cap G = \emptyset$.
        \item\label{CLAIM:two-edges-k-b} Either $\{u_i a_j, u_i b_j\} \cap G = \emptyset$, or $b_iv_j \not\in G$, or $\{a_i c_j, a_i d_j\} \cap G = \emptyset$. 
        \item\label{CLAIM:two-edges-k-c} Either $\{u_j a_i, u_j b_i\} \cap G = \emptyset$, or $a_jv_i \not\in G$, or $\{b_j c_i, b_j d_i\} \cap G = \emptyset$.
        \item\label{CLAIM:two-edges-k-d} Either $\{u_j a_i, u_j b_i\} \cap G = \emptyset$, or $b_jv_i \not\in G$, or $\{a_j c_i, a_j d_i\} \cap G = \emptyset$. 
        \end{enumerate}
    \end{claim}
    \begin{proof}[Proof of Claim~\ref{CLAIM:two-edges-k}]
        By symmetry, it suffices to prove~\ref{CLAIM:two-edges-k-a}. 
        Suppose to the contrary that it fails. By symmetry, we may assume that $\{u_i a_j, a_i v_j, b_i d_j\} \subseteq G$. 
        Then $G[W]$ contains two disjoint copies of $H$ and one disjoint edge, which together cover $14$ vertices (see Figure~\ref{Fig:two-edges-rotation-17}). By~\eqref{equ:w1-w2-choices}, this means that $\{H_i, H_j\}$ is $2$-extendable, which is a contradiction.
    \end{proof} 

\begin{figure}[H]
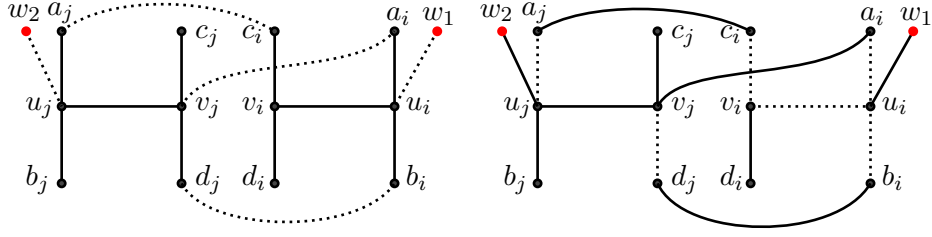

\centering
\tikzset{every picture/.style={line width=1pt}} 

\caption{Auxiliary figure for the proof of Claim~\ref{CLAIM:two-edges-l}.} 
\label{Fig:two-edges-rotation-18}
\end{figure}

    \begin{claim}\label{CLAIM:two-edges-l}
        The following statements hold. 
        \begin{enumerate}[label=(\roman*)]
        \item\label{CLAIM:two-edges-l-a} Either no pair in $\{c_i, d_i\} \times \{a_j, b_j\}$ is an edge in $G$, or $a_iv_j \not\in G$, or $\{b_i c_j, b_i d_j\} \cap G = \emptyset$. 
        \item\label{CLAIM:two-edges-l-b} Either no pair in $\{c_i, d_i\} \times \{a_j, b_j\}$ is an edge in $G$, or $b_iv_j \not\in G$, or $\{a_i c_j, a_i d_j\} \cap G = \emptyset$. 
        \item\label{CLAIM:two-edges-l-c} Either no pair in $\{c_j, d_j\} \times \{a_i, b_i\}$ is an edge in $G$, or $a_jv_i \not\in G$, or $\{b_j c_i, b_j d_i\} \cap G = \emptyset$.
        \item\label{CLAIM:two-edges-l-d} Either no pair in $\{c_j, d_j\} \times \{a_i, b_i\}$ is an edge in $G$, or $b_jv_i \not\in G$, or $\{a_j c_i, a_j d_i\} \cap G = \emptyset$. 
        \end{enumerate} 
    \end{claim}
    \begin{proof}[Proof of Claim~\ref{CLAIM:two-edges-l}]
        By symmetry, it suffices to prove~\ref{CLAIM:two-edges-l-a}. 
        Suppose to the contrary that it fails. By symmetry, we may assume that $\{c_i a_j, a_i v_j, b_i d_j\} \subseteq G$. 
        Then $G[W]$ contains a copy of $H$ and four disjoint edges, which together cover $14$ vertices (see Figure~\ref{Fig:two-edges-rotation-18}). By~\eqref{equ:w1-w2-choices}, this means that $\{H_i, H_j\}$ is $2$-extendable, which is a contradiction.
    \end{proof} 

\begin{figure}[H]
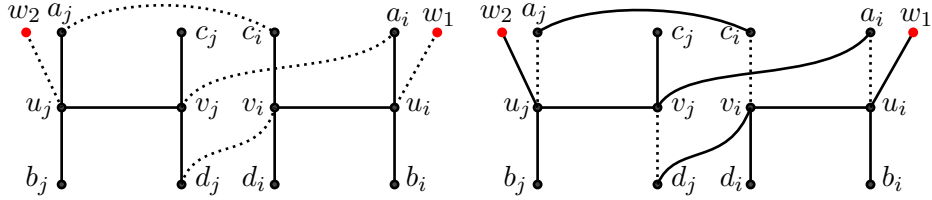

\centering
\tikzset{every picture/.style={line width=1pt}} 

\caption{Auxiliary figure for the proof of Claim~\ref{CLAIM:two-edges-m}.} 
\label{Fig:two-edges-rotation-19}
\end{figure}

    \begin{claim}\label{CLAIM:two-edges-m}
        The following statements hold. 
        \begin{enumerate}[label=(\roman*)]
            \item\label{CLAIM:two-edges-m-a} Either no pair in $\{c_i, d_i\} \times \{a_j, b_j\}$ is an edge in $G$, or $\{a_i v_j, b_i v_j\} \cap G = \emptyset$, or $\{v_i c_j, v_i d_j\} \cap G = \emptyset$. 
            \item\label{CLAIM:two-edges-m-b} Either no pair in $\{c_j, d_j\} \times \{a_i, b_i\}$ is an edge in $G$, or $\{a_j v_i, b_j v_i\} \cap G = \emptyset$, or $\{v_j c_i, v_j d_i\} \cap G = \emptyset$. 
        \end{enumerate}
    \end{claim}
    \begin{proof}[Proof of Claim~\ref{CLAIM:two-edges-m}]
        By symmetry, it suffices to prove~\ref{CLAIM:two-edges-m-a}. 
        Suppose to the contrary that it fails. By symmetry, we may assume that $\{c_i a_j, a_i v_j, v_i d_j\} \subseteq G$. 
        Then $G[W]$ contains two disjoint copies of $H$ and one disjoint edge, which together cover $14$ vertices (see Figure~\ref{Fig:two-edges-rotation-19}). By~\eqref{equ:w1-w2-choices}, this means that $\{H_i, H_j\}$ is $2$-extendable, which is a contradiction.
    \end{proof} 

\begin{figure}[H]
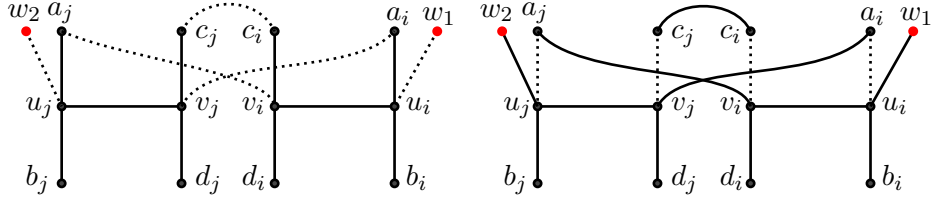

\centering
\tikzset{every picture/.style={line width=1pt}} 

\caption{Auxiliary figure for the proof of Claim~\ref{CLAIM:two-edges-n}.} 
\label{Fig:two-edges-rotation-20}
\end{figure}

    \begin{claim}\label{CLAIM:two-edges-n}
        Either no pair in $\{c_i, d_i\} \times \{c_j, d_j\}$ is an edge in $G$, or $\{a_i v_j, b_i v_j\} \cap G = \emptyset$, or $\{v_i a_j, v_i b_j\} \cap G = \emptyset$. 
    \end{claim}
    \begin{proof}[Proof of Claim~\ref{CLAIM:two-edges-n}]
        Suppose to the contrary that this claim fails. By symmetry, we may assume that $\{c_i c_j, a_i v_j, v_i a_j\} \subseteq G$. 
        Then $G[W]$ contains two disjoint copies of $H$ and one disjoint edge, which together cover $14$ vertices (see Figure~\ref{Fig:two-edges-rotation-20}). By~\eqref{equ:w1-w2-choices}, this means that $\{H_i, H_j\}$ is $2$-extendable, which is a contradiction.
    \end{proof} 

\begin{figure}[H]
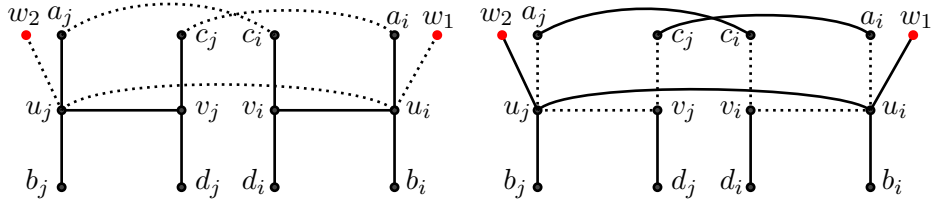

\centering
\tikzset{every picture/.style={line width=1pt}} 

\caption{Auxiliary figure for the proof of Claim~\ref{CLAIM:two-edges-o}.} 
\label{Fig:two-edges-rotation-21}
\end{figure}

    \begin{claim}\label{CLAIM:two-edges-o}
        Either no pair in $\{a_i, b_i\} \times \{c_j, d_j\}$ is an edge in $G$, or no pair in $\{c_i, d_i\} \times \{a_j, b_j\}$ is an edge in $G$, or $u_i u_j \not\in G$. 
    \end{claim}
    \begin{proof}[Proof of Claim~\ref{CLAIM:two-edges-o}]
        Suppose to the contrary that this claim fails. By symmetry, we may assume that $\{a_i c_j, c_i a_j, u_i u_j\} \subseteq G$. 
        Then $G[W]$ contains a copy of $H$ and four disjoint edges, which together cover $14$ vertices (see Figure~\ref{Fig:two-edges-rotation-21}). By~\eqref{equ:w1-w2-choices}, this means that $\{H_i, H_j\}$ is $2$-extendable, which is a contradiction. 
    \end{proof} 

    \begin{claim}\label{CLAIM:two-edges-p}
        The following statements hold. 
        \begin{enumerate}[label=(\roman*)]
            \item\label{CLAIM:two-edges-p-1} $\{a_i v_j, b_i v_j\} \cap G = \emptyset$. 
            \item\label{CLAIM:two-edges-p-2} $\{v_i a_j, v_i b_j\} \cap G = \emptyset$. 
        \end{enumerate} 
    \end{claim}
    \begin{proof}[Proof of Claim~\ref{CLAIM:two-edges-p}]
        Note from the assumption $e(H_i, H_j) \ge 25$ that there are at most $36 - 25 = 11$ pairs in $V(H_i) \times V(H_j)$ that are not edges in $G$ (let us call them missing edges). 
        If this claim fails, then by applying Claims~\ref{CLAIM:two-edges-j},~\ref{CLAIM:two-edges-k},~\ref{CLAIM:two-edges-l},~\ref{CLAIM:two-edges-m},~\ref{CLAIM:two-edges-n}, and~\ref{CLAIM:two-edges-o}, one could verify (albeit somewhat tediously) that there would be at least $12$ missing edges, which is a contradiction. 
    \end{proof} 

    \begin{claim}\label{CLAIM:two-edges-q} 
        We have $\{u_i u_j, v_i v_j\} \cap G = \emptyset$.   
    \end{claim}
    \begin{proof}[Proof of Claim~\ref{CLAIM:two-edges-q}]
        Let us first prove that $u_i u_j \not\in G$. Suppose to the contrary that $u_i u_j \in G$. 
        Then by Claim~\ref{CLAIM:two-edges-j},~\ref{CLAIM:two-edges-o}, and~\ref{CLAIM:two-edges-p}, there will be at least $12 > 36 - e(H_i, H_j)$ pairs in $V(H_i) \times V(H_j)$ that are not edges in $G$, which is a contradiction. 

\begin{figure}[H]
\centering
\tikzset{every picture/.style={line width=1pt}} 

\caption{Auxiliary figure for the proof of Claim~\ref{CLAIM:two-edges-q}.} 
\label{Fig:two-edges-rotation-22}
\end{figure}

        We now prove that $v_i v_j \not\in G$. 
        First, suppose for contradiction that $\{c_i a_j, c_i b_j\} \cap G = \emptyset$. 
        Then combined with the fact $u_i u_j \not\in G$ (as established in the argument above) and Claims~\ref{CLAIM:two-edges-j}~\ref{CLAIM:two-edges-p}, we would have $11 \ge 36 - e(H_i, H_j)$ missing edges. This implies that all other pairs in $V(H_i) \times V(H_j)$ are edges in $G$. In particular, $\{d_iu_j, d_ib_j, d_ju_i, d_jb_i, c_ic_j\} \subseteq G$. 
        Then $G[W]$ would contain two disjoint copies of $H$ and one disjoint edge, which together cover $14$ vertices (see Figure~\ref{Fig:two-edges-rotation-22}). By~\eqref{equ:w1-w2-choices}, this means that $\{H_i, H_j\}$ is $2$-extendable, which is a contradiction. Therefore, we have $\{c_i a_j, c_i b_j\} \cap G \neq \emptyset$. By symmetry, we have
        \begin{align*}
            \{d_i a_j, d_i b_j\} \cap G \neq \emptyset,\quad
            \{a_j c_i, a_j d_i\} \cap G \neq \emptyset,\quad 
            \{b_j c_i, b_j d_i\} \cap G \neq \emptyset, 
        \end{align*}
        and 
        \begin{align*}
            \{c_j a_i, c_j b_i\} \cap G \neq \emptyset, \quad 
            \{d_j a_i, d_j b_i\} \cap G \neq \emptyset, \quad 
            \{a_i c_j, a_i d_j\} \cap G \neq \emptyset,\quad 
            \{b_i c_j, b_i d_j\} \cap G \neq \emptyset. 
        \end{align*}
        It follows that both induced bipartite graphs $G[\{c_i, d_i\}, \{a_j, b_j\}]$ and $G[\{c_j, d_j\}, \{a_i, b_i\}]$ contain a matching of size two. By symmetry, we may assume that they are $\{c_ia_j, d_i b_j\}$ and $\{a_i c_j, b_i d_j\}$, respectively. In other words, $\{c_i a_j, d_i b_j, a_i c_j, b_i d_j\} \subseteq G$. 

\begin{figure}[H]
\centering
\tikzset{every picture/.style={line width=1pt}} 

\caption{Auxiliary figure for the proof of Claim~\ref{CLAIM:two-edges-q}.} 
\label{Fig:two-edges-rotation-23}
\end{figure}
        
        Now suppose to the contrary that $v_i v_j \in G$. Then $G[W]$ would contain seven pairwise disjoint edges, which together cover $14$ vertices (see Figure~\ref{Fig:two-edges-rotation-23}). By~\eqref{equ:w1-w2-choices}, this means that $\{H_i, H_j\}$ is $2$-extendable, which is a contradiction.
    \end{proof} 

\begin{figure}[H]
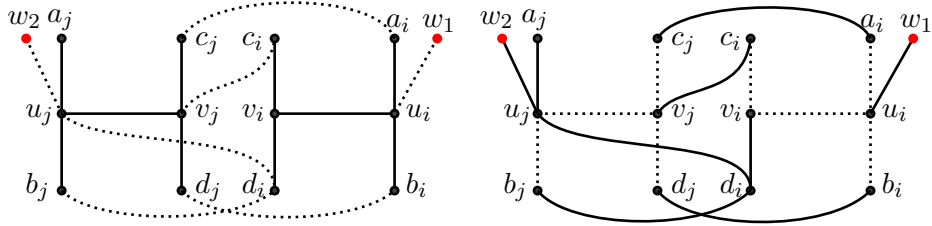

\centering
\tikzset{every picture/.style={line width=1pt}} 

\caption{Decomposition of $W$ into a copy of $H$ and four disjoint edges.} 
\label{Fig:two-edges-rotation-24}
\end{figure}

    Notice that Claims~\ref{CLAIM:two-edges-j},~\ref{CLAIM:two-edges-p}, and~\ref{CLAIM:two-edges-q} already contribute $10$ missing edges. Since $e(H_i, H_j) \ge 25 = 36 - 10 - 1$, there must be a matching of size two in $\{u_j, v_j\} \times \{c_i, d_i\}$, and also in $\{a_i, b_i\} \times \{c_j, d_j\}$. Similarly, $d_i$ has neighbor in $\{a_j, b_j\}$. By symmetry, we may assume that $\{c_iv_j, d_iu_j, a_i c_j, b_i d_j, d_i b_j\} \subseteq G$. 
    Then $G[W]$ contains a copy of $H$ and four disjoint edges, which together cover $14$ vertices (see Figure~\ref{Fig:two-edges-rotation-24}). By~\eqref{equ:w1-w2-choices}, this means that $\{H_i, H_j\}$ is $2$-extendable, which is a contradiction.
    
    \medskip

    \textbf{Case 2}: $(x, y) \in \{a_i, b_i, c_i, d_i\} \times \{a_j, b_j, c_j, d_j\}$. 

    By symmetry, we may assume that $\{x,y\} = \{a_i, a_j\}$. 

\begin{figure}[H]
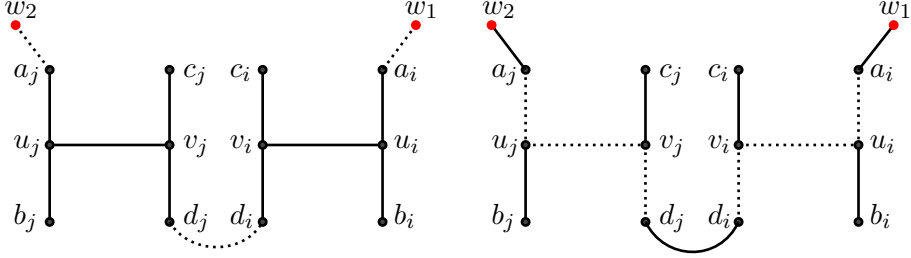

\centering
\tikzset{every picture/.style={line width=1pt}} 

\caption{Auxiliary figure for the proof of Claim~\ref{CLAIM:two-edges-r}.} 
\label{Fig:two-edges-rotation-25}
\end{figure}

    \begin{claim}\label{CLAIM:two-edges-r}
        No pair in $\{c_i, d_i\} \times \{c_j, d_j\}$ is an edge in $G$. 
    \end{claim}
    \begin{proof}[Proof of Claim~\ref{CLAIM:two-edges-r}]
        Suppose to the contrary that this claim fails. By symmetry, we may assume that $d_i d_j \in G$. 
        Then $G[W]$ would contain seven disjoint edges, which together cover $14$ vertices (see Figure~\ref{Fig:two-edges-rotation-25}). By~\eqref{equ:w1-w2-choices}, this means that $\{H_i, H_j\}$ is $2$-extendable, which is a contradiction. 
    \end{proof} 

    \begin{claim}\label{CLAIM:two-edges-s}
        The following statements hold. 
        \begin{enumerate}[label=(\roman*)]
            \item\label{CLAIM:two-edges-s-1} $\{u_i c_j, u_i d_j\} \cap G = \emptyset$. 
            \item\label{CLAIM:two-edges-s-2} $\{u_j c_i, u_j d_i\} \cap G = \emptyset$. 
            \item\label{CLAIM:two-edges-s-3} $v_i v_j \not\in G$.
        \end{enumerate}  
    \end{claim}
    \begin{proof}[Proof of Claim~\ref{CLAIM:two-edges-s}]
\begin{figure}[H]
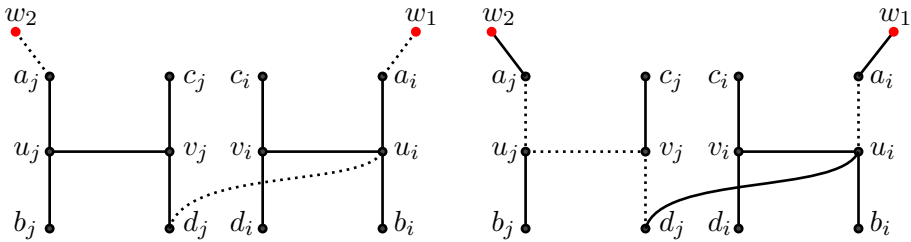

\centering
\tikzset{every picture/.style={line width=1pt}} 

\caption{Auxiliary figure for the proof of Claim~\ref{CLAIM:two-edges-s}~\ref{CLAIM:two-edges-s-1} and~\ref{CLAIM:two-edges-s-2}.} 
\label{Fig:two-edges-rotation-26}
\end{figure}
        
        Let us first prove~\ref{CLAIM:two-edges-s-1}. Suppose to the contrary that~\ref{CLAIM:two-edges-s-1} fails. By symmetry, we may assume that $u_i d_j \in G$. 
        Then $G[W]$ contains a copy of $H$ and four disjoint edges, which together cover $14$ vertices (see Figure~\ref{Fig:two-edges-rotation-26}). 
        By~\eqref{equ:w1-w2-choices}, this means that $\{H_i, H_j\}$ is $2$-extendable, which is a contradiction.
        
        By symmetry,~\ref{CLAIM:two-edges-s-2} holds as well. 

\begin{figure}[H]
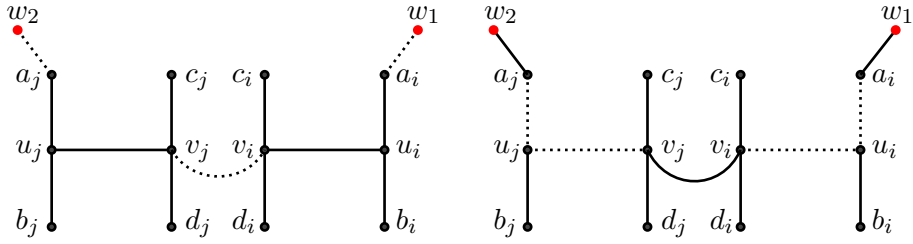

\centering
\tikzset{every picture/.style={line width=1pt}} 

\caption{Auxiliary figure for the proof of Claim~\ref{CLAIM:two-edges-s}~\ref{CLAIM:two-edges-s-3}.} 
\label{Fig:two-edges-rotation-27}
\end{figure}

        Next, we prove~\ref{CLAIM:two-edges-s-3}. Suppose to the contrary that $v_i v_j \in G$. Then $G[W]$ contains a copy of $H$ and four disjoint edges, which together cover $14$ vertices (see Figure~\ref{Fig:two-edges-rotation-27}). 
        By~\eqref{equ:w1-w2-choices}, this means that $\{H_i, H_j\}$ is $2$-extendable, which is a contradiction.
    \end{proof} 

\begin{figure}[H]
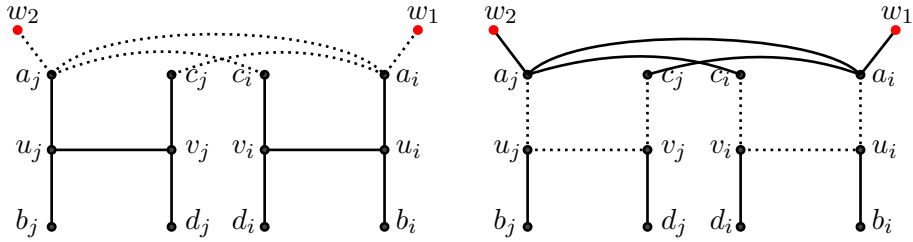

\centering
\tikzset{every picture/.style={line width=1pt}} 

\caption{Auxiliary figure for the proof of Claim~\ref{CLAIM:two-edges-t}.} 
\label{Fig:two-edges-rotation-28}
\end{figure}

    \begin{claim}\label{CLAIM:two-edges-t}
        Either $\{a_i c_j, a_i d_j\} \cap G = \emptyset$, or $\{a_j c_i, a_j d_i\} \cap G = \emptyset$, or $a_i a_j \not\in G$.
    \end{claim}
    \begin{proof}[Proof of Claim~\ref{CLAIM:two-edges-t}]
        Suppose to the contrary that this claim fails. By symmetry, we may assume that $\{a_i a_j, a_i c_j, a_j c_i\} \subseteq G$. 
        Then $G[W]$ contains a copy of $H$ and four disjoint edges, which together cover $14$ vertices (see Figure~\ref{Fig:two-edges-rotation-28}). By~\eqref{equ:w1-w2-choices}, this means that $\{H_i, H_j\}$ is $2$-extendable, which is a contradiction. 
    \end{proof} 

\begin{figure}[H]
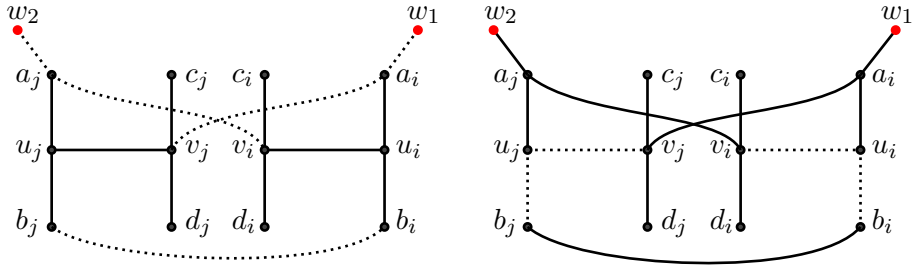

\centering
\tikzset{every picture/.style={line width=1pt}} 

\caption{Auxiliary figure for the proof of Claim~\ref{CLAIM:two-edges-u}.} 
\label{Fig:two-edges-rotation-29}
\end{figure}

    \begin{claim}\label{CLAIM:two-edges-u}
        At least one pair in $\{a_i v_j, v_i a_j, b_i b_j\}$ is not an edge in $G$. 
    \end{claim}
    \begin{proof}[Proof of Claim~\ref{CLAIM:two-edges-u}]
        Suppose to the contrary that $\{a_i v_j, v_i a_j, b_i b_j\} \subseteq G$. 
        Then $G[W]$ would contain two disjoint copies of $H$ and one disjoint edge, which together cover $14$ vertices (see Figure~\ref{Fig:two-edges-rotation-29}). By~\eqref{equ:w1-w2-choices}, this means that $\{H_i, H_j\}$ is $2$-extendable, which is a contradiction.
    \end{proof} 

\begin{figure}[H]
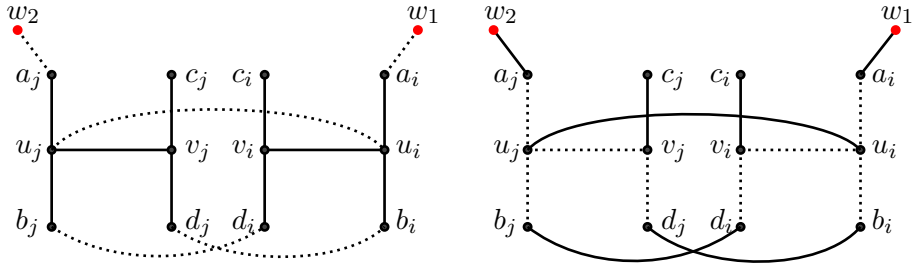

\centering
\tikzset{every picture/.style={line width=1pt}} 

\caption{Auxiliary figure for the proof of Claim~\ref{CLAIM:two-edges-v}.} 
\label{Fig:two-edges-rotation-30}
\end{figure}

    \begin{claim}\label{CLAIM:two-edges-v}
        Either $\{b_i c_j, b_i d_j\} \cap G = \emptyset$, or $\{b_j c_i, b_j d_i\} \cap G = \emptyset$, or $u_i u_j \not\in G$. 
    \end{claim}
    \begin{proof}[Proof of Claim~\ref{CLAIM:two-edges-v}]
        Suppose to the contrary that this claim fails. By symmetry, we may assume that $\{u_i u_j, b_i d_j, b_j d_i\} \subseteq G$. 
        Then $G[W]$ contains seven disjoint edges, which together cover $14$ vertices (see Figure~\ref{Fig:two-edges-rotation-30}). By~\eqref{equ:w1-w2-choices}, this means that $\{H_i, H_j\}$ is $2$-extendable, which is a contradiction. 
    \end{proof} 

    Note that Claims~\ref{CLAIM:two-edges-r},~\ref{CLAIM:two-edges-s},~\ref{CLAIM:two-edges-t},~\ref{CLAIM:two-edges-u}, and~\ref{CLAIM:two-edges-v} contribute at least $12 > 36 - e(H_i, H_j)$ missing edges, which is a contradiction.

    \medskip

    \textbf{Case 3}: $(x,y) \in \{u_i, v_i\} \times \{a_j, b_j, c_j, d_j\}$  or $(x,y) \in \{u_j, v_j\} \times \{a_i, b_i, c_i, d_i\}$. 

    By symmetry, we may assume that  $\{x,y\} = \{u_i, a_j\}$. 
    
    \begin{claim}\label{CLAIM:two-edges-w}
        The following statements hold. 
        \begin{enumerate}[label=(\roman*)]
            \item\label{CLAIM:two-edges-w-1} No pair in $\{a_i, b_i\} \times \{c_j, d_j\}$ is an edge in $G$. 
            \item\label{CLAIM:two-edges-w-2} $\{u_j a_i, u_j b_i\} \cap G = \emptyset$. 
        \end{enumerate} 
    \end{claim}
    \begin{proof}[Proof of Claim~\ref{CLAIM:two-edges-w}]
\begin{figure}[H]
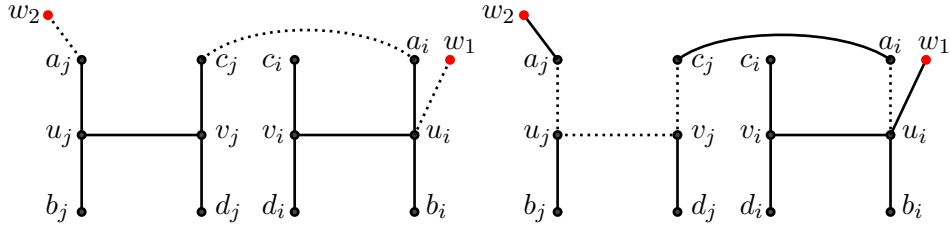

\centering
\tikzset{every picture/.style={line width=1pt}} 

\caption{Auxiliary figure for the proof of Claim~\ref{CLAIM:two-edges-w}~\ref{CLAIM:two-edges-w-1}.} 
\label{Fig:two-edges-rotation-31}
\end{figure}
        
        Let us first prove~\ref{CLAIM:two-edges-w-1}. Suppose to the contrary that~\ref{CLAIM:two-edges-w-1} fails. By symmetry, we may assume that $a_i c_j \in G$. 
        Then $G[W]$ contains a copy of $H$ and four disjoint edges, which together cover $14$ vertices (see Figure~\ref{Fig:two-edges-rotation-31}). By~\eqref{equ:w1-w2-choices}, this means that $\{H_i, H_j\}$ is $2$-extendable, which is a contradiction. 

\begin{figure}[H]
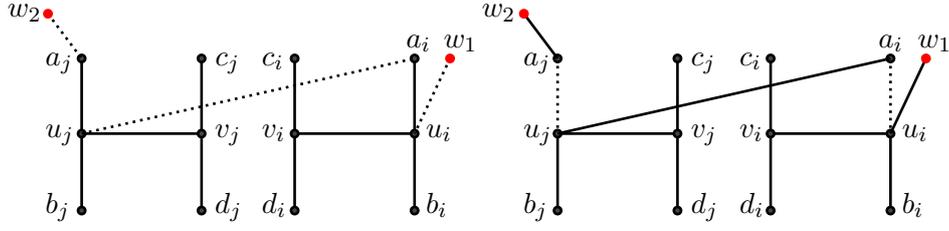

\centering
\tikzset{every picture/.style={line width=1pt}} 

\caption{Auxiliary figure for the proof of Claim~\ref{CLAIM:two-edges-w}~\ref{CLAIM:two-edges-w-2}.} 
\label{Fig:two-edges-rotation-32}
\end{figure}

        Next, we prove~\ref{CLAIM:two-edges-w-2}. Suppose to the contrary that~\ref{CLAIM:two-edges-w-2} fails. By symmetry, we may assume that $a_i u_j \in G$. Then $G[W]$ contains two disjoint copies of $H$ and one disjoint edge, which together cover $14$ vertices (see Figure~\ref{Fig:two-edges-rotation-32}). 
        By~\eqref{equ:w1-w2-choices}, this means that $\{H_i, H_j\}$ is $2$-extendable, which is a contradiction.
    \end{proof} 

\begin{figure}[H]
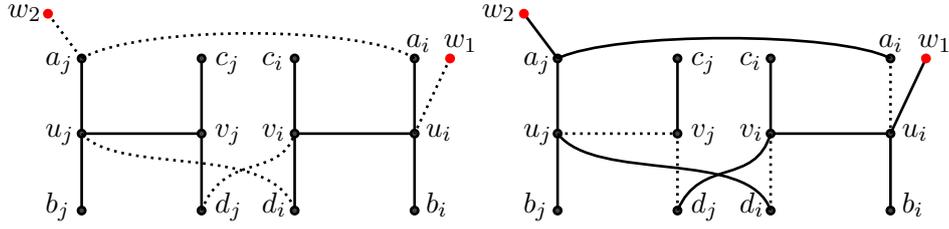

\centering
\tikzset{every picture/.style={line width=1pt}} 

\caption{Auxiliary figure for the proof of Claim~\ref{CLAIM:two-edges-x}.} 
\label{Fig:two-edges-rotation-33}
\end{figure}

    \begin{claim}\label{CLAIM:two-edges-x}
        Either $\{v_i c_j, v_i d_j\} \cap G = \emptyset$, or $\{u_j c_i, u_j d_i\} \cap G = \emptyset$, or $\{a_j a_i, a_j b_i\} \cap G = \emptyset$. 
    \end{claim}
    \begin{proof}[Proof of Claim~\ref{CLAIM:two-edges-x}]
        Suppose to the contrary that this claim fails. By symmetry, we may assume that $\{v_i d_j, d_i u_j, a_i a_j\} \subseteq G$. 
        Then $G[W]$ contains two disjoint copies of $H$ and one disjoint edge, which together cover $14$ vertices (see Figure~\ref{Fig:two-edges-rotation-33}). By~\eqref{equ:w1-w2-choices}, this means that $\{H_i, H_j\}$ is $2$-extendable, which is a contradiction.
    \end{proof} 

\begin{figure}[H]
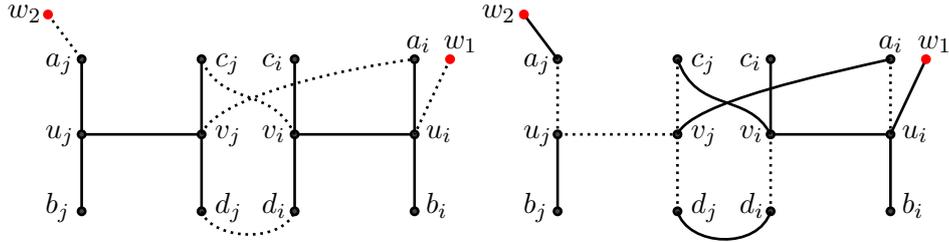

\centering
\tikzset{every picture/.style={line width=1pt}} 

\caption{Auxiliary figure for the proof of Claim~\ref{CLAIM:two-edges-y}.} 
\label{Fig:two-edges-rotation-34}
\end{figure}

    \begin{claim}\label{CLAIM:two-edges-y}
        The following statements hold. 
        \begin{enumerate}[label=(\roman*)]
            \item\label{CLAIM:two-edges-y-1} Either $v_i c_j \not\in G$, or $\{c_i d_j, d_id_j\} \cap G = \emptyset$, or $\{v_j v_i, a_j b_i\} \cap G = \emptyset$.
            \item\label{CLAIM:two-edges-y-2} Either $v_i d_j \not\in G$, or $\{c_ic_j, d_ic_j\} \cap G = \emptyset$, or $\{v_j v_i, a_j b_i\} \cap G = \emptyset$.
        \end{enumerate} 
    \end{claim}
    \begin{proof}[Proof of Claim~\ref{CLAIM:two-edges-y}]
        By symmetry, it suffices to prove~\ref{CLAIM:two-edges-y-1}. 
        Suppose to the contrary that this claim fails. By symmetry, we may assume that $\{v_i c_j, d_i d_j, a_i v_j\} \subseteq G$. 
        Then $G[W]$ contains a copy of $H$ and four disjoint edges, which together cover $14$ vertices (see Figure~\ref{Fig:two-edges-rotation-34}). By~\eqref{equ:w1-w2-choices}, this means that $\{H_i, H_j\}$ is $2$-extendable, which is a contradiction. 
    \end{proof} 

\begin{figure}[H]
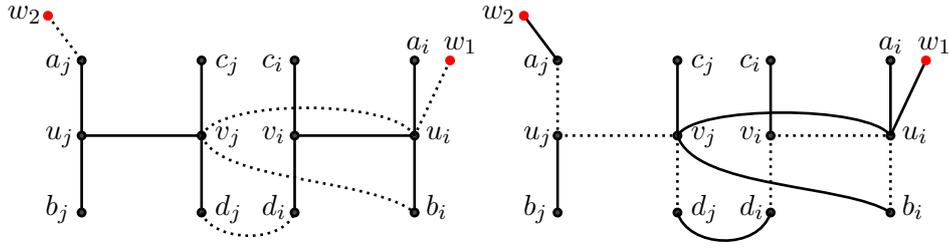

\centering
\tikzset{every picture/.style={line width=1pt}} 

\caption{Auxiliary figure for the proof of Claim~\ref{CLAIM:two-edges-z}.} 
\label{Fig:two-edges-rotation-35}
\end{figure}

    \begin{claim}\label{CLAIM:two-edges-z}
        Either $u_i v_j \not\in G$, or $\{a_i v_j, b_i v_j\} \cap G = \emptyset$, or no pair in $\{c_i, d_i\} \times \{c_j, d_j\}$ is an edge in $G$. 
    \end{claim}
    \begin{proof}[Proof of Claim~\ref{CLAIM:two-edges-z}]
        Suppose to the contrary that this claim fails. By symmetry, we may assume that $\{u_i v_j, b_i v_j, d_i d_j\} \subseteq G$. 
        Then $G[W]$ contains a copy of $H$ and four disjoint edges, which together cover $14$ vertices (see Figure~\ref{Fig:two-edges-rotation-35}). By~\eqref{equ:w1-w2-choices}, this means that $\{H_i, H_j\}$ is $2$-extendable, which is a contradiction.
    \end{proof} 

\begin{figure}[H]
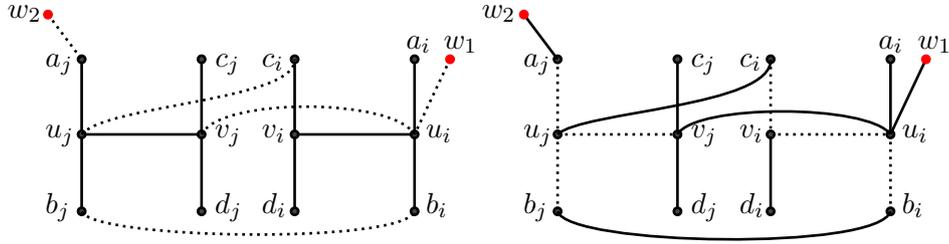

\centering
\tikzset{every picture/.style={line width=1pt}} 

\caption{Auxiliary figure for the proof of Claim~\ref{CLAIM:two-edges-A}.} 
\label{Fig:two-edges-rotation-36}
\end{figure}

    \begin{claim}\label{CLAIM:two-edges-A}
        Either $u_i v_j \not\in G$, or $\{c_i u_j, d_i u_j\} \cap G = \emptyset$, or $\{a_i b_j, b_i b_j\} \cap G = \emptyset$. 
    \end{claim}
    \begin{proof}[Proof of Claim~\ref{CLAIM:two-edges-A}]
        Suppose to the contrary that this claim fails. By symmetry, we may assume that $\{u_i v_j, c_i u_j, b_i b_j\} \subseteq G$. 
        Then $G[W]$ contains a copy of $H$ and four disjoint edges, which together cover $14$ vertices (see Figure~\ref{Fig:two-edges-rotation-36}). By~\eqref{equ:w1-w2-choices}, this means that $\{H_i, H_j\}$ is $2$-extendable, which is a contradiction. 
    \end{proof} 

\begin{figure}[H]
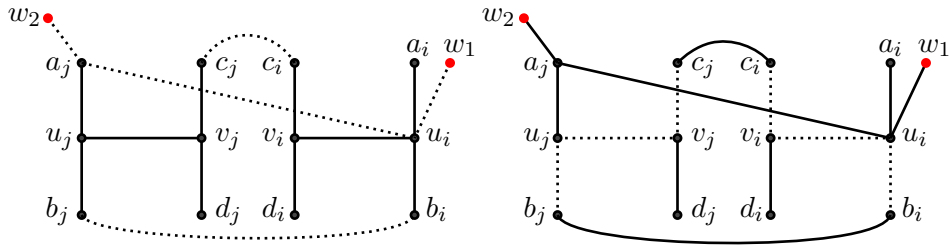

\centering
\tikzset{every picture/.style={line width=1pt}} 

\caption{Auxiliary figure for the proof of Claim~\ref{CLAIM:two-edges-B}.} 
\label{Fig:two-edges-rotation-37}
\end{figure}

    \begin{claim}\label{CLAIM:two-edges-B}
        Either $u_i a_j \not\in G$, or $\{a_i b_j, b_i b_j\} \cap G = \emptyset$, no pair in $\{c_i, d_i\} \times \{c_j, d_j\}$ is an edge in $G$.
    \end{claim}
    \begin{proof}[Proof of Claim~\ref{CLAIM:two-edges-B}]
        Suppose to the contrary that this claim fails. By symmetry, we may assume that $\{u_i a_j, c_i c_j, b_i b_j\} \subseteq G$. 
        Then $G[W]$ contains a copy of $H$ and four disjoint edges, which together cover $14$ vertices (see Figure~\ref{Fig:two-edges-rotation-37}). By~\eqref{equ:w1-w2-choices}, this means that $\{H_i, H_j\}$ is $2$-extendable, which is a contradiction. 
    \end{proof} 

\begin{figure}[H]
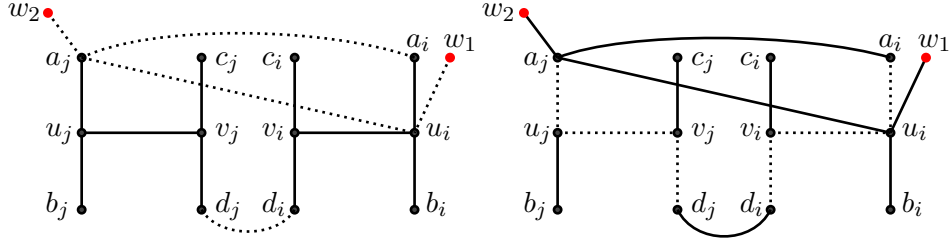

\centering
\tikzset{every picture/.style={line width=1pt}} 

\caption{Auxiliary figure for the proof of Claim~\ref{CLAIM:two-edges-C}.} 
\label{Fig:two-edges-rotation-38}
\end{figure}

    \begin{claim}\label{CLAIM:two-edges-C}
        Either $u_i a_j \not\in G$, or $\{a_i a_j, b_i a_j\} \cap G = \emptyset$, or no pair in $\{c_i, d_i\} \times \{c_j, d_j\}$ is an edge in $G$. 
    \end{claim}
    \begin{proof}[Proof of Claim~\ref{CLAIM:two-edges-C}]
        Suppose to the contrary that this claim fails. By symmetry, we may assume that $\{u_i a_j, d_i d_j, a_i a_j\} \subseteq G$. 
        Then $G[W]$ contains a copy of $H$ and four disjoint edges, which together cover $14$ vertices (see Figure~\ref{Fig:two-edges-rotation-38}). By~\eqref{equ:w1-w2-choices}, this means that $\{H_i, H_j\}$ is $2$-extendable, which is a contradiction. 
    \end{proof} 

\begin{figure}[H]
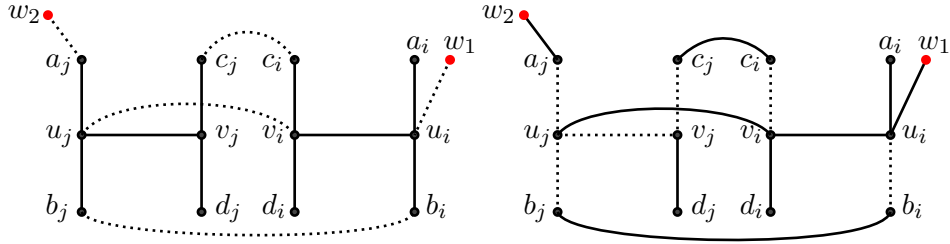

\centering
\tikzset{every picture/.style={line width=1pt}} 

\caption{Auxiliary figure for the proof of Claim~\ref{CLAIM:two-edges-D}.} 
\label{Fig:two-edges-rotation-39}
\end{figure}

    \begin{claim}\label{CLAIM:two-edges-D}
        Either $v_i u_j \not\in G$, or $\{a_i b_j, b_i b_j\} \cap G = \emptyset$, or no pair in $\{c_i, d_i\} \times \{c_j, d_j\}$ is an edge in $G$. 
    \end{claim}
    \begin{proof}[Proof of Claim~\ref{CLAIM:two-edges-D}]
        Suppose to the contrary that this claim fails. By symmetry, we may assume that $\{v_i u_j, c_i c_j, b_i b_j\} \subseteq G$. 
        Then $G[W]$ contains a copy of $H$ and four disjoint edges, which together cover $14$ vertices (see Figure~\ref{Fig:two-edges-rotation-39}). By~\eqref{equ:w1-w2-choices}, this means that $\{H_i, H_j\}$ is $2$-extendable, which is a contradiction. 
    \end{proof} 

    Therefore, by the claims above, we have the following results.

    \begin{claim}\label{CLAIM:two-edges-E}
        The following statements hold. 
        \begin{enumerate}[label=(\roman*)]
            \item\label{CLAIM:two-edges-E-1} $\{v_i c_j, v_i d_j\} \cap G = \emptyset$. 
            \item\label{CLAIM:two-edges-E-2} $u_i v_j \not\in G$. 
            \item\label{CLAIM:two-edges-E-3} $u_i a_j \not\in G$. 
            \item\label{CLAIM:two-edges-E-4} $v_i u_j \not\in G$. 
        \end{enumerate} 
    \end{claim}
    \begin{proof}[Proof of Claim~\ref{CLAIM:two-edges-E}]
        Let us first prove~\ref{CLAIM:two-edges-E-1}. Suppose to the contrary that~\ref{CLAIM:two-edges-E-1} fails. By symmetry, we may assume that $v_i c_j \in G$. Then by applying Claims~\ref{CLAIM:two-edges-w},~\ref{CLAIM:two-edges-x},~\ref{CLAIM:two-edges-y},~\ref{CLAIM:two-edges-B}, and~\ref{CLAIM:two-edges-D}, one could verify (albeit somewhat tediously) that there would be at least $12 > 36 - e(H_i, H_j)$ missing edges, which is a contradiction.   

        Suppose to the contrary that $u_i v_j \in G$. Then by applying Claims~\ref{CLAIM:two-edges-w},~\ref{CLAIM:two-edges-z},~\ref{CLAIM:two-edges-A}, and Claim~\ref{CLAIM:two-edges-E}~\ref{CLAIM:two-edges-E-1}, one could verify  that there would be at least $12$ missing edges, which is a contradiction. 

        Suppose to the contrary that $u_i a_j \in G$. Then by applying Claims~\ref{CLAIM:two-edges-w},~\ref{CLAIM:two-edges-B},~\ref{CLAIM:two-edges-C}, and Claim~\ref{CLAIM:two-edges-E}~\ref{CLAIM:two-edges-E-1}, one could verify  that there would be at least $12$ missing edges, which is a contradiction.

        Suppose to the contrary that $v_i u_j \in G$. Then by applying Claims~\ref{CLAIM:two-edges-w}~\ref{CLAIM:two-edges-D}, and Claim~\ref{CLAIM:two-edges-E}~\ref{CLAIM:two-edges-E-1} \ref{CLAIM:two-edges-E-2} \ref{CLAIM:two-edges-E-3}, one could verify  that there would be at least $12$ missing edges, which is a contradiction. 
    \end{proof} 

\begin{figure}[H]
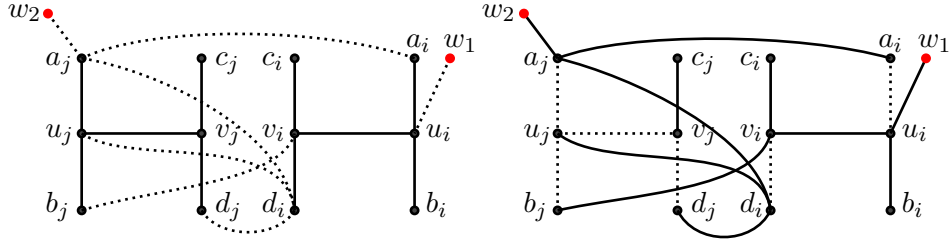

\centering
\tikzset{every picture/.style={line width=1pt}} 

\caption{Decomposition of $W$ into two disjoint copies of $H$ and one edge.} 
\label{Fig:two-edges-rotation-40}
\end{figure}

    Note that Claims~\ref{CLAIM:two-edges-w} and~\ref{CLAIM:two-edges-E} contribute $11 \ge 36 - e(H_i, H_j)$ missing edges. Thus all other pairs in $V(H_i) \times V(H_j)$ are edges in $G$. In particular, $\{v_i b_j, d_i u_j, d_i d_j, d_i a_j, a_i a_j\} \subseteq G$. Then $G[W]$ contains two disjoint copies of $H$ and one disjoint edge, which together cover $14$ vertices (see Figure~\ref{Fig:two-edges-rotation-40}). By~\eqref{equ:w1-w2-choices}, this means that $\{H_i, H_j\}$ is $2$-extendable, which is a contradiction.
    This completes the proof of Lemma~\ref{LEMMA:H1-H1-upper-bound}. 
\end{proof}

\subsection{Upper bound for $e(\mathcal{H}_2, \mathcal{H}_i)$ for $i \in \{0,1,2\}$}\label{SUBSEC:H2-Hi-upper-bound}
\begin{lemma}\label{LEMMA:H2-Hi-upper-bound}
    Suppose that $\{H_i, H_j\} \subseteq \mathcal{H}'$ is a pair such that $|V(H_i) \cap L| \ge 2$. 
    Then $e(H_i, H_j) \le 24$. 
    In particular, $e(\mathcal{H}_0, \mathcal{H}_2) \le 24 y_0 y_2 n^2$ and 
    \begin{align*}
        e(\mathcal{H}_2) 
        \le 24 \binom{y_2 n}{2} + \binom{6}{2}y_2 n
        \le 12 y_2^2 n^2 + 3y_2 n. 
    \end{align*}
\end{lemma}
\begin{proof}[Proof of Lemma~\ref{LEMMA:H2-Hi-upper-bound}]
    Fix a pair $\{H_i, H_j\} \subseteq \mathcal{H}'$ such that $|V(H_i) \cap L| \ge 2$. 
    Suppose to the contrary that $e(H_i, H_j) \ge 25$. 
    Fix two vertices $\{x,y\} \subseteq V(H_i) \cap L$. 
    Fix an arbitrary pair of distinct vertices $(w_1, w_2) \in N_{G}(x, \overline{U}) \times  N_{G}(y, \overline{U})$. 
    Let $W\coloneqq V(H_i \cup H_j) \cup \{w_1, w_2\}$. 
    Note from the definition of $L$ that the number of choices for such (unordered) pairs $\{w_1, w_2\}$ is at least 
    \begin{align}\label{equ:w-pair-choices-a}
        \frac{1}{2} \cdot 19 \delta^{1/6} n \cdot (19 \delta^{1/6} n -1)
        \ge 19 \delta n^2. 
    \end{align}

   By Lemma~\ref{LEMMA:H-R-at-most-3-edges} and Assumption~\eqref{equ:assum-ci-di-not-in-L}, the pair $\{x,y\}$ must be one of the following: $\{u_i, v_i\}$, $\{a_i, b_i\}$, $\{v_i, a_i\}$, or $\{v_i, b_i\}$. 

    \medskip 

    \textbf{Case 1}: $\{x, y\} = \{u_i, v_i\}$.

\begin{figure}[H]
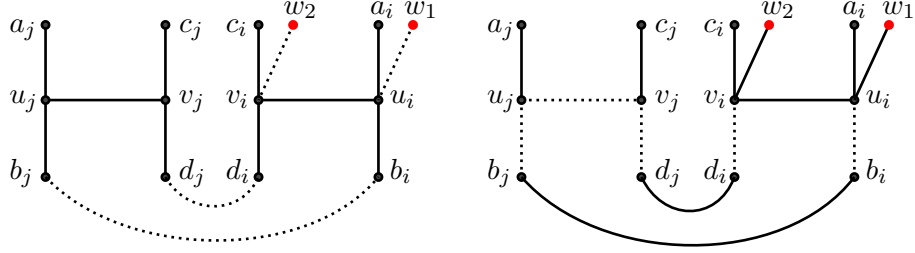

\centering
\tikzset{every picture/.style={line width=1pt}} 

\caption{Auxiliary figure for the proof of Claim~\ref{CLAIM:two-edges-a}.} 
\label{Fig:two-edges-rotation-1}
\end{figure}
    
    \begin{claim}\label{CLAIM:two-edges-a}
        The following statements hold. 
        \begin{enumerate}[label=(\roman*)]
            \item\label{CLAIM:two-edges-a-1} Either no pair in $\{a_i, b_i\} \times \{a_j, b_j\}$ is an edge in $G$, or no pair in $\{c_i, d_i\} \times \{c_j, d_j\}$ is an edge in $G$. 
            \item\label{CLAIM:two-edges-a-2} Either no pair in $\{a_i, b_i\} \times \{c_j, d_j\}$ is an edge in $G$, or no pair in $\{c_i, d_i\} \times \{a_j, b_j\}$ is an edge in $G$. 
        \end{enumerate}  
    \end{claim}
    \begin{proof}[Proof of Claim~\ref{CLAIM:two-edges-a}]
         By symmetry, it suffices to show~\ref{CLAIM:two-edges-a-1}. Suppose to the contrary that~\ref{CLAIM:two-edges-a-1} fails. By symmetry,  we may assume that $\{b_i b_j, d_i d_j\} \subseteq G$. Then $G[W]$ contains a copy of $H$ and four pairwise disjoint edges, which together cover $14$ vertices (see Figure~\ref{Fig:two-edges-rotation-1}).
         By~\eqref{equ:w-pair-choices-a}, this means that $\{H_i, H_j\}$ is $2$-extendable, which is a contradiction. 
    \end{proof}

\begin{figure}[H]
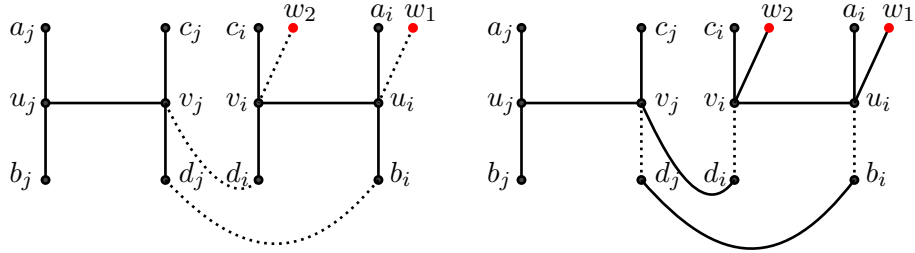

\centering
\tikzset{every picture/.style={line width=1pt}} 

\caption{Auxiliary figure for the proof of Claim~\ref{CLAIM:two-edges-b}.} 
\label{Fig:two-edges-rotation-2}
\end{figure}
    
    \begin{claim}\label{CLAIM:two-edges-b}
        The following statements hold. 
        \begin{enumerate}[label=(\roman*)]
            \item\label{CLAIM:two-edges-b-1} Either no pair in $\{a_i, b_i\} \times \{c_j, d_j\}$ is an edge in $G$, or $\{c_i v_j, d_i v_j\} \cap G = \emptyset$. 
            \item\label{CLAIM:two-edges-b-2} Either no pair in $\{a_i, b_i\} \times \{a_j, b_j\}$ is an edge in $G$, or $\{c_i u_j, d_i u_j\} \cap G = \emptyset$. 
            \item\label{CLAIM:two-edges-b-3} Either $\{a_i u_j, b_i u_j\} \cap G = \emptyset$, or no pair in $\{c_i, d_i\} \times \{c_j, d_j\}$ is an edge in $G$. 
            \item\label{CLAIM:two-edges-b-4} Either $\{a_i v_j, b_i v_j\} \cap G = \emptyset$, or no pair in $\{c_i, d_i\} \times \{a_j, b_j\}$ is an edge in $G$. 
        \end{enumerate}  
    \end{claim}
    \begin{proof}[Proof of Claim~\ref{CLAIM:two-edges-b}]
         By symmetry, it suffices to show~\ref{CLAIM:two-edges-b-1}. Suppose to the contrary that~\ref{CLAIM:two-edges-b-1} fails. By symmetry,  we may assume that $\{b_i d_j, d_i v_j\} \subseteq G$. Then $G[W]$ contains two disjoint copies of $H$ and one disjoint edge, which together cover $14$ vertices (see Figure~\ref{Fig:two-edges-rotation-2}). By~\eqref{equ:w-pair-choices-a}, this means that $\{H_i, H_j\}$ is $2$-extendable, which is a contradiction.  
    \end{proof}

    It follows from Claim~\ref{CLAIM:two-edges-a} and~\ref{CLAIM:two-edges-b} that 
    \begin{align*}
        e(H_i, H_j)
        \le 36 - 4 - 4 - 2 - 2
        = 24 < 25, 
    \end{align*} 
    a contradiction.

    \medskip 

    \textbf{Case 2}: $\{x, y\} = \{a_i, b_i\}$. 
    
    \begin{claim}\label{CLAIM:two-edges-c}
        The following statements hold. 
        \begin{enumerate}[label=(\roman*)]
            \item\label{CLAIM:two-edges-c-1} There is no vertex in $\{u_i, c_i, d_i\}$ that has neighbors in both $\{a_j, b_j\}$ and $\{c_j, d_j\}$. 
            \item\label{CLAIM:two-edges-c-2} There is no matching of size two with one edge belonging to $\{u_i, c_i, d_i\} \times \{c_j, d_j\}$ and another edge belonging to $\{u_i, c_i, d_i\} \times \{a_j, b_j\}$. 
        \end{enumerate}  
        In particular, either no pair in $\{u_i, c_i, d_i\} \times \{a_j, b_j\}$ is an edge in $G$, or no pair in $\{u_i, c_i, d_i\} \times \{c_j, d_j\}$ is an edge in $G$. 
    \end{claim}
    \begin{proof}[Proof of Claim~\ref{CLAIM:two-edges-c}]
\begin{figure}[H]
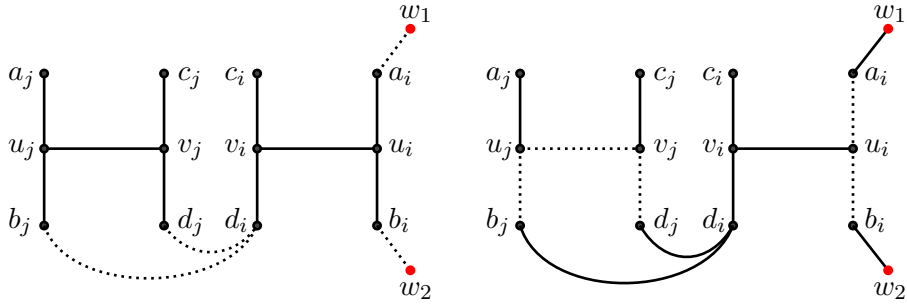

\centering
\tikzset{every picture/.style={line width=1pt}} 

\caption{Auxiliary figure for the proof of Claim~\ref{CLAIM:two-edges-c}~\ref{CLAIM:two-edges-c-1}.} 
\label{Fig:two-edges-rotation-3}
\end{figure}
        
        Let us first prove~\ref{CLAIM:two-edges-c-1}. Suppose to the contrary that~\ref{CLAIM:two-edges-c-1} fails. By symmetry, we may assume that $\{d_i b_j, d_i d_j\} \subseteq G$. Then $G[W]$ contains a copy of $H$ and four pairwise disjoint edges, which together cover $14$ vertices (see Figure~\ref{Fig:two-edges-rotation-3}). By~\eqref{equ:w-pair-choices-a}, this means that $\{H_i, H_j\}$ is $2$-extendable, which is a contradiction. 

\begin{figure}[H]
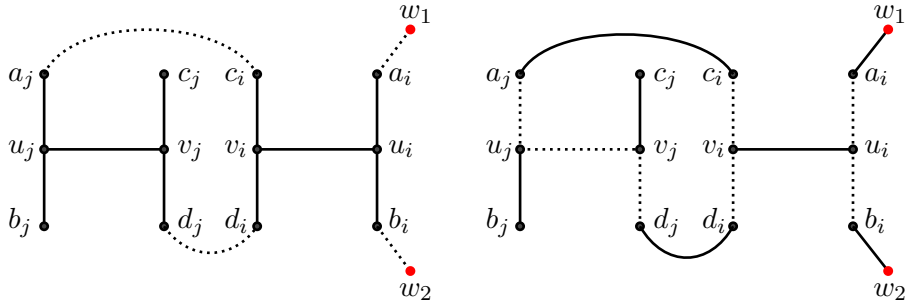

\centering
\tikzset{every picture/.style={line width=1pt}} 

\caption{Auxiliary figure for the proof of Claim~\ref{CLAIM:two-edges-c}~\ref{CLAIM:two-edges-c-2}.} 
\label{Fig:two-edges-rotation-4}
\end{figure}

        Next, we prove~\ref{CLAIM:two-edges-c-2}. Suppose to the contrary that~\ref{CLAIM:two-edges-c-2} fails. By symmetry, we may assume that $\{c_i a_j, d_i d_j\} \subseteq G$. Then $G[W]$ contains seven pairwise disjoint edges, which together cover $14$ vertices (see Figure~\ref{Fig:two-edges-rotation-4}). By~\eqref{equ:w-pair-choices-a}, this means that $\{H_i, H_j\}$ is $2$-extendable, which is a contradiction. 
    \end{proof} 

    By symmetry, we may assume that 
    \begin{align*}
        \{u_i a_j, u_i b_j, c_i a_j, c_i b_j, d_i a_j, d_i b_j\} \cap G = \emptyset. 
    \end{align*}

    \begin{claim}\label{CLAIM:two-edges-d}
        The following statements hold. 
        \begin{enumerate}[label=(\roman*)]
            \item\label{CLAIM:two-edges-d-1} Either $\{a_i u_j, b_i u_j\} \cap G = \emptyset$ or no pair in $\{c_i, d_i\} \times \{c_j, d_j\}$ is an edge in $G$. 
            \item\label{CLAIM:two-edges-d-2} Either $\{a_i u_j, b_i u_j\} \cap G = \emptyset$ or $v_i v_j \not\in G$.
            \item\label{CLAIM:two-edges-d-3} Either $\{a_i v_j, b_i v_j\} \cap G = \emptyset$ or $v_i u_j \not\in G$.
            \item\label{CLAIM:two-edges-d-4} Either $v_i u_j \not\in G$ or no pair in $\{u_i, c_i, d_i\} \times \{c_j, d_j\}$ is an edge in $G$. 
            \item\label{CLAIM:two-edges-d-5} There is no matching of size two with one edge belonging to $\{u_i, c_i, d_i\} \times \{v_j\}$ and the other edge belonging to $\{u_i, c_i, d_i\} \times \{c_j, d_j\}$. 
        \end{enumerate}  
    \end{claim}
    \begin{proof}[Proof of Claim~\ref{CLAIM:two-edges-d}]
\begin{figure}[H]
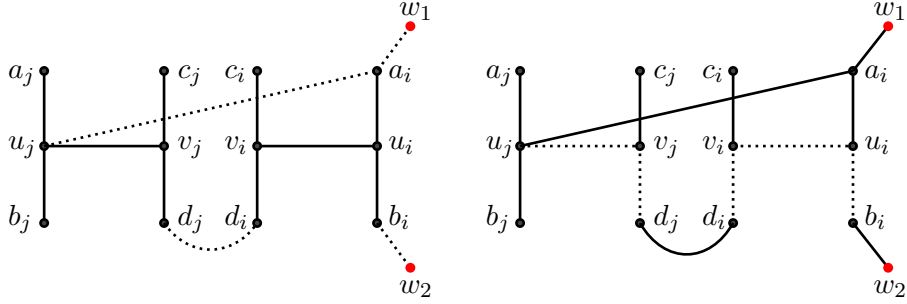

\centering
\tikzset{every picture/.style={line width=1pt}} 

\caption{Auxiliary figure for the proof of Claim~\ref{CLAIM:two-edges-d}~\ref{CLAIM:two-edges-d-1}.} 
\label{Fig:two-edges-rotation-5}
\end{figure}

        Let us first prove~\ref{CLAIM:two-edges-d-1}. Suppose to the contrary that~\ref{CLAIM:two-edges-d-1} fails. By symmetry, we may assume that $\{a_i u_j, d_i d_j\} \subseteq G$. Then $G[W]$ contains a copy of $H$ and four pairwise disjoint edges, which together cover $14$ vertices (see Figure~\ref{Fig:two-edges-rotation-5}). By~\eqref{equ:w-pair-choices-a}, this means that $\{H_i, H_j\}$ is $2$-extendable, which is a contradiction. 

\begin{figure}[H]
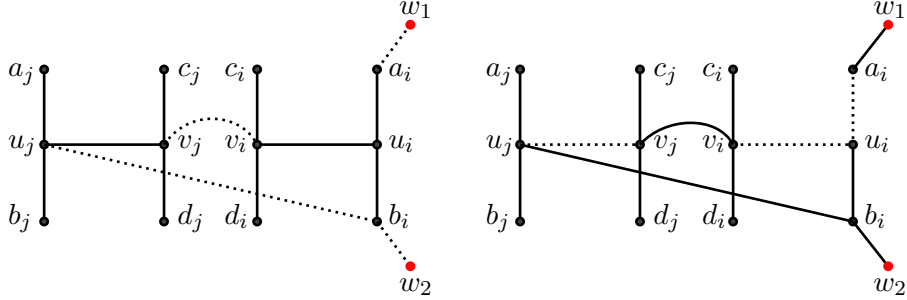

\centering
\tikzset{every picture/.style={line width=1pt}} 

\caption{Auxiliary figure for the proof of Claim~\ref{CLAIM:two-edges-d}~\ref{CLAIM:two-edges-d-2} and~\ref{CLAIM:two-edges-d-3}.} 
\label{Fig:two-edges-rotation-6}
\end{figure}

        Next, we prove~\ref{CLAIM:two-edges-d-2}. Suppose to the contrary that~\ref{CLAIM:two-edges-d-2} fails. By symmetry, we may assume that $\{v_i v_j, b_i u_j\} \subseteq G$. Then $G[W]$ contains two disjoint copies of $H$ and one disjoint edge, which together cover $14$ vertices (see Figure~\ref{Fig:two-edges-rotation-6}). By~\eqref{equ:w-pair-choices-a}, this means that $\{H_i, H_j\}$ is $2$-extendable, which is a contradiction. 
        
        By symmetry,~\ref{CLAIM:two-edges-d-3} holds as well.

\begin{figure}[H]
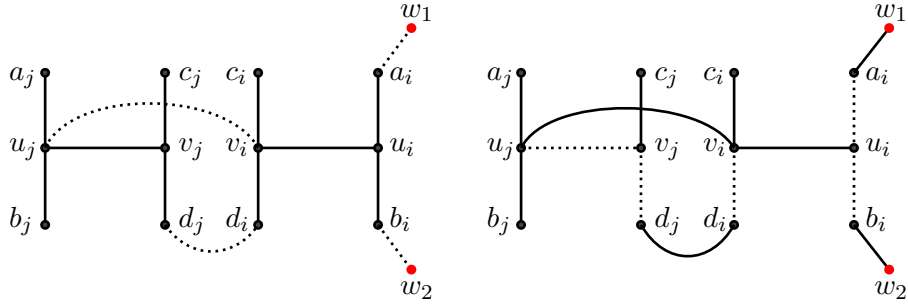

\centering
\tikzset{every picture/.style={line width=1pt}} 

\caption{Auxiliary figure for the proof of Claim~\ref{CLAIM:two-edges-d}~\ref{CLAIM:two-edges-d-4}.} 
\label{Fig:two-edges-rotation-7}
\end{figure}

        Next, we prove~\ref{CLAIM:two-edges-d-4}. Suppose to the contrary that~\ref{CLAIM:two-edges-d-4} fails. By symmetry, we may assume that $\{v_i u_j, d_i d_j\} \subseteq G$. Then $G[W]$ contains a copy of $H$ and four pairwise disjoint edges, which together cover $14$ vertices (see Figure~\ref{Fig:two-edges-rotation-7}). By~\eqref{equ:w-pair-choices-a}, this means that $\{H_i, H_j\}$ is $2$-extendable, which is a contradiction. 

\begin{figure}[H]
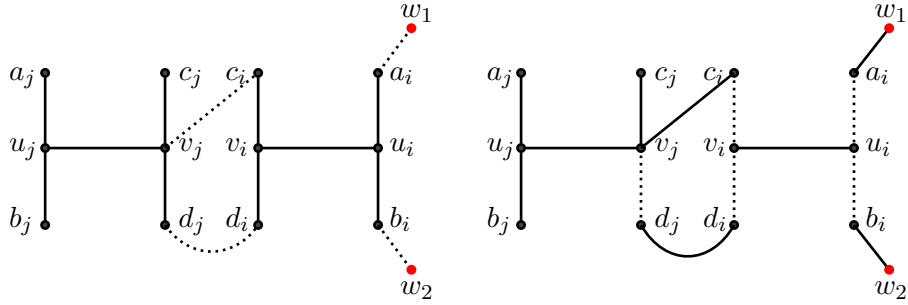

\centering
\tikzset{every picture/.style={line width=1pt}} 

\caption{Auxiliary figure for the proof of Claim~\ref{CLAIM:two-edges-d}~\ref{CLAIM:two-edges-d-5}.} 
\label{Fig:two-edges-rotation-8}
\end{figure}

        Finally, we prove~\ref{CLAIM:two-edges-d-5}. Suppose to the contrary that~\ref{CLAIM:two-edges-d-5} fails. By symmetry, we may assume that $\{c_i v_j, d_i d_j\} \in G$. Then $G[W]$ contains a copy of $H$ and four pairwise disjoint edges, which together cover $14$ vertices (see Figure~\ref{Fig:two-edges-rotation-7}). By~\eqref{equ:w-pair-choices-a}, this means that $\{H_i, H_j\}$ is $2$-extendable, which is a contradiction. 
    \end{proof}

    Note from Claim~\ref{CLAIM:two-edges-d} that there are at least six pairs in $V(H_i) \times V(H_j)$ that do not belong to $G[H_i, H_j]$, aside from $\{u_i a_j, u_i b_j, c_i a_j, c_i b_j, d_i a_j, d_i b_j\}$. Thus we have 
    \begin{align*}
        e(H_i, H_j)
        \le 36 - 6 - 6 
        = 24 < 25, 
    \end{align*} 
    which is a contradiction.

    \medskip 

    \textbf{Case 3}: $\{x, y\} = \{v_i, a_i\}$ or $\{x, y\} = \{v_i, b_i\}$.  

    By symmetry, we may assume that  $\{x, y\} = \{v_i, a_i\}$. 
    
    \begin{claim}\label{CLAIM:two-edges-e}
        The following statements hold. 
        \begin{enumerate}[label=(\roman*)]
            \item\label{CLAIM:two-edges-e-1} There is no vertex in $\{c_i, d_i\}$ that has neighbors in both $\{a_j, b_j\}$ and $\{c_j, d_j\}$. 
            \item\label{CLAIM:two-edges-e-2} There is no matching of size two with one edge belonging to $\{c_i, d_i\} \times \{c_j, d_j\}$ and another edge belonging to $\{c_i, d_i\} \times \{a_j, b_j\}$. 
        \end{enumerate}  
        In particular, either no pair in $\{c_i, d_i\} \times \{a_j, b_j\}$ is an edge in $G$, or no pair in $\{c_i, d_i\} \times \{c_j, d_j\}$ is an edge in $G$. 
    \end{claim}
    \begin{proof}[Proof of Claim~\ref{CLAIM:two-edges-e}]
\begin{figure}[H]
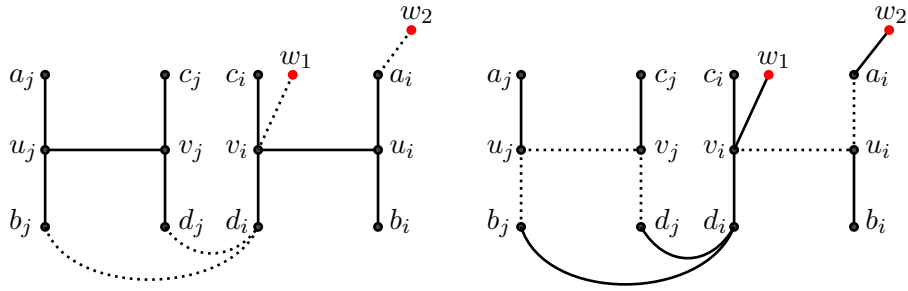

\centering
\tikzset{every picture/.style={line width=1pt}} 

\caption{Auxiliary figure for the proof of Claim~\ref{CLAIM:two-edges-e}~\ref{CLAIM:two-edges-e-1}.} 
\label{Fig:two-edges-rotation-9}
\end{figure}
        
        Let us first prove~\ref{CLAIM:two-edges-e-1}. Suppose to the contrary that~\ref{CLAIM:two-edges-e-1} fails. By symmetry, we may assume that $\{d_i b_j, d_i d_j\} \subseteq G$. Then $G[W]$ contains a copy of $H$ and four pairwise disjoint edges, which together cover $14$ vertices (see Figure~\ref{Fig:two-edges-rotation-9}). By~\eqref{equ:w-pair-choices-a}, this means that $\{H_i, H_j\}$ is $2$-extendable, which is a contradiction.

\begin{figure}[H]
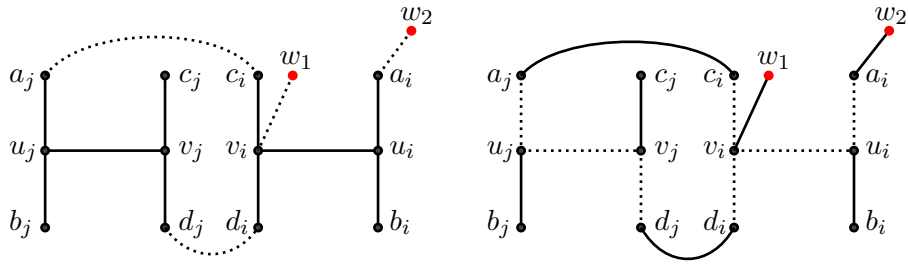

\centering
\tikzset{every picture/.style={line width=1pt}} 

\caption{Auxiliary figure for the proof of Claim~\ref{CLAIM:two-edges-e}~\ref{CLAIM:two-edges-e-2}.} 
\label{Fig:two-edges-rotation-10}
\end{figure}

        Next, we prove~\ref{CLAIM:two-edges-e-2}. Suppose to the contrary that~\ref{CLAIM:two-edges-e-2} fails. By symmetry, we may assume that $\{c_i a_j, d_i d_j\} \subseteq G$. Then $G[W]$ contains seven pairwise disjoint edges, which together cover $14$ vertices (see Figure~\ref{Fig:two-edges-rotation-10}). By~\eqref{equ:w-pair-choices-a}, this means that $\{H_i, H_j\}$ is $2$-extendable, which is a contradiction.
    \end{proof}

    By symmetry, we may assume that 
    \begin{align}\label{equ:assum-CLAIM-two-edges-e}
        \{c_i a_j, c_i b_j, d_i a_j, d_i b_j\} \cap G = \emptyset. 
    \end{align}

\begin{figure}[H]
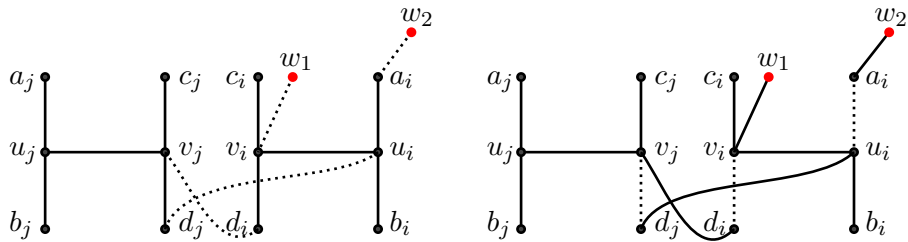

\centering
\tikzset{every picture/.style={line width=1pt}} 

\caption{Auxiliary figure for the proof of Claim~\ref{CLAIM:two-edges-f}.} 
\label{Fig:two-edges-rotation-11}
\end{figure}

    \begin{claim}\label{CLAIM:two-edges-f}
        The following statements hold. 
        \begin{enumerate}[label=(\roman*)]
            \item\label{CLAIM:two-edges-f-1} Either $\{c_i v_j, d_i v_j\} \cap G = \emptyset$ or $\{u_i c_j, u_i d_j\} \cap G = \emptyset$. 
            \item\label{CLAIM:two-edges-f-2} Either $\{c_i u_j, d_i u_j\} \cap G = \emptyset$ or $\{u_i a_j, u_i b_j\} \cap G = \emptyset$.
        \end{enumerate}  
    \end{claim}
    \begin{proof}[Proof of Claim~\ref{CLAIM:two-edges-f}]
        By symmetry, it suffices to show~\ref{CLAIM:two-edges-f-1}. Suppose to the contrary that~\ref{CLAIM:two-edges-f-1} fails. By symmetry, we may assume that $\{d_i v_j, u_i d_j\} \subseteq G$. Then $G[W]$ contains two disjoint copies of $H$ and one disjoint edge, which together cover $14$ vertices (see Figure~\ref{Fig:two-edges-rotation-11}). By~\eqref{equ:w-pair-choices-a}, this means that $\{H_i, H_j\}$ is $2$-extendable, which is a contradiction. 
    \end{proof}

\begin{figure}[H]
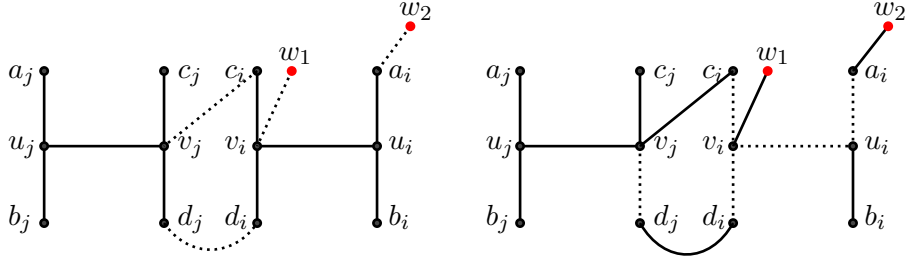

\centering
\tikzset{every picture/.style={line width=1pt}} 

\caption{Auxiliary figure for the proof of Claim~\ref{CLAIM:two-edges-g}.} 
\label{Fig:two-edges-rotation-12}
\end{figure}

    \begin{claim}\label{CLAIM:two-edges-g}
        The following statements hold. 
        \begin{enumerate}[label=(\roman*)]
            \item\label{CLAIM:two-edges-g-1} Either $c_i v_j \not\in G$ or $\{d_i c_j, d_i d_j\} \cap G = \emptyset$. 
            \item\label{CLAIM:two-edges-g-2} Either $d_i v_j \not\in G$ or $\{c_i c_j, c_i d_j\} \cap G = \emptyset$.
        \end{enumerate}  
    \end{claim}
    \begin{proof}[Proof of Claim~\ref{CLAIM:two-edges-g}]
        By symmetry, it suffices to show~\ref{CLAIM:two-edges-g-1}. Suppose to the contrary that~\ref{CLAIM:two-edges-g-1} fails. By symmetry, we may assume that $\{d_i d_j, c_i v_j\} \subseteq G$. Then $G[W]$ contains a copy of $H$ and four pairwise disjoint edges, which together cover $14$ vertices (see Figure~\ref{Fig:two-edges-rotation-12}). By~\eqref{equ:w-pair-choices-a}, this means that $\{H_i, H_j\}$ is $2$-extendable, which is a contradiction.
    \end{proof}

    \begin{claim}\label{CLAIM:two-edges-h}
        The following statements hold. 
        \begin{enumerate}[label=(\roman*)]
            \item\label{CLAIM:two-edges-h-1} $v_i u_j \not\in G$. 
            \item\label{CLAIM:two-edges-h-2} $\{c_i v_j, d_i v_j\} \cap G = \emptyset$.
        \end{enumerate}  
    \end{claim}
    \begin{proof}[Proof of Claim~\ref{CLAIM:two-edges-h}]
\begin{figure}[H]
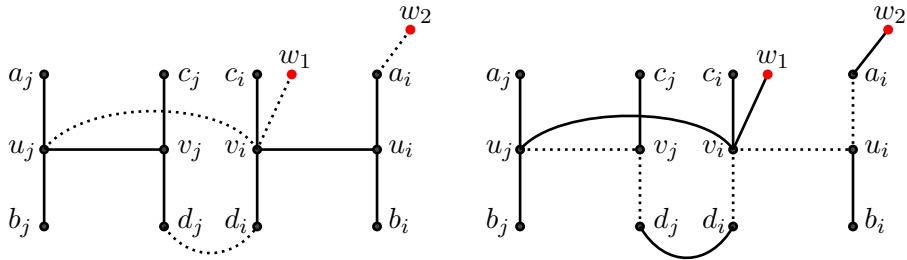

\centering
\tikzset{every picture/.style={line width=1pt}} 

\caption{Auxiliary figure for the proof of Claim~\ref{CLAIM:two-edges-h}~\ref{CLAIM:two-edges-h-1}.} 
\label{Fig:two-edges-rotation-13}
\end{figure}
        
        Let us first prove~\ref{CLAIM:two-edges-h-1}. Note from Claim~\ref{CLAIM:two-edges-f} that there are at least four pairs in $V(H_i) \times V(H_j)$ that do not belong to $G[H_i, H_j]$, aside from $\{c_i a_j, c_i b_j, d_i a_j, d_i b_j\}$. Since $e(H_i, H_j) \ge 25 = 36 - 8 - 3$, there exists an edge in $\{c_i, d_i\} \times \{c_j, d_j\}$. By symmetry, we may assume that $d_i d_j \in G$.
        Suppose to the contrary that~\ref{CLAIM:two-edges-h-1} fails. That is, $v_i u_j \in G$. Then $G[W]$ contains a copy of $H$ and four pairwise disjoint edges, which together cover $14$ vertices (see Figure~\ref{Fig:two-edges-rotation-13}). By~\eqref{equ:w-pair-choices-a}, this means that $\{H_i, H_j\}$ is $2$-extendable, which is a contradiction.

        Suppose to the contrary that~\ref{CLAIM:two-edges-h-2} fails. By symmetry, we may assume that $c_i v_j \in G$. Then, by Claim~\ref{CLAIM:two-edges-f},~\ref{CLAIM:two-edges-g},~\ref{CLAIM:two-edges-h}~\ref{CLAIM:two-edges-h-1}, and~\eqref{equ:assum-CLAIM-two-edges-e}, there are at least $12 > 11 = 36 - 25$ pairs in $V(H_i) \times V(H_j)$ that do not belong to $G[H_i, H_j]$, which is a contradiction.
    \end{proof} 

\begin{figure}[H]
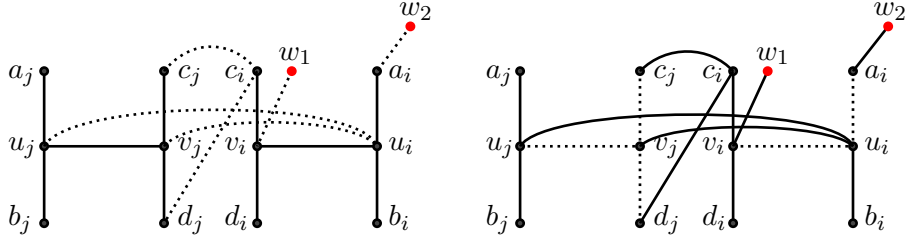

\centering
\tikzset{every picture/.style={line width=1pt}} 

\caption{Auxiliary figure for the proof of Claim~\ref{CLAIM:two-edges-i}.} 
\label{Fig:two-edges-rotation-14}
\end{figure}

    \begin{claim}\label{CLAIM:two-edges-i}
        The following statements hold. 
        \begin{enumerate}[label=(\roman*)]
            \item\label{CLAIM:two-edges-i-1} Either $c_i c_j \not\in G$ or $c_i d_j \not\in G$. 
            \item\label{CLAIM:two-edges-i-2} Either $d_i c_j \not\in G$ or $d_i d_j \not\in G$.
        \end{enumerate}  
    \end{claim}
    \begin{proof}[Proof of Claim~\ref{CLAIM:two-edges-i}]
        
        By symmetry, it suffices to show~\ref{CLAIM:two-edges-i-1}. 
        Note from Claim~\ref{CLAIM:two-edges-f}~\ref{CLAIM:two-edges-f-2},~\ref{CLAIM:two-edges-h}, and~\eqref{equ:assum-CLAIM-two-edges-e}, there are at least $9$ pairs in $V(H_i) \times V(H_j)$ that do not belong to $G[H_i, H_j]$. 
        Therefore, there exists a vertex $z$ in $\{u_i, a_i, b_i\}$ such that $z u_j, z v_j \in G$ since $e(H_i, H_j) \ge 25 = 36 - 9 - 2$. By symmetry, we may assume that $\{u_i u_j, u_i v_j\} \subseteq G$.
        Suppose to the contrary that~\ref{CLAIM:two-edges-i-1} fails. That is, $\{c_i c_j, c_i d_j\} \subseteq G$. 
        Then $G[W]$ contains two disjoint copies of $H$ and one disjoint edge, which together cover $14$ vertices (see Figure~\ref{Fig:two-edges-rotation-14}). By~\eqref{equ:w-pair-choices-a}, this means that $\{H_i, H_j\}$ is $2$-extendable, which is a contradiction.
    \end{proof} 

\begin{figure}[H]
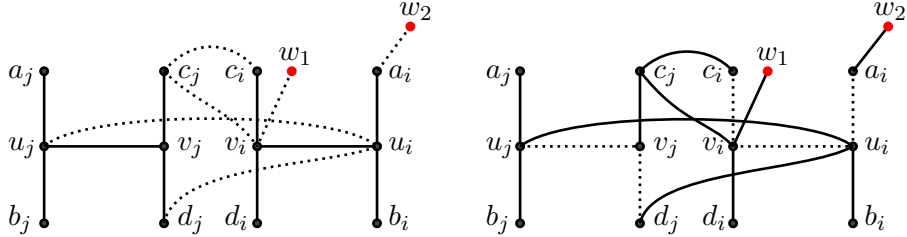

\centering
\tikzset{every picture/.style={line width=1pt}} 

\caption{Decomposition of $W$ into two disjoint copies of $H$ and one edge.} 
\label{Fig:two-edges-rotation-15}
\end{figure}

    By symmetry, (and by Claim~\ref{CLAIM:two-edges-i}~\ref{CLAIM:two-edges-i-1}), we may assume that $c_i d_j \not\in G$. 
    Note from Claim~\ref{CLAIM:two-edges-f}~\ref{CLAIM:two-edges-f-2},~\ref{CLAIM:two-edges-h},~\ref{CLAIM:two-edges-i}~\ref{CLAIM:two-edges-i-2}, and~\eqref{equ:assum-CLAIM-two-edges-e}, there are at least $11 \ge 36 - e(H_i, H_j)$ pairs in $V(H_i) \times V(H_j)$ that do not belong to $G[H_i, H_j]$. 
    Therefore, all other pairs in $V(H_i) \times V(H_j)$ are edges in $G$, In particular, $\{c_i c_j, v_i c_j, u_i u_j, u_i d_j\} \subseteq G$. Then $G[W]$ contains two disjoint copies of $H$ and one disjoint edge, which together cover $14$ vertices (see Figure~\ref{Fig:two-edges-rotation-15}). By~\eqref{equ:w-pair-choices-a}, this means that $\{H_i, H_j\}$ is $2$-extendable, which is a contradiction.
    This completes the proof of Lemma~\ref{LEMMA:H2-Hi-upper-bound}. 
\end{proof}

\subsection{Upper bound for $e(\mathcal{H}_2, \mathcal{H}_3)$}\label{SUBSEC:H2-H3-upper-bound}
\begin{lemma}\label{LEMMA:H2-H3-upper-bound}
    Suppose that $\{H_i, H_j\} \subseteq \mathcal{H}'$ is a pair such that $|V(H_i) \cap L| = 3$ and $|V(H_j) \cap L| \ge 2$. 
    Then $e(H_i, H_j) \le 21$. 
    In particular, $e(\mathcal{H}_2, \mathcal{H}_3) \le 21 y_2 y_3 n^2$. 
\end{lemma}
\begin{proof}[Proof of Lemma~\ref{LEMMA:H2-H3-upper-bound}]
    Fix a pair $\{H_i, H_j\} \subseteq \mathcal{H}'$  such that $|V(H_i) \cap L| = 3$ and $|V(H_j) \cap L| \ge 2$.
    Suppose to the contrary that $e(H_i, H_j) \ge 22$. 
    By Lemma~\ref{LEMMA:H-R-at-most-3-edges} and Assumption~\eqref{equ:assum-ci-di-not-in-L}, we have $V(H_i) \cap L = \{v_i, a_i, b_i\}$.
    Fix two vertices $\{x, y\} \subseteq V(H_j) \cap L$. 
    Fix an arbitrary $5$-tuple 
    \begin{align*}  
        (w_1, \ldots, w_5) \in 
        N_{G}(v_i, \overline{U}) \times N_{G}(a_i, \overline{U}) \times N_{G}(b_i, \overline{U}) \times N_{G}(x, \overline{U}) \times N_{G}(y, \overline{U}). 
    \end{align*}
    Let $W\coloneqq V(H_i \cup H_j) \cup \{w_1, \ldots, w_5\}$. 
    Note from the definition of $L$ that the number of choices for such (unordered) $5$-sets $\{w_1, \ldots, w_5\}$ is at least 
    \begin{align}\label{equ:w-5-tuple-choices-a}
        \frac{1}{5!} \cdot 19 \delta^{1/6} n \cdot (19 \delta^{1/6} n) \cdots (19 \delta^{1/6} n - 4)
        \ge 19 \delta n^5. 
    \end{align}

    By Lemma~\ref{LEMMA:H-R-at-most-3-edges} and Assumption~\eqref{equ:assum-ci-di-not-in-L}, the pair $\{x,y\}$ must be one of the following: $\{u_j, v_j\}$, $\{a_j, b_j\}$, $\{v_j, a_j\}$, or $\{v_j, b_j\}$.  
    
    \medskip 

    \textbf{Case 1}: $\{x, y\} = \{u_j, v_j\}$. 

\begin{figure}[H]
\centering
\tikzset{every picture/.style={line width=1pt}} 

\caption{Auxiliary figure for the proof of Claim~\ref{CLAIM:five-edges-a}.} 
\label{Fig:five-edges-rotation-1}
\end{figure}
    
    \begin{claim}\label{CLAIM:five-edges-a}
        No pair in $\{u_i, c_i, d_i\} \times \{a_j, b_j, c_j, d_j\}$ is an edge in $G$. 
    \end{claim}
    \begin{proof}[Proof of Claim~\ref{CLAIM:five-edges-a}]
        Suppose to the contrary that this claim fails. By symmetry, we may assume that $d_id_j \in G$. Then $G[W]$ contains a copy of $H$ and four pairwise disjoint edges, which together cover $14$ vertices (see Figure~\ref{Fig:five-edges-rotation-1}). By~\eqref{equ:w-5-tuple-choices-a}, this means that $\{H_i, H_j\}$ is $5$-extendable, which is a contradiction. 
    \end{proof} 

\begin{figure}[H]
\centering
\tikzset{every picture/.style={line width=1pt}} 

\caption{Auxiliary figure for the proof of Claim~\ref{CLAIM:five-edges-a}.} 
\label{Fig:five-edges-rotation-2}
\end{figure}
    
    \begin{claim}\label{CLAIM:five-edges-b}
        The following statements hold. 
        \begin{enumerate}[label=(\roman*)]
            \item\label{CLAIM:five-edges-b-1} Either $a_i u_j \not\in G$ or $v_i v_j \not\in G$.
            \item\label{CLAIM:five-edges-b-2} Either $b_i u_j \not\in G$ or $v_i v_j \not\in G$.
            \item\label{CLAIM:five-edges-b-3} Either $a_i v_j \not\in G$ or $v_i u_j \not\in G$.
            \item\label{CLAIM:five-edges-b-4} Either $b_i v_j \not\in G$ or $v_i u_j \not\in G$.
        \end{enumerate} 
    \end{claim}
    \begin{proof}[Proof of Claim~\ref{CLAIM:five-edges-b}]
        By symmetry, it suffices to show~\ref{CLAIM:five-edges-b-1}. Suppose to the contrary that~\ref{CLAIM:five-edges-b-1} fails, that is, $\{a_i u_j, v_i v_j\} \in G$. Then $G[W]$ contains two disjoint copies of $H$ and one disjoint edge, which together cover $14$ vertices (see Figure~\ref{Fig:five-edges-rotation-2}). By~\eqref{equ:w-5-tuple-choices-a}, this means that $\{H_i, H_j\}$ is $5$-extendable, which is a contradiction.
    \end{proof} 

    \begin{claim}\label{CLAIM:five-edges-c}
         We have $\{v_i u_j, v_i v_j\} \cap G = \emptyset$. 
    \end{claim}
    \begin{proof}[Proof of Claim~\ref{CLAIM:five-edges-c}]
        By symmetry, it suffices to show $v_i u_j \not\in G$. Suppose to the contrary that $v_i u_j \in G$. Then by Claim~\ref{CLAIM:five-edges-a} and~\ref{CLAIM:five-edges-b}, there would be at least $15 > 36 - e(H_i, H_j)$ pairs in $V(H_i) \times V(H_j)$ that do not belong to $G[H_i, H_j]$, which is a contradiction. 
    \end{proof} 

\begin{figure}[H]
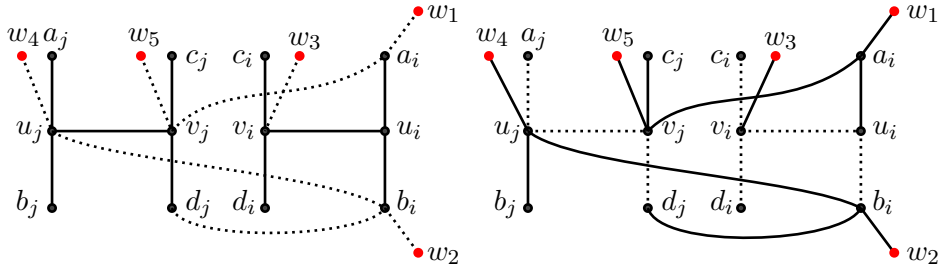

\centering
\tikzset{every picture/.style={line width=1pt}} 

\caption{Decomposition of $W$ into two disjoint copies of $H$ and one edge.} 
\label{Fig:five-edges-rotation-3}
\end{figure}

    Note from Claim~\ref{CLAIM:five-edges-a} and~\ref{CLAIM:five-edges-c}, there are already $14 = 36 - e(H_i, H_j)$ pairs in $V(H_i) \times V(H_j)$ that belong to $G[H_i, H_j]$. Therefore, all other pairs in $V(H_i) \times V(H_j)$ are edges in $G$. In particular, $\{a_i v_j, b_i u_j, b_i d_j\} \subseteq G$. Then $G[W]$ contains two disjoint copies of $H$ and one disjoint edge, which together cover $14$ vertices (see Figure~\ref{Fig:five-edges-rotation-3}). By~\eqref{equ:w-5-tuple-choices-a}, this means that $\{H_i, H_j\}$ is $5$-extendable, which is a contradiction.

    \medskip 

    \textbf{Case 2}: $\{x, y\} = \{a_j, b_j\}$. 

\begin{figure}[H]
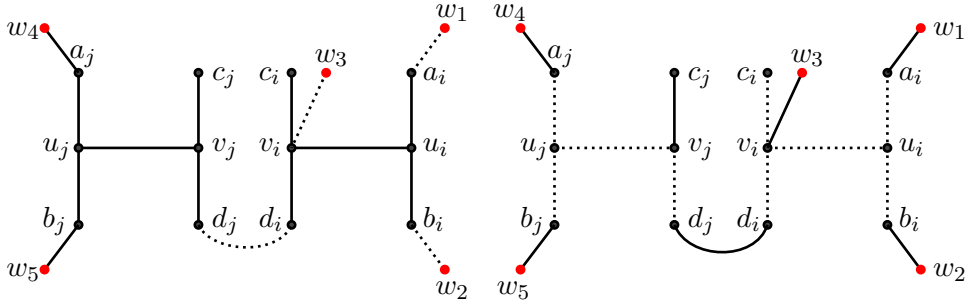

\centering
\tikzset{every picture/.style={line width=1pt}} 

\caption{Auxiliary figure for the proof of Claim~\ref{CLAIM:five-edges-d}.} 
\label{Fig:five-edges-rotation-4}
\end{figure}
    
    \begin{claim}\label{CLAIM:five-edges-d}
        No pair in $\{u_i, c_i, d_i\} \times \{u_j, c_j, d_j\}$ is an edge in $G$. 
    \end{claim}
    \begin{proof}[Proof of Claim~\ref{CLAIM:five-edges-d}]
        Observe that $G[W]$ already contains five pairwise disjoint edges 
        \begin{align}\label{Equ:five-edge-rotation}
            a_iw_1, b_iw_2, v_iw_3, a_jw_4, b_jw_5. 
        \end{align}
        Since $\{u_j, c_j, d_j\}$ is a subset of the neighbor of $v_j$ and every pair in $\{u_i, c_i, d_i\} \times \{u_j, c_j, d_j\}$ is vertex-disjoint from edges in~\eqref{Equ:five-edge-rotation}, if this claim fails, then $G[W]$ would contain seven pairwise disjoint edges, which together cover $14$ vertices (see Figure~\ref{Fig:five-edges-rotation-4}). By~\eqref{equ:w-5-tuple-choices-a}, this means that $\{H_i, H_j\}$ is $5$-extendable, which is a contradiction.
    \end{proof} 

\begin{figure}[H]
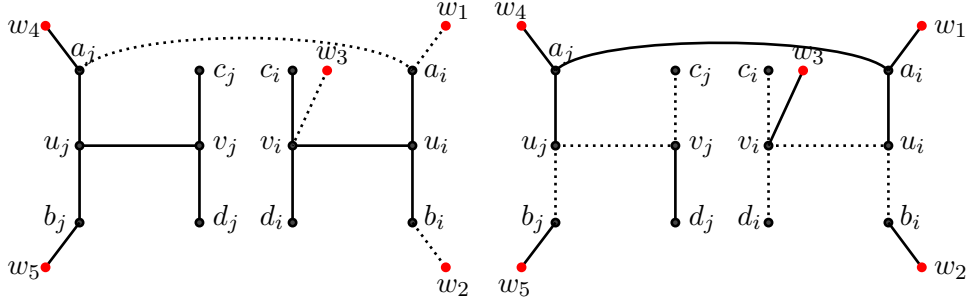

\centering
\tikzset{every picture/.style={line width=1pt}} 

\caption{Auxiliary figure for the proof of Claim~\ref{CLAIM:five-edges-e}.} 
\label{Fig:five-edges-rotation-5}
\end{figure}
    
    \begin{claim}\label{CLAIM:five-edges-e}
        No pair in $\{a_i, b_i\} \times \{a_j, b_j\}$ is an edge in $G$. 
    \end{claim}
    \begin{proof}[Proof of Claim~\ref{CLAIM:five-edges-e}]
        Suppose to the contrary that this claim fails. By symmetry, we may assume that $a_i a_j \in G$. Then $G[W]$ would contain a copy of $H$ and four pairwise disjoint edges, which together cover $14$ vertices (see Figure~\ref{Fig:five-edges-rotation-5}). By~\eqref{equ:w-5-tuple-choices-a}, this means that $\{H_i, H_j\}$ is $5$-extendable, which is a contradiction.
    \end{proof} 

\begin{figure}[H]
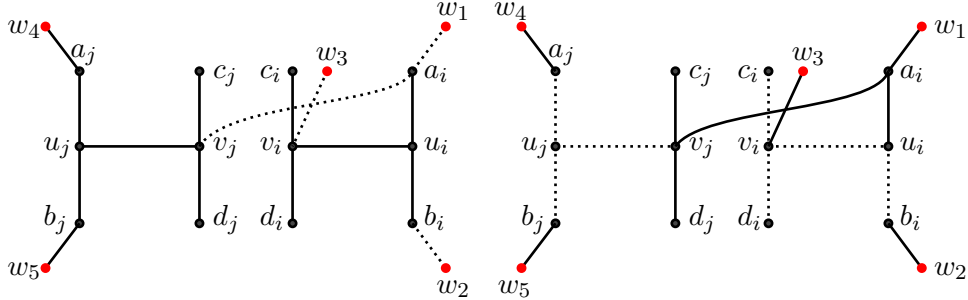

\centering
\tikzset{every picture/.style={line width=1pt}} 

\caption{Auxiliary figure for the proof of Claim~\ref{CLAIM:five-edges-f}.} 
\label{Fig:five-edges-rotation-6}
\end{figure}
    
    \begin{claim}\label{CLAIM:five-edges-f}
        We have $\{a_i v_j, b_i v_j\} \cap G = \emptyset$. 
    \end{claim}
    \begin{proof}[Proof of Claim~\ref{CLAIM:five-edges-f}]
        Suppose to the contrary that this claim fails. By symmetry, we may assume that $a_i v_j \in G$. Then $G[W]$ would contain a copy of $H$ and four pairwise disjoint edges, which together cover $14$ vertices (see Figure~\ref{Fig:five-edges-rotation-6}). By~\eqref{equ:w-5-tuple-choices-a}, this means that $\{H_i, H_j\}$ is $5$-extendable, which is a contradiction. 
    \end{proof} 

    Note from Claim~\ref{CLAIM:five-edges-d},~\ref{CLAIM:five-edges-e} and~\ref{CLAIM:five-edges-f}, there are at least $15 > 14 = 36 - e(H_i, H_j)$ pairs in $V(H_i) \times V(H_j)$ that do not belong to $G[H_i, H_j]$, which is a contradiction.

    \medskip 

    \textbf{Case 3}: $\{x, y\} = \{v_j, a_j\}$ or $\{x, y\} = \{v_j, b_j\}$.

    By symmetry, we may assume that  $\{x, y\} = \{v_j, a_j\}$. 
    
    \begin{claim}\label{CLAIM:five-edges-g}
        The following statements hold. 
        \begin{enumerate}[label=(\roman*)]
            \item\label{CLAIM:five-edges-g-1} No pair in $\{u_i, c_i, d_i\} \times \{c_j, d_j\}$ is an edge in $G$. 
            \item\label{CLAIM:five-edges-g-2} $\{u_i u_j, c_i u_j, d_i u_j\} \cap G = \emptyset$. 
            \item\label{CLAIM:five-edges-g-3} $\{a_i v_j, b_i v_j\} \cap G = \emptyset$. 
            \item\label{CLAIM:five-edges-g-4} $v_i v_j \not\in G$. 
        \end{enumerate} 
    \end{claim}
    \begin{proof}[Proof of Claim~\ref{CLAIM:five-edges-g}]
\begin{figure}[H]
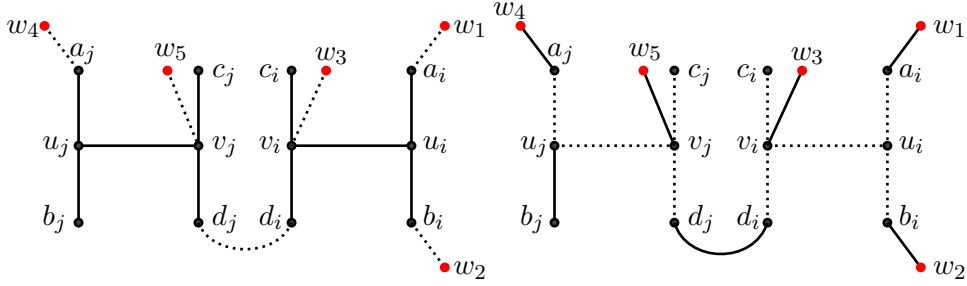

\centering
\tikzset{every picture/.style={line width=1pt}} 

\caption{Auxiliary figure for the proof of Claim~\ref{CLAIM:five-edges-g}~\ref{CLAIM:five-edges-g-1}.} 
\label{Fig:five-edges-rotation-7}
\end{figure}
        
        Let us first prove~\ref{CLAIM:five-edges-g-1}. Observe that $G[W]$ already contains five pairwise disjoint edges 
        \begin{align*}
            a_iw_1, b_iw_2, v_iw_3, a_jw_4, v_jw_5, u_j b_j. 
        \end{align*}
        Since every pair in $\{u_i, c_i, d_i\} \times \{c_j, d_j\}$ is vertex-disjoint from edges in~\eqref{Equ:five-edge-rotation}, if~\ref{CLAIM:five-edges-g-1} fails, then $G[W]$ would contain seven pairwise disjoint edges, which together cover $14$ vertices (see Figure~\ref{Fig:five-edges-rotation-7}). By~\eqref{equ:w-5-tuple-choices-a}, this means that $\{H_i, H_j\}$ is $5$-extendable, which is a contradiction. 

\begin{figure}[H]
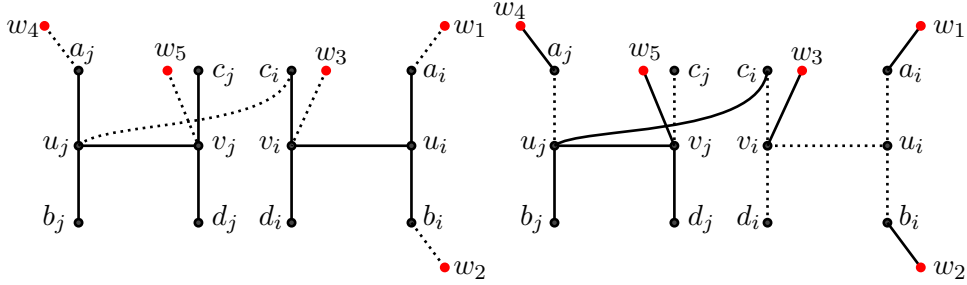

\centering
\tikzset{every picture/.style={line width=1pt}} 

\caption{Auxiliary figure for the proof of Claim~\ref{CLAIM:five-edges-g}~\ref{CLAIM:five-edges-g-2}.} 
\label{Fig:five-edges-rotation-8}
\end{figure}

        Next we prove~\ref{CLAIM:five-edges-g-2}. Suppose to the contrary that~\ref{CLAIM:five-edges-g-2} fails. By symmetry, we may assume that $c_i u_j \in G$. Then $G[W]$ would contain a copy of $H$ and four pairwise disjoint edges, which together cover $14$ vertices (see Figure~\ref{Fig:five-edges-rotation-8}). By~\eqref{equ:w-5-tuple-choices-a}, this means that $\{H_i, H_j\}$ is $5$-extendable, which is a contradiction.

\begin{figure}[H]
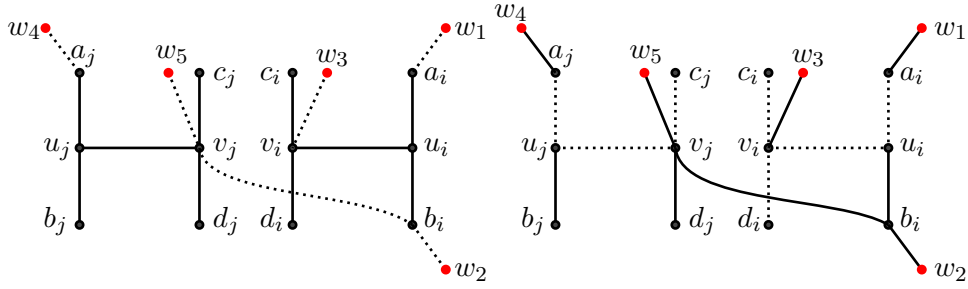

\centering
\tikzset{every picture/.style={line width=1pt}} 

\caption{Auxiliary figure for the proof of Claim~\ref{CLAIM:five-edges-g}~\ref{CLAIM:five-edges-g-3}.} 
\label{Fig:five-edges-rotation-9}
\end{figure}

        Next, we prove~\ref{CLAIM:five-edges-g-3}. Suppose to the contrary that~\ref{CLAIM:five-edges-g-3} fails. By symmetry, we may assume that $b_i v_j \in G$. Then $G[W]$ would contain a copy of $H$ and four pairwise disjoint edges, which together cover $14$ vertices (see Figure~\ref{Fig:five-edges-rotation-9}). By~\eqref{equ:w-5-tuple-choices-a}, this means that $\{H_i, H_j\}$ is $5$-extendable, which is a contradiction.

\begin{figure}[H]
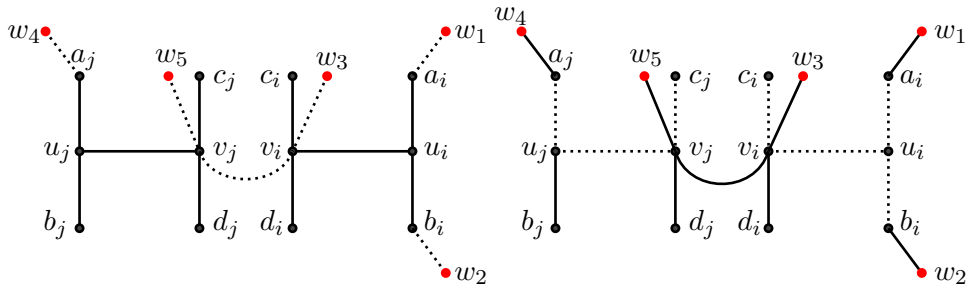

\centering
\tikzset{every picture/.style={line width=1pt}} 

\caption{Auxiliary figure for the proof of Claim~\ref{CLAIM:five-edges-g}~\ref{CLAIM:five-edges-g-4}.} 
\label{Fig:five-edges-rotation-10}
\end{figure}

        Next, we prove~\ref{CLAIM:five-edges-g-4}. Suppose to the contrary that~\ref{CLAIM:five-edges-g-4} fails. By symmetry, we may assume that $v_i v_j \in G$. Then $G[W]$ would contain a copy of $H$ and four pairwise disjoint edges, which together cover $14$ vertices (see Figure~\ref{Fig:five-edges-rotation-10}). By~\eqref{equ:w-5-tuple-choices-a}, this means that $\{H_i, H_j\}$ is $5$-extendable, which is a contradiction.
    \end{proof} 

    \begin{claim}\label{CLAIM:five-edges-h}
        We have $\{v_i a_j, v_i b_j, v_i u_j\} \cap G = \emptyset$. 
    \end{claim}
    \begin{proof}[Proof of Claim~\ref{CLAIM:five-edges-h}]
        Note that from Claim~\ref{CLAIM:five-edges-g}, there are already $12$ pairs in $V(H_i) \times V(H_j)$ that do not belong to $G[H_i, H_j]$. 
        
\begin{figure}[H]
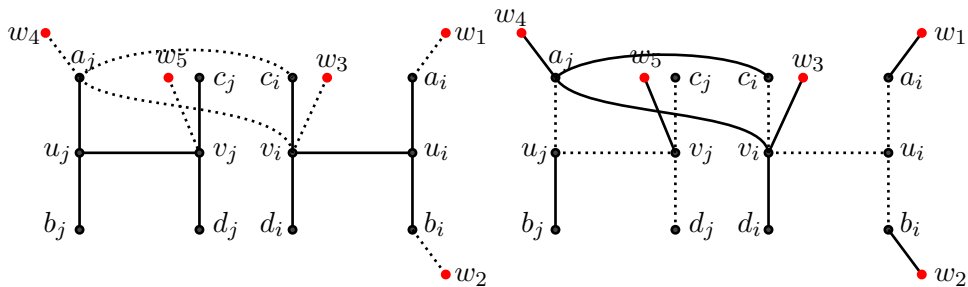

\centering
\tikzset{every picture/.style={line width=1pt}} 

\caption{Auxiliary figure for the proof of Claim~\ref{CLAIM:five-edges-h}.} 
\label{Fig:five-edges-rotation-11}
\end{figure}
        
        Let us first prove that $v_i a_j \not\in G$. Suppose to the contrary that $v_i a_j \in G$. We claim that $\{u_i a_j, c_i a_j, d_i a_j\} \cap G = \emptyset$. This will imply that there are at least $15 > 14 = 36 - e(H_i, H_j)$ pairs in $V(H_i) \times V(H_j)$ that do not belong to $G[H_i, H_j]$, which is a contradiction. 
         Suppose to the contrary that $\{u_i a_j, c_i a_j, d_i a_j\} \cap G \neq \emptyset$. By symmetry, we may assume that $c_i a_j \in G$. 
        Then $G[W]$ would contain a copy of $H$ and four pairwise disjoint edges, which together cover $14$ vertices (see Figure~\ref{Fig:five-edges-rotation-11}). By~\eqref{equ:w-5-tuple-choices-a}, this means that $\{H_i, H_j\}$ is $5$-extendable, which is a contradiction.

\begin{figure}[H]
\centering
\tikzset{every picture/.style={line width=1pt}} 

\caption{Auxiliary figure for the proof of Claim~\ref{CLAIM:five-edges-h}.} 
\label{Fig:five-edges-rotation-12}
\end{figure}

        Next we prove $v_i b_j \not\in G$. Suppose to the contrary that $v_i b_j \in G$. Then we can show that $u_i b_j, c_i b_j, d_i b_j \not\in G$, which will imply that there are already $15 > 14 = 36 - e(H_i, H_j)$ pairs in $V(H_i) \times V(H_j)$ that do not belong to $G[H_i, H_j]$, a contradiction. 

        It left to show $u_i b_j, c_i b_j, d_i b_j \not\in G$. By symmetry, suppose to the contrary that $d_i b_j \in G$. 
        Then $G[W]$ would contain a copy of $H$ and four pairwise disjoint edges, which together cover $14$ vertices (see Figure~\ref{Fig:five-edges-rotation-12}). By~\eqref{equ:w-5-tuple-choices-a}, this means that $\{H_i, H_j\}$ is $5$-extendable, which is a contradiction.

\begin{figure}[H]
\centering
\tikzset{every picture/.style={line width=1pt}} 

\caption{Auxiliary figure for the proof of Claim~\ref{CLAIM:five-edges-h}.} 
\label{Fig:five-edges-rotation-13}
\end{figure}

        Finally we prove $v_i u_j \not\in G$. Suppose to the contrary that $v_i u_j \in G$. We claim that $\{u_i u_j, c_i u_j, d_i u_j\} \cap G = \emptyset$. This will imply that there are at least $15 > 14 = 36 - e(H_i, H_j)$ pairs in $V(H_i) \times V(H_j)$ that do not belong to $G[H_i, H_j]$, which is a contradiction. 
        Suppose to the contrary that $\{u_i u_j, c_i u_j, d_i u_j\} \cap G \neq \emptyset$. By symmetry, we may assume that $d_i u_j \in G$. 
        Then $G[W]$ would contain a copy of $H$ and four pairwise disjoint edges, which together cover $14$ vertices (see Figure~\ref{Fig:five-edges-rotation-13}). By~\eqref{equ:w-5-tuple-choices-a}, this means that $\{H_i, H_j\}$ is $5$-extendable, which is a contradiction.
    \end{proof} 

    Note from Claim~\ref{CLAIM:five-edges-g} and~\ref{CLAIM:five-edges-h}, there are at least $15 > 14 = 36 - e(H_i, H_j)$ pairs in $V(H_i) \times V(H_j)$ that belong to $G[H_i, H_j]$, which is a contradiction.
    This completes the proof of Lemma~\ref{LEMMA:H2-H3-upper-bound}. 
\end{proof}

\subsection{Upper bound for $e(\mathcal{H}_3)$}\label{SUBSEC:H3-upper-bound}
\begin{lemma}\label{LEMMA:H3-upper-bound}
    Suppose that $\{H_i, H_j\} \subseteq \mathcal{H}'$ is a pair such that $|V(H_i) \cap L|  = |V(H_j) \cap L| = 3$. 
    Then $e(H_i, H_j) \le 18$. 
    In particular,
    \begin{align*}
        e(\mathcal{H}_3) 
        \le 18 \binom{y_3 n}{2} + \binom{6}{2}y_3 n
        \le 9 y_3^2 n^2 + 6 y_3 n.   
    \end{align*}
\end{lemma}
\begin{proof}[Proof of Lemma~\ref{LEMMA:H3-upper-bound}]
    Fix a pair $\{H_i, H_j\} \subseteq \mathcal{H}'$  such that $|V(H_i) \cap L|  = |V(H_j) \cap L| = 3$. 
    Suppose that $|V(H_i \cup H_j) \cap L| \ge 6$. 
    By Lemma~\ref{LEMMA:H-R-at-most-3-edges} and Assumption~\eqref{equ:assum-ci-di-not-in-L}, we have $\{v_i, a_i, b_i\} = V(H_i) \cap L$ and $\{v_j, a_j, b_j\} = V(H_j) \cap L$. 
    Fix an arbitrary $6$-tuple of distinct vertices 
    \begin{align*}
        (w_1, \ldots, w_6) \in
        N_{G}(v_i, \overline{U}) \times N_{G}(a_i, \overline{U}) \times N_{G}(b_i, \overline{U}) \times N_{G}(v_j, \overline{U}) \times N_{G}(a_j, \overline{U}) \times N_{G}(b_j, \overline{U}). 
    \end{align*}
    Let $W\coloneqq V(H_i \cup H_j) \cup \{w_1, \ldots, w_6\}$. 
    Note that the number of choices for such (unordered) $6$-sets $\{w_1, \ldots, w_6\}$ is at least 
    \begin{align}\label{equ:choices-w-tuple}
        \frac{1}{6!} \cdot 19 \delta^{1/6} n \cdot (19 \delta^{1/6} n - 1) \cdots (19 \delta^{1/6} n - 5)
        \ge 19 \delta n^6. 
    \end{align}

\begin{figure}[H]
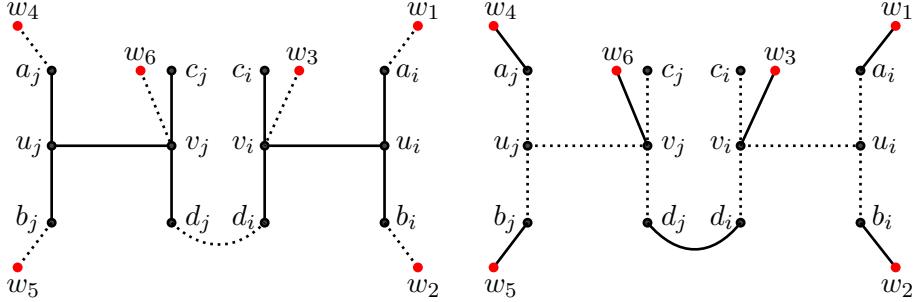

\centering
\tikzset{every picture/.style={line width=1pt}} 

\caption{Auxiliary figure for the proof of Claim~\ref{CLAIM:six-edges-a}.} 
\label{Fig:six-edge-rotation-1}
\end{figure}
    
    \begin{claim}\label{CLAIM:six-edges-a}
        No pair in $\{u_i, c_i, d_i\} \times \{u_j, c_j, d_j\}$ is an edge in $G$. 
    \end{claim}
    \begin{proof}[Proof of Claim~\ref{CLAIM:six-edges-a}]
        Observe that $G[W]$ already contains six pairwise disjoint edges: 
        \begin{align*}
            a_iw_1, b_iw_2, v_iw_3, a_jw_4, b_jw_5, v_jw_6. 
        \end{align*}
        Since every pair in $\{u_i, c_i, d_i\} \times \{u_j, c_j, d_j\}$ is vertex-disjoint from them, if this claim fails, then $G[W]$ would contain seven pairwise disjoint edges (see Figure~\ref{Fig:six-edge-rotation-1}). 
        By~\eqref{equ:choices-w-tuple}, this means that $\{H_i,H_j\}$ is $6$-extendable, which is a contradiction. 
    \end{proof}

\begin{figure}[H]
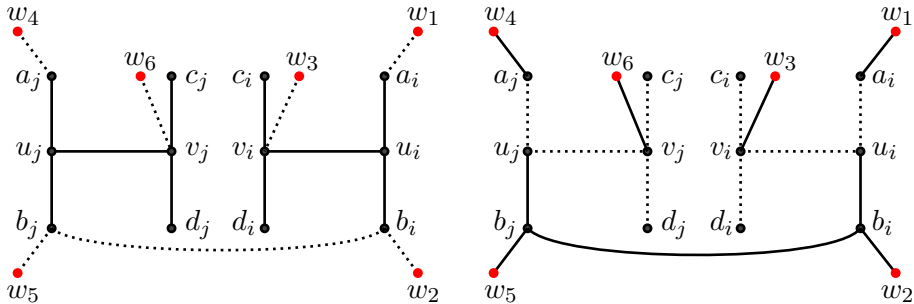

\centering
\tikzset{every picture/.style={line width=1pt}} 

\caption{Auxiliary figure for the proof of Claim~\ref{CLAIM:six-edges-b}.} 
\label{Fig:six-edge-rotation-2}
\end{figure}

    \begin{claim}\label{CLAIM:six-edges-b}
        No pair in $\{a_i, b_i\} \times \{a_j, b_j\}$ is an edge in $G$. 
    \end{claim}
    \begin{proof}[Proof of Claim~\ref{CLAIM:six-edges-b}]
        Suppose to the contrary that this claim fails. By symmetry, we may assume that $b_ib_j \in G$. Then $G[W]$ contains a copy of $H$ and four pairwise disjoint edges, which together cover $14$ vertices (see Figure~\ref{Fig:six-edge-rotation-2}). By~\eqref{equ:choices-w-tuple}, this means that $\{H_i,H_j\}$ is $6$-extendable, which is a contradiction. 
    \end{proof}

\begin{figure}[H]
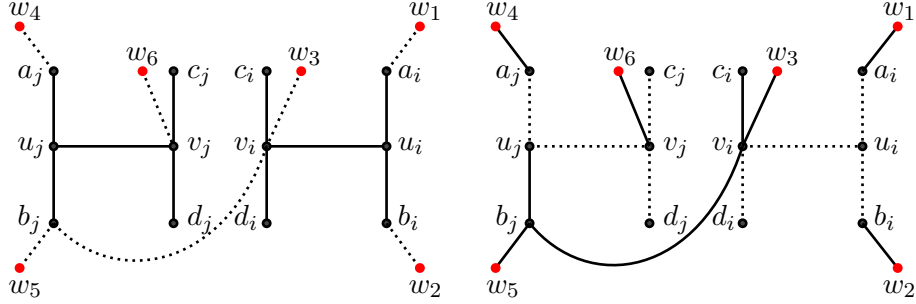

\centering
\tikzset{every picture/.style={line width=1pt}} 

\caption{Auxiliary figure for the proof of Claim~\ref{CLAIM:six-edges-c}.} 
\label{Fig:six-edge-rotation-3}
\end{figure}

    \begin{claim}\label{CLAIM:six-edges-c}
        We have $\{v_i a_j, v_i b_j\} \cap G = \emptyset$ and $\{v_j a_i, v_j b_i\} \cap G = \emptyset$.   
    \end{claim}
    \begin{proof}[Proof of Claim~\ref{CLAIM:six-edges-c}]
        Suppose to the contrary that this claim fails. By symmetry, we may assume that $v_i b_j \in G$. Then $G[W]$ contains a copy of $H$ and four pairwise disjoint edges, which together cover $14$ vertices (see Figure~\ref{Fig:six-edge-rotation-3}). By~\eqref{equ:choices-w-tuple}, this means that $\{H_i,H_j\}$ is $6$-extendable, which is a contradiction. 
    \end{proof}

\begin{figure}[H]
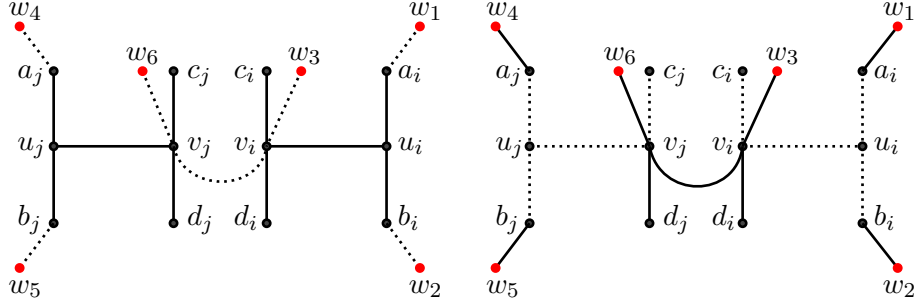

\centering
\tikzset{every picture/.style={line width=1pt}} 

\caption{Auxiliary figure for the proof of Claim~\ref{CLAIM:six-edges-d}.} 
\label{Fig:six-edge-rotation-4}
\end{figure}

    \begin{claim}\label{CLAIM:six-edges-d}
        We have $v_i v_j \not\in G$. 
    \end{claim}
    \begin{proof}[Proof of Claim~\ref{CLAIM:six-edges-d}]
        Observe that $G[W]$ already contains four pairwise disjoint edges $$a_i w_1, b_i w_2, a_j w_4, b_j w_5.$$
        If $v_i v_j \in G$, then the set $\{v_i, v_j, c_i, d_i, c_j, d_j\}$ would span a copy of $H$ in $G$ (see Figure~\ref{Fig:six-edge-rotation-4}). By~\eqref{equ:choices-w-tuple}, this means that $\{H_i,H_j\}$ is $6$-extendable, which is a contradiction.
    \end{proof}
    
    It follows from Claims~\ref{CLAIM:six-edges-a},~\ref{CLAIM:six-edges-b},~\ref{CLAIM:six-edges-c}, and~\ref{CLAIM:six-edges-d} that 
    \begin{align*}
        e(H_i, H_j)
        \le 36 - 9 - 4 - 4 - 1
        = 18, 
    \end{align*}
    proving  Lemma~\ref{LEMMA:H3-upper-bound}. 
\end{proof}

\section{Concluding remarks}\label{SEC:Remark}
A straightforward modification of the proof of Theorem~\ref{THM:Main-H-tiling} yields the following theorem, which provide infinitely many counterexamples to Conjecture~\ref{CONJ:Lang}. 

\begin{theorem}
    Let $t \ge 1$ be an integer and $\beta \in \left[0, \frac{1}{6t}\right]$ be a real number. 
    Then 
    \begin{align*}
        \mathrm{ex}(n,\beta n \cdot H[t])
        \le \left(\max\left\{3t\beta(1-3t\beta),~18t^2\beta^2\right\} + o(1)\right) n^2.  
    \end{align*}
\end{theorem}

For an $r$-graph $F$ and an integer $i \in [r]$, the \textbf{$i$-covering number} $\tau_{i}(F)$ of $F$ is defined as 
\begin{align*}
    \tau_{i}(F) 
    \coloneqq \min\left\{|S| \colon S \subseteq V(F) \ \text{and}\  |S\cap e| \ge i \ \text{for all} \  e \in F\right\}.
\end{align*}
Note that $\tau_{1}(F)$ is simply the covering number of $F$, and $\tau_{r}(F) = v(F)$. 

We call an $r$-graph $F$ \textbf{rigid} if there exist integers $1 \le s_1 \le \cdots \le s_r$ satisfying $v(F) = s_1 + \cdots + s_r$ such that $F \subseteq K_{s_1, \ldots, s_r}^{r}$ and $\tau_{i}(F) = s_1 + \cdots + s_i$ for every $i \in [r-1]$.
In particular, $K_{s_1, \ldots, s_r}^{r}$ is rigid for all integers $1 \le s_1 \le \cdots \le s_r$. 
Observe that $K_{s_1, \ldots, s_r}^{r}$ in~\eqref{equ:trivial-lower-bound} can be replaced by any rigid spanning subgraph of $K_{s_1, \ldots, s_r}^{r}$. 
Thus,  we propose the following revised version of Conjecture~\ref{CONJ:Lang}. 

\begin{conjecture}\label{CONJ:Lang-modified}
    Let $r\ge 3$ and $1 \le s_1 \le \cdots \le s_r$ be integers. Let $m \coloneqq s_1 + \cdots + s_r$. Suppose that $F$ is a rigid spanning subgraph of $K_{s_1, \ldots, s_r}^{r}$. Then for every real number $\beta \in \left(0, \frac{1}{m}\right)$,
    \begin{align*}
        \mathrm{ex}(n,\beta n \cdot F) 
        = \max\left\{|G_{n,i,\beta}(s_1, \ldots, s_r)| \colon i \in [r]\right\} + o(n^r).
    \end{align*}
\end{conjecture}
\textbf{Remark.}
Note that the case $r=2$ follows from Theorem~\ref{THM:GH12}.  

The asymptotic behavior of $\mathrm{ex}(n,\beta n \cdot F)$ remains unclear to us when $F$ is nonrigid, as there are numerous parameters of $F$ that could be leveraged to construct lower bounds for $\mathrm{ex}(n,\beta n \cdot F)$ (see e.g. the concluding remarks in~\cite{HHLLYZ23a,GLMP24} for more details). 
Thus we propose the following question concerning nonrigid $r$-partite $r$-graphs, which appears to be plausible. 
\begin{problem}\label{PROB:non-rigid}
    Let $r\ge 2$ and $1 \le s_1 \le \cdots \le s_r$ be integers. Let $m \coloneqq s_1 + \cdots + s_r$. Suppose that $F$ is a nonrigid spanning subgraph of $K_{s_1, \ldots, s_r}^{r}$.
    Is it true that there exist $\beta \in \left(0, \frac{1}{m}\right)$ and $\varepsilon > 0$ such that the following holds for all sufficiently large $n$:
    \begin{align*}
        \mathrm{ex}(n,\beta n \cdot F) 
        \le \max\left\{|G_{n,i,\beta}(s_1, \ldots, s_r)| \colon i \in [r]\right\} - \varepsilon n^{r}.
    \end{align*}
\end{problem}

As a first step toward understanding the density tiling problem in general, gaining a complete understanding of it for trees would be helpful. 
\begin{problem}\label{PROB:tiling-tree}
    Let $T$ be a tree. 
    Determine the value of $\mathrm{ex}(n,\beta n \cdot T)$ for large $n$ and all $\beta \in (0, 1/v(T))$. 
\end{problem}
\section*{Acknowledgment}
We would like to thank Donglei Yang for valuable discussions.
\bibliographystyle{alpha}
\bibliography{Htiling-new}
\end{document}